\documentclass[letterpaper,11pt,oneside,reqno]{amsart}
\usepackage{comment}
\usepackage[english]{babel}%
\usepackage{amsmath,amssymb,amsthm,amsfonts}%
\usepackage{hyperref}%
\usepackage{enumerate}%
\usepackage{array}%
\usepackage{graphicx}
\usepackage[mathscr]{euscript}
\usepackage{color,tikz}
\usetikzlibrary{calc}
\usetikzlibrary{patterns}
\usetikzlibrary{arrows}
\usetikzlibrary{decorations.pathreplacing}
\usetikzlibrary{decorations.pathmorphing}
\usepackage[DIV15]{typearea}
\usepackage[width=.9\textwidth]{caption}
\allowdisplaybreaks%
\numberwithin{equation}{section}%

\begin{document}

\newcommand{\VI}{\mathbb{V}}
\newcommand{\VJ}{\mathbb{H}}
\newcommand{\al}{\alpha}
\newcommand{\be}{\beta}
\newcommand{\de}{\delta}

\bibliographystyle{alpha}
\newcommand{\fredg}{\ensuremath{\mathsf{g}}}
\newcommand{\e}{\epsilon}
\newcommand{\GG}{\ensuremath{\mathsf{G}}}
\newcommand{\GGhat}{\ensuremath{\mathsf{\widehat{G}}}}
\newcommand{\HH}{\ensuremath{\mathsf{H}}}
\newcommand{\EE}{\ensuremath{\mathbb{E}}}
\newcommand{\PP}{\ensuremath{\mathbb{P}}}
\newcommand{\R}{\ensuremath{\mathbb{R}}}
\newcommand{\Rplus}{\ensuremath{\mathbb{R}_{+}}}
\newcommand{\C}{\ensuremath{\mathbb{C}}}
\newcommand{\Z}{\ensuremath{\mathbb{Z}}}
\newcommand{\N}{\ensuremath{\mathbb{N}}}
\newcommand{\Q}{\ensuremath{\mathbb{Q}}}
\newcommand{\MM}{\ensuremath{\mathbb{M}}}
\newcommand{\Real}{\ensuremath{\mathrm{Re}}}
\newcommand{\Imag}{\ensuremath{\mathrm{Im}}}
\newcommand{\re}{\ensuremath{\mathrm{Re}}}
\newcommand{\la}{\ensuremath{\lambda}}
\newcommand{\bernw}[4]{\ensuremath{L^{(#1)}_{#4}}}
\newcommand{\bernwtilde}[4]{\ensuremath{\tilde{L}^{(#1)}_{#4}}}
\newcommand{\bernwbeta}[4]{\ensuremath{L^{#1}_{#4}}}
\newcommand{\bernwtensor}[5]{\ensuremath{\big(L^{(#1)}_{#4}\big)^{\otimes_{#2}#5}}}
\newcommand{\Borodinw}[4]{\ensuremath{w^{(#1)}_{#4}}}
\newcommand{\barBorodinw}[4]{\ensuremath{\bar{w}^{(#1)}_{#4}}}
\newcommand{\phiup}[4]{\ensuremath{p^{(#1)}_{#4}}}
\newcommand{\qHahnHR}[3]{\ensuremath{\mathcal{H}^{#3}}}
\newcommand{\qHahnH}[3]{\ensuremath{\widetilde{\mathcal{H}}^{#3}}}
\newcommand{\qHahnA}[3]{\ensuremath{\widetilde{\mathcal{H}}^{#3}}}
\newcommand{\highspinTASEP}[2]{\ensuremath{\mathcal{T}^{#1,#2}}}
\newcommand{\highspinBoson}[2]{\ensuremath{\mathcal{B}^{#1,#2}}}
\newcommand{\revhighspinBoson}[2]{\ensuremath{\mathcal{\tilde{B}}^{#1,#2}}}
\newcommand{\Xinf}{\ensuremath{\mathbb{X}}}
\newcommand{\Xspace}[1]{\ensuremath{\mathbb{X}}^{#1}}
\newcommand{\Wc}{\mathcal{W}} 
\newcommand{\tWc}{\widetilde{\mathcal{W}}} 
\newcommand{\Cc}{\mathcal{C}} 

\newcommand{\Nup}{\ensuremath{\mathsf{N}^{\uparrow}}}
\newcommand{\Ndown}{\ensuremath{\mathsf{N}^{\downarrow}}}
\newcommand{\Ginf}{\mathbb{G}}
\newcommand{\Yinf}{\mathbb{Y}}
\newcommand{\Gspace}[1]{\mathbb{G}^{#1}}
\newcommand{\Yspace}[1]{\mathbb{Y}^{#1}}
\newcommand{\Gmspace}[1]{\mathbb{WG}^{#1}}
\newcommand{\Ynspace}[1]{\mathbb{WY}^{#1}}
\newcommand{\Psibwd}{\Psi^\mathrm{bwd}}
\newcommand{\Psicfwd}{\Psi^\mathrm{cfwd}}
\newcommand{\Psifwd}{\Psi^\mathrm{fwd}}
\newcommand{\Psir}{\Psi^{r}}
\newcommand{\Psird}{\Psi^{r;\dilp}}
\newcommand{\Psil}{\Psi^{\ell}}
\newcommand{\Psild}{\Psi^{\ell;\dilp}}

\newcommand{\Lmat}{\mathrm{L}}

\newcommand{\Phir}{\Phi^{r}}
\newcommand{\Phird}{\Phi^{r;\dilp}}
\newcommand{\Phil}{\Phi^{\ell}}
\newcommand{\Phild}{\Phi^{\ell;\dilp}}

\newcommand{\st}{\mathfrak{m}_{q,\nu}} 

\newcommand{\Pld}{\mathcal{F}^{q,\nu}} 
\newcommand{\Pli}{\mathcal{J}^{q,\nu}} 
\newcommand{\Plspatial}{\mathcal{K}^{q,\nu}} 
\newcommand{\Plspectral}{\mathcal{M}^{q,\nu}} 

\newcommand{\Plde}{\mathcal{F}} 
\newcommand{\Plie}{\mathcal{J}} 
\newcommand{\Plspatiale}{\mathcal{K}} 
\newcommand{\Plspectrale}{\mathcal{M}} 
\newcommand{\Plm}{\mathsf{m}} 
\newcommand{\ga}{\boldsymbol\gamma}
\newcommand{\Vand}{\mathbf{V}}

\newcommand{\alqr}{\mathsf{a}}
\newcommand{\beqr}{\mathsf{b}}
\newcommand{\gammaqr}{\mathsf{c}}
\newcommand{\deqr}{\mathsf{d}}

\def \Ai {{\rm Ai}}
\def \sgn {{\rm sgn}}
\newcommand{\var}{{\rm var}}
\newcommand{\I}{{\rm i}}
\renewcommand{\i}{\mathbf i}
\newtheorem{theorem}{Theorem}[section]
\newtheorem{partialtheorem}{Partial Theorem}[section]
\newtheorem{conj}[theorem]{Conjecture}
\newtheorem{lemma}[theorem]{Lemma}
\newtheorem{proposition}[theorem]{Proposition}
\newtheorem{corollary}[theorem]{Corollary}
\newtheorem{claim}[theorem]{Claim}
\newtheorem{experiment}[theorem]{Experimental Result}

\newcommand{\Zsd}{\ensuremath{\mathbf{Z}}}
\newcommand{\Fsd}{\ensuremath{\mathbf{F}}}

\def\todo#1{\marginpar{\raggedright\footnotesize #1}}
\def\change#1{{\color{green}\todo{change}#1}}
\def\note#1{\textup{\textsf{\color{blue}(#1)}}}

\theoremstyle{definition}
\newtheorem{rem}[theorem]{Remark}

\theoremstyle{definition}
\newtheorem{com}[theorem]{Comment}

\theoremstyle{definition}
\newtheorem{definition}[theorem]{Definition}

\theoremstyle{definition}
\newtheorem{definitions}[theorem]{Definitions}
\pagestyle{plain}

\title[Stochastic higher spin vertex models on the line]{Stochastic higher spin vertex models on the line}

\author[I. Corwin]{Ivan Corwin}
\address{I. Corwin, Columbia University,
Department of Mathematics,
2990 Broadway,
New York, NY 10027, USA,
and Clay Mathematics Institute, 10 Memorial Blvd. Suite 902, Providence, RI 02903, USA,
and
Institut Henri Poincar\'e,
11 Rue Pierre et Marie Curie, 75005 Paris, France}
\email{ivan.corwin@gmail.com}

\author[L. Petrov]{Leonid Petrov}
\address{L. Petrov, University of Virginia, Department of Mathematics,
141 Cabell Drive, Kerchof Hall,
P.O. Box 400137,
Charlottesville, VA 22904-4137, USA,
and Institute for Information Transmission Problems, Bolshoy Karetny per. 19, Moscow, 127994, Russia}
\email{lenia.petrov@gmail.com}

\date{\today}
\maketitle

\begin{abstract}
We introduce a four-parameter family of interacting particle systems on the line which can be diagonalized explicitly via a complete set of Bethe ansatz eigenfunctions, and which enjoy certain Markov dualities. Using this, for the systems started in step initial data we write down nested contour integral formulas for moments and Fredholm determinant formulas for Laplace-type transforms. Taking various choices or limits of parameters, this family degenerates to many of the known exactly solvable models in the Kardar-Parisi-Zhang universality class, as well as leads to many new examples of such models. In particular, ASEP, the stochastic six-vertex model, $q$-TASEP and various directed polymer models all arise in this manner. Our systems are constructed from stochastic versions of the $R$-matrix related to the six-vertex model. One of the key tools used here is the fusion of $R$-matrices and we provide a probabilistic proof of this procedure.
\end{abstract}

\setcounter{tocdepth}{1}
\tableofcontents
\setcounter{tocdepth}{2}

\section{Introduction}

Integrable probability is an active area of research at the interface of probability/mathematical physics/statistical mechanics on the one hand, and representation theory/integrable systems on the other. Integrable probabilistic systems are broadly characterized by two properties:
\begin{enumerate}
\item It is possible to write down concise and exact formulas for expectations of a variety of interesting observables of the system.
\item Asymptotics of the system, observable and associated formulas provide access to exact descriptions of new phenomena as well as large universality classes.
\end{enumerate}

In light of these properties, there are two main goals in this area:
\begin{enumerate}
\item Build bridges between algebraic structures and probabilistic systems and in so doing, discover new integrable probabilistic systems and new tools by which to analyze them.
\item Study scaling limits of these integrable probabilistic systems and in so doing, expand and refine the scope of their associated universality classes and discover new asymptotic phenomena displayed by these systems.
\end{enumerate}

In this paper we work to advance the first goal. We develop a four-parameter family of stochastic interacting particle systems which are built off of higher spin representations of the six-vertex model $R$-matrix. These systems benefit from properties inherited from the $R$-matrix. In particular, they are diagonalizable explicitly in terms of a complete set of Bethe ansatz eigenfunctions, and they also enjoy certain Markov dualities. We use these two facts to compute moment and then Laplace-type transform formulas for these processes. Asymptotics of these systems and the associated exact formulas are in line with the second goal defined above. We do not pursue this here, but note that such asymptotics have previously been performed on various degenerations of this family of systems \cite{BorodinCorwin2011Macdonald,BorodinCorwinFerrari2012,BorodinCorwinRemenik,BorodinCorwinFerrariVeto2014,CorwinSeppalainenShen2014,OConnellOrthmann2014,Barraquand_qTASEP_2014,Veto2014qhahn,BCG6V}. In all of those cases, the resulting phenomena were that of the Kardar-Parisi-Zhang (KPZ) universality class. It would be quite interesting to see whether other phenomena can be accessed beyond that of the KPZ class.

Vertex models and more generally quantum integrable systems have long been objects of intense research within mathematics and physics (see, for example, the reviews \cite{Fadeev1996,Reshetikhin2008}). Generally, $R$-matrices are not stochastic, and nor are their associated transfer matrices. Here we work with a variant of the $R$-matrix which is stochastic and arises from conjugating the associated transfer matrix. We call this variant our $\Lmat$-matrix (see Remark \ref{rem:BorR} for the relation to the usual $R$-matrix). This enables us to define our vertex models on the entire line such as in \cite{BCG6V}. The associated transfer matrix is Bethe ansatz diagaonalizable, as follows either from taking a limit of the finite lattice algebraic Bethe ansatz or from the recent work \cite{Borodin2014vertex}. Moreover, as opposed to on a finite lattice, there are relatively simple direct and inverse transforms with respect to these Bethe ansatz eigenfunctions \cite{BCPS2014} (see also Appendix \ref{sec:Bethe}). Our proofs of duality are largely based on earlier methods from \cite{BorodinCorwin2013discrete,Corwin2014qmunu}. In both cases (diagonalization and duality) we first prove our results for the horizontal spin $1/2$ ($J=1$) and arbitrary vertical spin $\Lmat$-matrix. Then we construct higher horizontal spin $\Lmat$-matrices ($J\in \Z_{\geq 1}$) via Markov functions theory, and the diagonalization and duality easily extends. This construction provides a probabilistic proof of the fusion procedure \cite{KirillovReshetikhin1987Fusion} on the line. Though initially the form of the $\Lmat$-matrix we find is not very explicit, we are able to demonstrate a recursion relation it satisfies, and then explicitly solve that in terms of $q$-Racah polynomials \cite{Mangazeev2014}.

We also have a more concrete motivation behind this work, which we now explain. The totally asymmetric simple exclusion process (TASEP) is arguably the paradigm for integrable stochastic interacting particle systems. The asymmetric simple exclusion process\footnote{Particles jump left with rate (i.e., according to independent exponentially distributed waiting times of rate) $p$ and right at rate $q$, assuming the destination is not already occupied.} (ASEP) and $q$-deformed totally asymmetric simple exclusion process\footnote{Particles jump one to the right with rate given by $1-q^{{\rm gap}}$, where ${\rm gap}$ represents the number of holes before the next particle.} ($q$-TASEP) are both one-parameter generalizations of TASEP\footnote{TASEP arises from ASEP by setting either $p$ or $q$ to zero, and from $q$-TASEP by setting $q$ to zero.}. In \cite{BorodinCorwinSasamoto2012} it was recognized that both systems enjoy Markov dualities and have moment formulas which can be written in terms of nested contour integrals. This prompted the question of whether both systems and their associated results can be united as special cases of a more general integrable probabilistic system. In \cite{BCPS2014} this question was addressed at a spectral level -- it was shown that the eigenfunctions which diagonalize both systems are special cases of a more general set of eigenfunctions (recalled in Appendix~\ref{sec:Bethe}). In this present work we provide a complete answer to this question in the affirmative. The four-parameter family of systems we introduce here have degenerations\footnote{Saying that one system degenerates to another will mean that the second is accessed from the first either through a special choice of parameters, or through a limit transition.} to both ASEP and $q$-TASEP, at the level of their transition operators. In other words, not only do the eigenfunctions degenerate, but so do the eigenvalues.

In recent work, \cite{CarinciGiardinaRedigSasamoto2014} discovered a family of higher spin particle systems which interpolate from ASEP (spin $1/2$) to $q$-TASEP (spin infinity) and proved dualities for these systems using $U_q(\mathfrak{sl}_2)$ symmetry. This two-parameter family of systems\footnote{Spin and $q$ being the two parameters.} does not appear to be diagonalizable via Bethe ansatz except for the cases of ASEP and $q$-TASEP. Hence it is unclear whether many of the nice properties of ASEP and $q$-TASEP hold for the systems in between them. Besides the cases of ASEP and $q$-TASEP, the systems and dualities proved in \cite{CarinciGiardinaRedigSasamoto2014} appear to be different than any degenerations of our four-parameter family of systems.

\begin{figure}
\begin{center}
\includegraphics[scale=.85]{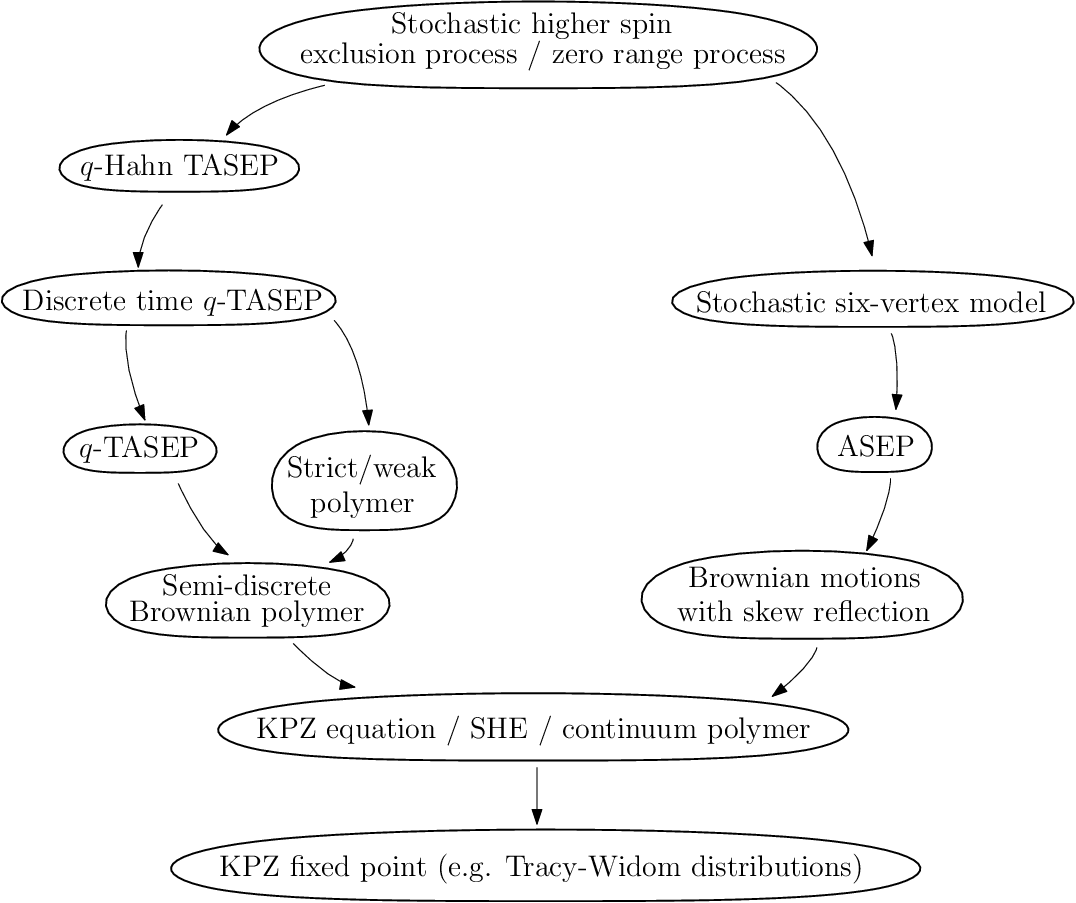}
\end{center}
\caption{Hierarchy of various degenerations of the four-parameter family of higher spin particle systems. Each arrow represents a particular degeneration of models (with the exception of the arrow to the KPZ fixed point which is at the level of one-point limiting distributions). The arrows should be transitive, though not every additional arrow has been proved.}
\label{fig:Diagramofprocesses}
\end{figure}

Besides ASEP and $q$-TASEP, our family of particle systems has a number of other interesting degenerations. Some of this hierarchy of degenerations is illustrated in Figure \ref{fig:Diagramofprocesses} wherein we show the relations to various known systems in the literature. There are, of course, many other degenerations to study. The exact nature of the degenerations are indicated below or in the cited literature. The left-hand side of the figure involves $q$-Hahn TASEP \cite{Povolotsky2013,Corwin2014qmunu} (see also Section \ref{sec:qHahn}), discrete time $q$-TASEP \cite{BorodinCorwin2013discrete}, $q$-TASEP \cite{BorodinCorwin2011Macdonald,BorodinCorwinSasamoto2012}, the strict/weak polymer \cite{CorwinSeppalainenShen2014,OConnellOrthmann2014}, and the semi-discrete Brownian polymer \cite{OConnellYor2001,Oconnell2009_Toda}. The right-hand side of the figure involves the stochastic six-vertex model \cite{BCG6V} (see Section \ref{sec:ssvm}), ASEP \cite{TW_ASEP1,TW_ASEP2,BorodinCorwinSasamoto2012}, and Brownian motions with skew reflection \cite{SasamotoSpohn_2014}.
Both sides have limits to the KPZ equation / stochastic heat equation / continuum polymer \cite{bertini1997,AmirCorwinQuastel2011} and yet further to the KPZ fixed point (e.g. with cube-root scaling and limiting GUE Tracy-Widom one-point fluctuations) \cite{TW_ASEP2,AmirCorwinQuastel2011,SasamotoSpohn2010,BorodinCorwin2011Macdonald,BorodinCorwinFerrari2012,BorodinCorwinRemenik,BorodinCorwinFerrariVeto2014,CorwinSeppalainenShen2014,OConnellOrthmann2014,Barraquand_qTASEP_2014,Veto2014qhahn,BCG6V}.
Besides these examples, there are also `determinantal' particle systems (such as TASEP) which fall into the hierarchy of degenerations and are not depicted.

Let us note that the log-gamma polymer \cite{seppalainen2012,COSZ2011,LEDoussalandStudent}, $q$-pushASEP \cite{BorodinPetrov2013NN,CorwinPetrov2013}, and two-sided $q$-Hahn ASEP \cite{BarrCorwinqHahn} are not included in Figure \ref{fig:Diagramofprocesses}. These systems are all diagonalized in the same basis (or degenerations thereof) as our four-parameter family of systems, and they have explicit and elementary eigenvalues. While we expect that these systems can also be incorporated in some form into our hierarchy, we do not pursue this direction here and leave it for future work. 

\subsection{Outline}
Section \ref{sec:jones} is devoted to the study of the $J=1$ (or horizontal spin $1/2$) vertex models. In particular, Section \ref{sec:def} provides definitions of the $J=1$ $\Lmat$-matrix, describes certain conditions on parameters under which the matrix is stochastic, and constructs three-parameter discrete time zero range and exclusion processes from the $\Lmat$-matrix. Section \ref{sec:diagonal} provides the Bethe ansatz diagonalization of these processes. Section \ref{sec:implicitduality} contains the statements and proofs of dualities enjoyed by these systems.

Section \ref{sec:fus} implements the fusion procedure through which we go from $J=1$ to arbitrary $J\in \Z_{\geq 1}$ -- thus yielding the fourth parameter\footnote{The four parameters we work with are $q,\alpha,I,J$ or in a different parametrization, $q,\alpha,\nu,\beta$.}. In particular, Section \ref{sec:fusionMFT} explains the fusion procedure by which the horizontal spin is taken from $1/2$ to $J/2$ for arbitrary $J\in \Z_{\geq 1}$. The higher spin $\Lmat$-matrix is constructed through use of the theory of Markov functions which draws on certain special properties we check for the $J=1$ case. The diagonalizability and dualities for the higher spin systems follow immediately from the $J=1$ cases. The form of the $\Lmat$-matrix (and hence higher spin zero range and exclusion processes) is initially rather inexplicit. Section \ref{sec:RR} deduces a recursion relation in $J$ satisfied by the $\Lmat$-matrix. Section \ref{sec:qRacah} provides an explicit solution to that recursion relation in terms of $q$-Racah polynomials (or terminating basic hypergeometric functions).

Section \ref{sec:mom} utilizes the duality between the $J$ higher spin zero range and exclusion processes to compute nested contour integral moment formulas for the exclusion process with step initial data. These lead to a Fredholm determinant formula for an $e_{q}$-Laplace transform which characterizes the exclusion process's one-point marginal distribution. We do not pursue asymptotics, though this type of Fredholm determinant has been used before for such purposes (see references earlier in the introduction).

Proposition \ref{rem:stochweights} identifies four different cases of parameters under which the $J=1$ $\Lmat$-matrix is stochastic. Case (1) of that remark is assumed in the earlier sections of this paper. However, in Section~\ref{sec:spec} we explain how to extend our results to the other three cases. Section \ref{sec:case2} addresses case~(2). Section~\ref{sec:reflectioninversion} identifies  reflection and inversion symmetries of the higher spin $\Lmat$-matrix. Section~\ref{sec:case3} addresses case~(3) and Section \ref{sec:case4} addresses case (4). Section \ref{sec:ssvm} demonstrates how the stochastic six-vertex model \cite{BCG6V} arises from our systems. Section \ref{sec:qHahn} describes another degeneration to the $q$-Hahn processes \cite{Povolotsky2013,Corwin2014qmunu}. Section \ref{sec:inhomo} briefly notes which results extend to spatially or temporally inhomogeneous parameters.

Appendix \ref{sec:Bethe} recalls key facts (namely, the Plancherel theory) about the Bethe ansatz eigenfunctions which diagonalize the systems considered herein.
 Appendix \ref{sec:J123} contains explicit formulas for $\Lmat$-matrix elements for $J=1,2,3$. Appendix \ref{sec:YBE} states the Yang-Baxter type equation satisfied by the $\Lmat$-matrix.

\subsection{Acknowledgements}
We thank Alexei Borodin for providing us with an early draft of \cite{Borodin2014vertex}, and for early discussions related to this work as well as comments regarding an early version of the fusion proof described in Section \ref{sec:fusionMFT}. IC also appreciates discussions with Nicolai Reshetikhin regarding the general challenge of finding stochastic quantum integrable systems. IC was partially supported by the NSF through DMS-1208998, the Clay Mathematics Institute through a Clay Research Fellowship, the Institute Henri Poincare through the Poincare Chair, and the Packard Foundation through a Packard Fellowship for Science and Engineering.
LP was partially supported by the University of Virginia
through the EDF Fellowship.

\section{$J=1$ higher spin stochastic six-vertex model}\label{sec:jones}


\subsection{Definitions and construction of processes}\label{sec:def}

We will consider $q,\nu$ as fixed throughout and thus only include other variables explicitly in our notation. 
We will use the notation $\Z_{\geq i}= \{n\in \Z: n\geq i\}$, $\Z_{\leq i}=\{n\in \Z:n\leq i\}$, and $\mathbf{1}_{E}$ is the indicator function of an event $E$. The symbol $\EE$ will denote expectation with respect to the process or random variable that follows.

\begin{figure}
\begin{center}
\includegraphics[scale=.8]{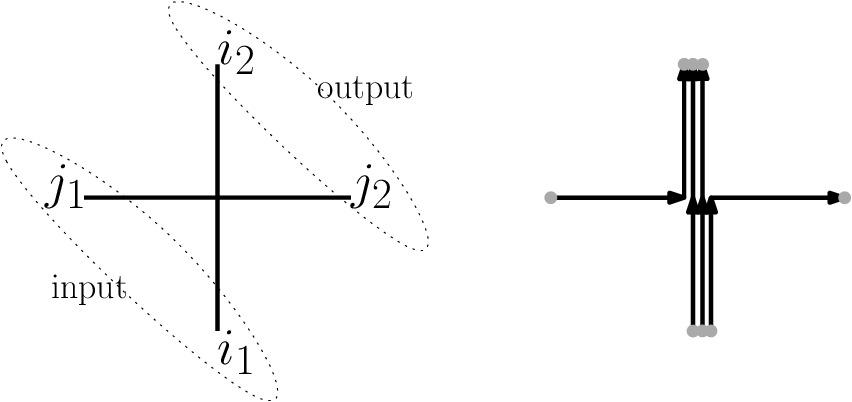}
\end{center}
\caption{Left: The weight associated to a vertex with (clockwise, starting on the bottom) labels $i_1, j_1, i_2, j_2$ in $\Z_{\geq 0}$ is given by $\bernw{1}{q}{\nu}{\alpha}(i_1,j_1;i_2,j_2)$. The $\Lmat$-matrix takes as input $(i_1,j_1)$ and produces output $(i_2,j_2)$. Right: The representation of a vertex associated with $(i_1,j_1,i_2,j_2)=(3,1,3,1)$ in terms of arrows and particles.}
\label{fig:crossing}
\end{figure}

We proceed now to define the $\Lmat$-matrix which will play a central role in all the follows.
\begin{definition}[The $J=1$ $\Lmat$-matrix]\label{defweightsJ1}
For three generic complex parameters $q,\nu,\alpha$ and any four-tuple $(i_1,j_1,i_2,j_2)\in \big(\Z_{\geq 0}\big)^4$ define a corresponding vertex weight as follows: For any $m\geq 0$,
\begin{align}\label{eq:bernweights}
\bernw{1}{q}{\nu}{\alpha}(m,0;m,0) & = \frac{1+\alpha q^m}{1+\alpha}, & \bernw{1}{q}{\nu}{\alpha}(m,0;m-1,1) & = \frac{\alpha(1- q^m)}{1+\alpha},\\
\nonumber \bernw{1}{q}{\nu}{\alpha}(m,1;m+1,0) & = \frac{1-\nu q^m}{1+\alpha}, &\bernw{1}{q}{\nu}{\alpha}(m,1;m,1) & = \frac{\alpha+\nu q^m}{1+\alpha},
\end{align}
and $\bernw{1}{q}{\nu}{\alpha}(i_1,j_1;i_2,j_2)=0$ for all other values of $(i_1,j_1,i_2,j_2)\in \big(\Z_{\geq 0}\big)^4$. Notice that all non-zero weights correspond to four-tuples such that $i_1+j_1=i_2+j_2$, a property we consider as being particle conservation.  Weights are associated graphically with crosses labeled by  $(i_1,j_1,i_2,j_2)$ in the manner of Figure \ref{fig:crossing} (see also Appendix \ref{sec:J123}). It will be convenient, at different points, to think about the $i_1,j_1,i_2,j_2$ as recording the number of arrows (either up or right pointing) along the edges incident to a given vertex, or as recording the number of particles in a given location (the $i$'s) and the number of particles to cross a given edge (the $j$'s).

We will treat these weights as matrix elements. Define a vector space $\VI^{I}$ with basis elements $\{0,1,\ldots,I\}$ if $I\in \Z_{\geq 1}$, and $\{0,1,\ldots\}$ otherwise. Likewise define a vector space $\VJ^{J}$ with basis elements $\{0,1,\ldots,J\}$ if $J\in \Z_{\geq 1}$, and $\{0,1,\ldots\}$ otherwise. As a matter of convention, when describing a linear operator acting between these spaces (or their tensor products) we will only describe matrix elements in the above basis. Then for $I\in \C$ such that $\nu=q^{-I}$ and $J=1$, $\bernw{J}{q}{\nu}{\alpha}:\VI^{I}\otimes \VJ^{J}\to \VI^{I}\otimes \VJ^{J}$ is defined by its matrix elements $\bernw{J}{q}{\nu}{\alpha}(i_1,j_1;i_1,j_2)$. We will associate $\Lmat$-matrices to vertices $(x,y)\in \Z^2$ and denote vector spaces $[\VI^I]_{x}$ and $[\VJ^J]_{y}$ as associated with column $x$ and row $y$. We write $[\bernw{J}{q}{\nu}{\alpha}]_{x,y}$ to mean the matrix which acts as $\bernw{J}{q}{\nu}{\alpha}$ on $[\VI^I]_{x}\otimes [\VJ^J]_{y}$ and the identity on all other $[\VI^I]_{x'}$ and $[\VJ^J]_{y'}$. This convention is illustrated in Figure \ref{fig:Raxes}.
\end{definition}

\begin{figure}
\begin{center}
\includegraphics[scale=.65]{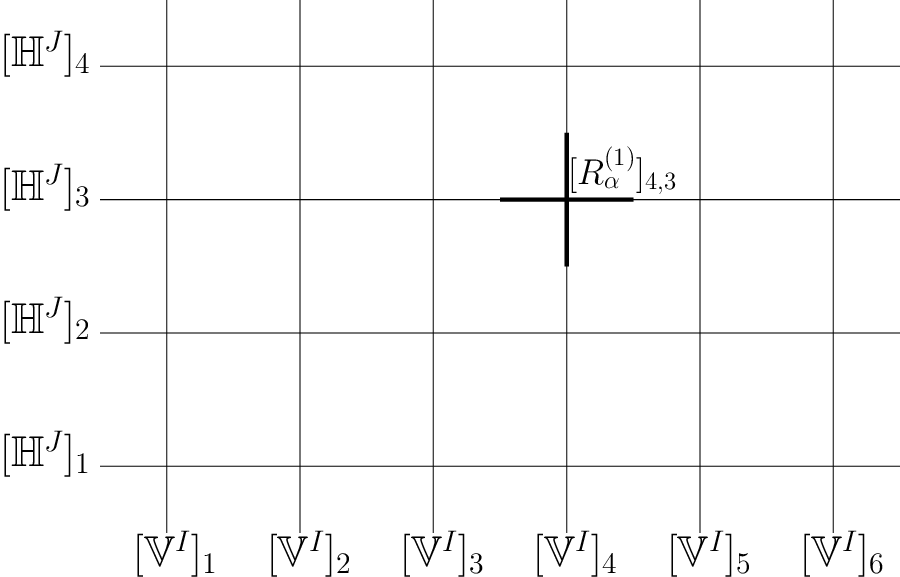}
\end{center}
\caption{The association of $\Lmat$-matrices to each vertex in $\Z^2$ as in Definition \ref{defweightsJ1}.}
\label{fig:Raxes}
\end{figure}

\begin{rem}\label{rem:BorR}
These weights are closely related to the matrix entries of the horizontal spin\footnote{In general, for $I,J\in \Z_{\geq 1}$ one says that our $\Lmat$-matrix is vertical spin $I/2$ and horizontal spin $J/2$.} $1/2$ (i.e. $J=1$) six-vertex model $R$-matrix. For instance, changing variables to $s,u$ via $\alpha = -su$ and $\nu=s^2$, we can match the above defined weights with those of \cite[Definition 2.1]{Borodin2014vertex}: $$\bernw{1}{q}{\nu}{\alpha}(i_1,j_1;i_2,j_2) = w_u(i_2,j_2;i_1,j_1)(-s)^{j_1}(-su)^{j_2-j_1}.$$
The factor $(-su)^{j_2-j_1}$ is a conjugation which does not affect the overall transfer matrix (cf. Remark \ref{rem:star}) whereas the factor $(-s)^{j_1}$ corresponds to conjugation of the transfer matrix (cf. Remark \ref{rem:trans}). \cite[Proposition 2.4]{Borodin2014vertex} further matches $w_u(i_1,j_1;i_2,j_2)$ to the matrix entries of a particular normalization of the $R$-matrix considered in \cite{Mangazeev2014} when $J=1$.
\end{rem}

We now identify various ranges of parameters under which our $\Lmat$-matrix is stochastic.  Notice that our weights have been normalized so that for $i_1,j_1$ fixed,
$$\sum_{i_2,j_2} \bernw{1}{q}{\nu}{\alpha}(i_1,j_1;i_2,j_2)=1.$$
If $q,\nu,\alpha$ are chosen so that all weights are non-negative, then $\bernw{1}{q}{\nu}{\alpha}$ is a stochastic matrix and provides the transition probabilities for going from the pair $(i_1,j_1)$ to $(i_2,j_2)$ (see Figure \ref{fig:crossing} for a graphical representation of this transition from inputs to outputs).

\begin{proposition}\label{rem:stochweights}
The following choices of parameters ensure non-negativity (and hence its stochasticity) of $\bernw{1}{q}{\nu}{\alpha}$:
\begin{enumerate}
\item $q,\nu\in [0,1)$, and $\alpha\geq0$,
\item $q\in (-1,0]$, $\alpha\in (0,1/|q|)$, and $\nu\in \big(-1/|q|, \min(1,\alpha/|q|)\big)$,
\item $q\in [0,1)$, $\nu=q^{-I}$ for $I\in \Z_{\geq 1}$, and $\alpha<-q^{-I}$,
\item $q\in (1,+\infty)$, $\nu=q^{-I}$ for $I\in \Z_{\geq 1}$, and $-q^{-I}<\alpha<0$.
\end{enumerate}
\end{proposition}
\begin{proof}
Each case follows from straightforward inspection.
\end{proof}
\begin{rem}
In the third and fourth cases, the weights are not always non-negative, however restricted to $i\in \{0,\ldots, I\}$ (i.e. $\VI^{I}$) they are. The only way to transition out of this range for $i$ is to utilize the weight $\bernw{1}{q}{\nu}{\alpha}(I,1;I+1,0)$, but due to our choice of $\nu$, this is zero. The first choice relates to models generalizing the $q$-Boson stochastic particle system \cite{SasamotoWadati1998,BorodinCorwinSasamoto2012} whereas the third and the fourth examples relate to models generalizing ASEP and the stochastic six-vertex model (coming from letting $I=1$) \cite{BCG6V}.
\end{rem}

We will assume that $q, \nu, \alpha$ satisfy case (1) above, namely that
\begin{equation}\label{eq:star7}
q,\nu\in [0,1), \quad \textrm{and}\qquad \alpha\geq0,
\end{equation}
and prove all of our results under those conditions on parameters. In Section \ref{sec:spec} we describe how our various results extend to the other choices of parameters from Proposition \ref{rem:stochweights}.

\begin{figure}
\begin{center}
\includegraphics[scale=.5]{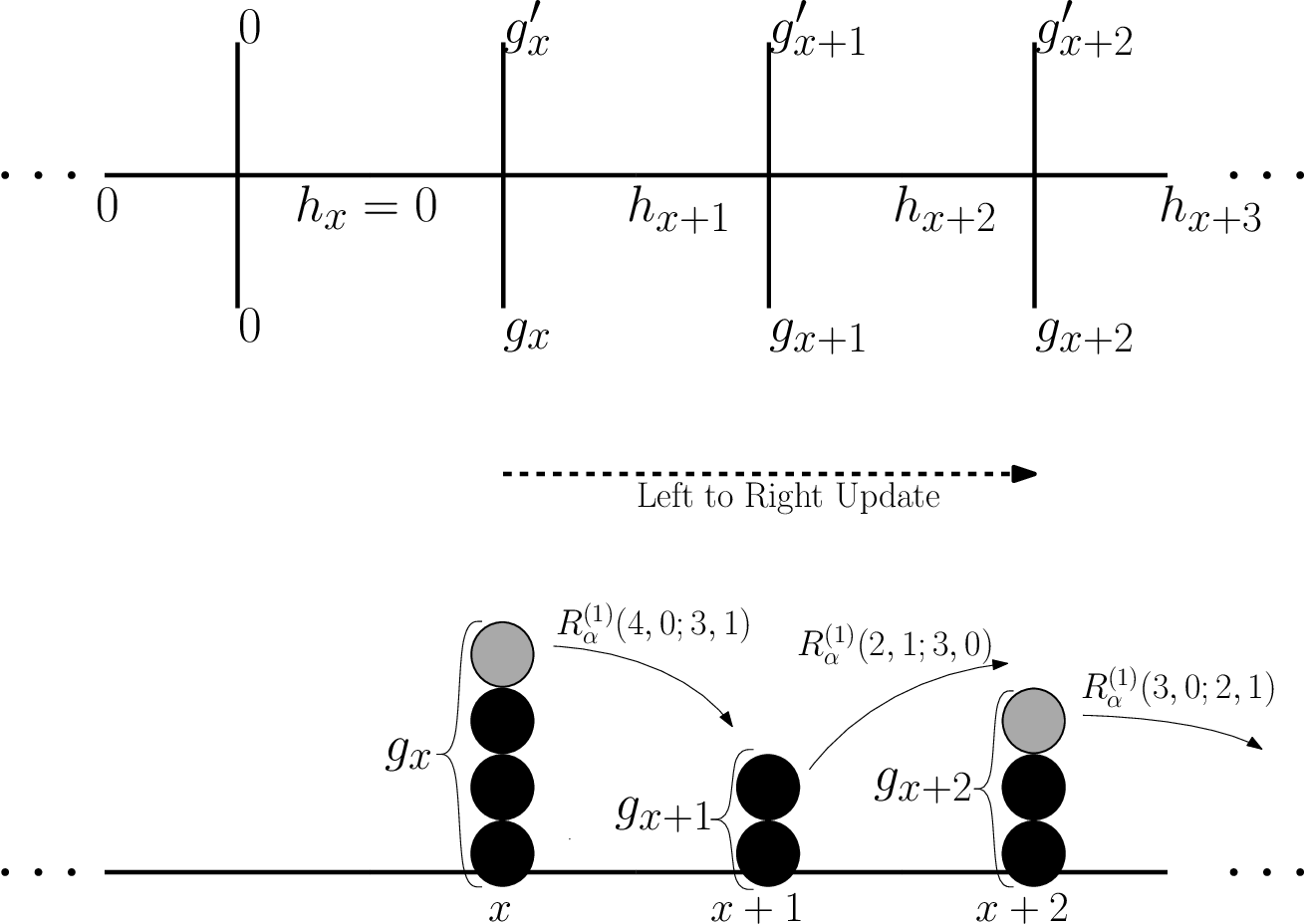}
\end{center}
\caption{A schematic showing the discrete time zero range process sequential update. The grey particles on the bottom of the figure are the ones which will move to the right by one.}
\label{fig:crossesjoined}
\end{figure}

\medskip
We turn now to define certain sequential update, discrete time zero range processes and exclusion processes based off of our stochastic $\Lmat$-matrix\footnote{Strictly speaking, these might not be the most accurate names for these processes, but for lack of a better name they will suffice.}. This is a general construction whose only input is the stochastic $\Lmat$-matrix. We will make use of this construction later when $\bernw{1}{q}{\nu}{\alpha}$ is replaced by $\bernw{J}{q}{\nu}{\alpha}$.

\begin{definition}[State spaces]\label{deftransprobJ1}
Define the space of left-finite particle configurations on the line as
$$
\Ginf= \Big\{\vec{g}= (\cdots, g_{-1},g_0,g_1,\ldots): \textrm{all } g_i\in \Z_{\geq 0}\cup\{+\infty\}, \textrm{ and there exists } x\in \Z \textrm{ such that } g_i=0 \textrm{ for all }i<x\Big\},
$$
and likewise the space of right-finite particle configurations on the line as
$$
\Yinf= \Big\{\vec{y}= (\cdots, y_{-1},y_0,y_1,\ldots): \textrm{all } y_i\in \Z_{\geq 0}\cup\{+\infty\}, \textrm{ and there exists } x\in \Z \textrm{ such that } y_i=0 \textrm{ for all }i>x\Big\}.
$$
For $k\in \Z_{\geq 1}$, define the space of $k$-particle configurations on the line as
$$
\Gspace{k} = \Big\{\vec{g}=(\cdots, g_{-1},g_0,g_1,\ldots):\sum_i g_i = k\Big\}, \qquad \textrm{and}\qquad \Yspace{k} = \Big\{\vec{y}=(\cdots, y_{-1},y_0,y_1,\ldots):\sum_i y_i = k\Big\}.
$$
Though $\Gspace{k}$ and $\Yspace{k}$ are the same, we differentiate to keep track of the separate Markov processes for which they will be the state spaces.
The spaces $\Gspace{k}$ and $\Yspace{k}$ are (respectively) in bijections with ($\mathbb{W}$ for Weyl chamber)
$$
\Gmspace{k} = \Big\{\vec{m} = (m_1\leq \cdots\leq m_k): \textrm{all } m_i\in \Z\Big\},\qquad \textrm{and}\qquad \Ynspace{k} = \Big\{\vec{n} = (n_1\geq \cdots\geq n_k): \textrm{all } n_i\in \Z\Big\}.
$$
The bijection is given by associating to a state $\vec{g}$ or $\vec{y}$, the ordered list of the $k$ particle locations. As a convention, for $\vec{g}$ we associate $\vec{m}$ with weakly increasing particle location order, and for $\vec{y}$ we associated $\vec{n}$ with weakly decreasing particle location order. Given an operator $B$ acting on functions from $\Gspace{k}$ to $\C$, we will overload notation and let $B$ also denote the operator on functions $\tilde{f}$ from $\Gmspace{k}$ to $\C$ defined via $(B\tilde{f})(\vec{m}) = (B f)(\vec{g})$ where $\vec{g}$ and $\vec{m}$ are associated via the bijection, and $f(\vec{g}) = \tilde{f}(\vec{m})$. Finally, define (for later use) $\Ginf_{I}, \Yinf_{I}, \Gspace{k}_I$, and $\Yspace{k}_I$ to be restrictions of the respective spaces so that each $g_i$ or $y_i$ lies in $\{0,\ldots, I\}$. Define $\Gmspace{k}_I$ and $\Ynspace{k}_I$ as the  respective images of $\Gspace{k}$ and $\Yspace{k}$. In other words, having no clusters of equal $m_i$'s or $n_i$'s of length large than $I$.

Define the space of right-finite exclusion particle configurations
$$
\Xinf = \Big\{\vec{x}=(x_1>x_2>\cdots \big): \textrm{all } x_i\in \Z\Big\},
$$
and for $N\geq 1$ define the space of $N$-particle configurations
$$
\Xspace{N} = \Big\{\vec{x}=(x_1>\cdots>x_N\big): \textrm{all } x_i\in \Z\Big\}.
$$
By convention, we define `virtual particles' $x_i\equiv +\infty$ for all $i\in \Z_{\leq 0}$.
\end{definition}

\begin{definition}[$J=1$ high spin zero range process]\label{deftransprobJ1ZRP}

For $k\geq 1$, define the $k$-particle discrete time {\it $J=1$ higher spin zero range process} $\vec{g}(t)$ with state space $\Gspace{k}$ according to the following update rule (see Figure \ref{fig:crossesjoined}). Given state $\vec{g}$ we update to state $\vec{g}'$ sequentially.. Start at the left-most site $x\in \Z$ such that $g_x>0$. Let $h_x=0$ and randomly choose $g'_x$ and $h_{x+1}$ according to the probability distribution $\bernw{1}{q}{\nu}{\alpha}(g_x,h_x;g'_x,h_{x+1})$. Now proceed sequentially so that given $g_{x+1}$ and the choice of $h_{x+1}$ randomly choose $g'_{x+1}$ and $h_{x+2}$ according to the probability distribution $\bernw{1}{q}{\nu}{\alpha}(g_{x+1},h_{x+1};g'_{x+1},h_{x+2})$. Continue in this manner, augmenting $x$. Since there are finitely many particles and since all probabilities of the form $\bernw{1}{q}{\nu}{\alpha}(0,j_1;i_2,j_2)$ are strictly less than 1, it follows that eventually the output of the sequential update step will be all zeros. The update can be stopped and all subsequent $g'$ values are set to zero. We will denote the transition probability (given by the product of $\Lmat$-matrix weights) from state $\vec{g}$ to $\vec{g}'$ as $\highspinBoson{\alpha}{q\alpha}(\vec{g},\vec{g}')$ and the transition operator\footnote{The operator acts on functions $f:\Gspace{k}\to \C$ as $\big(\highspinBoson{\alpha}{q\alpha} f\big)(\vec{g}) = \sum_{\vec{g}'\in \Gspace{k}} \highspinBoson{\alpha}{q\alpha}(\vec{g},\vec{g}') f(\vec{g}')$} with matrix entries as $\highspinBoson{\alpha}{q\alpha}$. This operator has a well-defined action on functions $f$ which are bounded as $\vec{g}$ goes to infinity\footnote{By $\vec{g}$ going to infinity, we mean that the right-most particle in $\vec{g}$ goes to infinity.}. The odd superscript notation $\highspinBoson{\alpha}{q\alpha}$ is due to our eventual extension to $\highspinBoson{\alpha}{\beta}$ for arbitrary $\beta$ (cf. Remark \ref{sec:analcont}). Note that the dynamics of this process preserves the total number of particles due to the particle conservation property of the $\Lmat$-matrix.

We also define the infinite-particle version of the process $\vec{g}(t)$ with state space $\Ginf$. Since the update is from left to right and we are dealing with left-finite initial data, we can show that the update is well-defined. Towards this end, for $M\in \Z$, define the restriction of a state $\vec{g}$ to $\Z_{\leq M}$ as $\vec{g}\big\vert_{M} =  (g_i \cdot \mathbf{1}_{i\leq M})_{i\in \Z}$ and define a sequence of finite-particle processes $\vec{g}\big\vert_M(t)$ such that $\vec{g}\big\vert_M(0) = \vec{g}\big\vert_M$. Each of these initial data has finitely many particles, and hence the evolution according to the above defined finite-particle version of the $J=1$ higher spin zero range process is well-defined. We define $\vec{g}(t)$ as the inverse limits (in law) of the restriction of $\vec{g}\big\vert_M(t)$ to $\Z_{\leq M}$ (Note: even though the initial data was restricted to $\Z_{\leq M}$, the evolution may have left that sector). Because of the left to right update and one-sided nature to particle movement, these restrictions form a consistent family and the desired inverse limit exists. Though we cannot write transition probabilities, we can still define a transition operator acting on a suitable domain of functions. The operator $\highspinBoson{\alpha}{q\alpha}$ will act on functions $f:\Ginf\to \C$ which are stable at infinity. By stable at infinity we mean that for each $\vec{g}\in \Ginf$, over all $\vec{g}' \geq \vec{g}$, $f(\vec{g}')$ is uniformly bounded and $f(\vec{g}'\big\vert_{M})$ converges uniformly to $f(\vec{g}')$. The inequality $\vec{g}'\geq \vec{g}$ means that for all $x\in \Z$, $\sum_{i\leq x} g_i \geq \sum_{i\leq x} g'_i$ (in other words, the state $\vec{g}'$ can be accessed from $\vec{g}$ via moving particles to the right). For such stable functions, $\lim_{M\to \infty}\big(\highspinBoson{\alpha}{q\alpha} f\big)(\vec{g}\big\vert_M)$ exists for all $\vec{g}$, and defines $\big(\highspinBoson{\alpha}{q\alpha} f\big)(\vec{g})$.

In the same manner as above, we define space reversed zero range processes $\vec{y}(t)$ in which particles are updated right to left with state space $\Yspace{k}$ (or $\Yinf$). This reversed process involves the $\Lmat$-matrix which is reflected in the $y$-axis. The space reversed transition operator $\revhighspinBoson{\alpha}{q\alpha} = P \highspinBoson{\alpha}{q\alpha} P^{-1}$ where $\big(P f\big)\big((y_i)_{i\in \Z}\big) = f\big((y_{-i})_{i\in \Z}\big)$ is the space reversal operator (note that $P^{-1}=P$). We may likewise extend from finite to infinite particle configurations.
\end{definition}

\begin{rem}\label{rem:trans}
The $J=1$ higher spin zero range process describe in Definition \ref{deftransprobJ1ZRP} should be thought of as a full line version of the transfer matrix built from the $\Lmat$-matrix with matrix elements given in Definition \ref{defweightsJ1}. The standard construction of a transfer matrix on $\Z/L\Z$ involves taking the product of $\Lmat$-matrices and then tracing out the horizontal space. In other words, one defines the transfer matrix as $\mathrm{tr}_{[\VJ^{J}]_1}\big([\bernw{1}{q}{\nu}{\alpha}]_{1,1}[\bernw{1}{q}{\nu}{\alpha}]_{2,1}\cdots [\bernw{1}{q}{\nu}{\alpha}]_{L,1}\big)$. The resulting matrix maps $[\VI^{I}]_1\otimes [\VI^{I}]_2\otimes \cdots\otimes [\VI^{I}]_L$ to itself. Under standard normalization, it is not clear how to directly construct transfer matrices on $\Z$. However, if we work with stochastic $\Lmat$-matrices, then all weights are strictly less than 1, except for the weight $\bernw{1}{q}{\nu}{\alpha}(0,0;0,0)=1$. This enables us to make sense of the infinite product of these (stochastic) $\Lmat$-matrices, at least when restricted to the sector in which there is a finite total number of particles. As in Definition \ref{deftransprobJ1ZRP}, with a little more work one can likewise construct infinite particle number versions of these resulting stochastic transfer matrices (see also \cite{BCG6V}).
\end{rem}

\begin{rem}
In light of Remark \ref{rem:BorR}, our ZRP transition probability $\revhighspinBoson{\alpha}{q\alpha}(\vec{n};\vec{n}')$ for $\vec{n},\vec{n}'\in \Ynspace{k}$ is related to the weight $G_{\vec{n}/\vec{n}'}$ up-to a simple conjugation by an eigenfunction of $G$. Thus, our stochastic transition operators are Doob $h$-transforms of the transfer matrices considered in \cite{Borodin2014vertex}. It would be interesting to investigate whether conjugation with respect to other eigenfunctions result in stochastic transfer matrices.
\end{rem}

\begin{rem}\label{rem:star}
The construction of the finite particle zero range process given in Definition \ref{deftransprobJ1ZRP} is invariant under conjugation of the $\Lmat$-matrix elements via multiplication by $f(j_1)/f(j_2)$ for any non-zero function $f:\VJ^J\to \C$. Such a conjugation may, however, destroy the stochasticity of the individual $\Lmat$-matrices (despite retaining that of the entire transfer matrix). 
\end{rem}

\begin{figure}
\begin{center}
	\includegraphics[scale=.8]{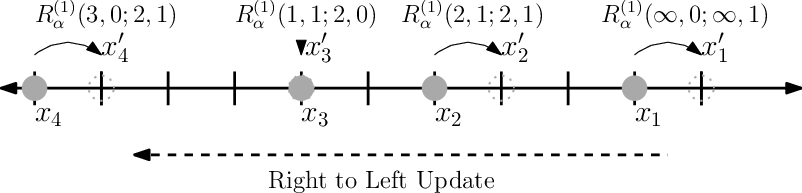}
\end{center}
\caption{The exclusion process constructed from the $\Lmat$-matrix. The first particle $x_1$ is updated and then, based on the length of the previous update, each subsequent particle is updated.}
\label{fig:TASEP}
\end{figure}

\begin{definition}[$J=1$ higher spin exclusion process]\label{deftransprobJ1AEP}
Via a gap/particle transform\footnote{We may consider our zero range process as describing the size of gaps between labeled particles. This equivalence is unique up to an overall shift in the labeled particle positions.}, we define an exclusion process. For $N\geq 1$, define the $N$-particle discrete time {\it $J=1$ higher spin exclusion process} $\vec{x}(t)$ with state space $\Xspace{N}$ according to the following update rule (see Figure \ref{fig:TASEP}) . Given state $\vec{x}$ we update to state $\vec{x}'$ sequentially. Start at $x_1$, randomly choose $x'_1\in \{x_1,x_1+1,\ldots\}$ according to the probability distribution\footnote{This corresponds to taking the well-defined limit $m\to \infty$ in Definition \ref{defweightsJ1}.} $\bernw{1}{q}{\nu}{\alpha}(\infty,0;\infty,x'_1-x_1)$. Proceed sequentially up to $j=N$ so that given $x'_j$ and $x_{j+1}$, randomly choose $x'_{j+1}\in \{x_{j+1},\ldots, x_j -1\}$ according to the probability distribution $\bernw{1}{q}{\nu}{\alpha}(x_j-x_{j+1}-1,x_j'-x_j;x'_j-x'_{j+1}-1,x'_{j+1}-x_{j+1})$. We will denote the transition probability from state $\vec{x}$ to $\vec{x}'$ as $\highspinTASEP{\alpha}{q\alpha}(\vec{x},\vec{x}')$, and the associated transition operator with these matrix entries as $\highspinTASEP{\alpha}{q\alpha}$. This transition operator has a well-defined action on functions which are bounded as $\vec{x}$ goes to infinity. In order to define the infinite version of this process with state space $\Xinf$ it suffices to observe that the first $N$ particles evolve according to the $N$-particle exclusion process (and hence the infinite-particle version can be defined as an inverse limit of consistent laws of these processes). We can also define the transition operator's action on functions $f:\Xinf\to \C$ which are stable at infinity. By stable at infinity we mean that for each $\vec{x}\in \Xinf$, over all $\vec{x}'\geq \vec{x}$, $f(\vec{x}')$ is uniformly bounded and $f(\vec{x}'\big\vert_{N})$ converges uniformly to $f(\vec{x}')$. The inequality $\vec{x}'\geq \vec{x}$ means that for all $i\in \Z_{\geq 1}$, $x_i'\geq x_i$, and $\vec{x}'\big\vert_{N} = (x_1,\ldots,x_N)$. For such stable functions  $\lim_{N\to \infty}\big(\highspinTASEP{\alpha}{q\alpha} f\big)(\vec{x}\big\vert_N)$ exists for all $\vec{x}$, and defines $\big(\highspinTASEP{\alpha}{q\alpha} f\big)(\vec{x})$.
\end{definition}

\begin{definition}[$q$-Hahn distribution and zero range process]\label{def:qhahn}
For generic complex $q,\mu,\nu$, and $y\in \Z_{\geq 0}$, define the (complex) probability distribution\footnote{One condition under which this is a bona-fide positive probability distribution is if $|q|<1$ and $0\leq \nu\leq\mu<1$. For discussion and references regarding this distribution, see \cite{Povolotsky2013,Corwin2014qmunu}.} on $s\in \{0,\ldots, y\}$ as
$$
\varphi_{\mu}(s|y)=\varphi_{q,\mu,\nu}(s|y) = \mu^s\frac{(\nu/\mu;q)_{s}(\mu;q)_{y-s}}{(\nu;q)_{y}} \frac{(q;q)_{y}}{(q;q)_{s}(q;q)_{y-s}}, \qquad\textrm{with}\qquad (a;q)_n :=\prod_{j=1}^{n} (1-aq^{j-1}).
$$
For $k\geq 1$ fixed and $i\in \Z$, define linear operators $\big[\qHahnA{q}{\nu}{\mu}\big]_i$ via their action on functions $f:\Yspace{k}\to \C$
$$
\Big(\big[\qHahnA{q}{\nu}{\alpha}\big]_i f\Big)(\vec{y}) = \sum_{s_i=0}^{y_i} \varphi_{-\alpha}(s_i|y_i) f\big(\vec{y}_{i,i-1}^{s_i}\big).
$$
Note the inclusion of the negative sign in $\varphi_{-\alpha}(s_i|y_i)$. Here, for $\vec{y}\in \Yspace{k}$, we have set $
\vec{y}_{i,i-1}^{s} = (\ldots, y_{i-1}+s,y_i -s,\ldots,)$.
Define the $k$-particle space reversed {\it $q$-Hahn zero range process} transition operator\footnote{If the parameters $q,\mu,\nu$ are such that $\varphi_{\mu}$ is always positive, then $[\qHahnA{q}{\nu}{\alpha}\big]_i$ corresponds with the Markov update by which $s_i$ particles are moved from position $i$ to $i-1$ according to the distribution $\varphi_{-\alpha}(s_i|y_i)$, and $\qHahnH{q}{\nu}{\alpha}$ is the Markov update by which one updates each site in parallel according to the $\qHahnA{q}{\nu}{\alpha}$ single-site update. We do not, however, rely upon such positivity in our use of this distribution and its associated operators.}
$$
\qHahnH{q}{\nu}{\alpha}= \cdots \big[\qHahnA{q}{\nu}{\alpha}\big]_{-1}\big[\qHahnA{q}{\nu}{\alpha}\big]_{0}\big[\qHahnA{q}{\nu}{\alpha}\big]_{1}\cdots.
$$
In the same way as for $\highspinBoson{\alpha}{q\alpha}$ (up to space reversal), we also define the transition operator's action on functions $f:\Yinf\to \C$ which are stable at negative infinity.
Likewise, define the non-space reversed operator $\qHahnHR{q}{\nu}{\alpha} = P\qHahnH{q}{\nu}{\alpha}P^{-1}$
where $P$ is the space reversal operator from Definition \ref{deftransprobJ1ZRP}.
\end{definition}

\begin{rem}\label{rem:followingprop}
	We will use the following properties of  $\varphi_{\mu}(s|y)$ which one can readily check (here $s\in\{1,\ldots,y\}$):
	\begin{align}
		\varphi_{\mu}(s|y)
		&=\varphi_{\mu}(s-1|y)
		\cdot\mu\, \frac{1-q^{y+1-s}}{1-q^{s}}
		\frac{1-\nu q^{s-1}/\mu}{1-\mu q^{y-s}},
		\label{varphi_property_1}
		\\
		\label{varphi_property_2}
		\varphi_{q\mu}(s|y)
		&=
		\frac{1}{1-\mu}
		\Big(
		(1-\mu q^{y-s})
		\varphi_{\mu}(s|y)
		-\mu(1-q^{y+1-s})
		\varphi_{\mu}(s-1|y)
		\Big).
	\end{align}
\end{rem}

\subsection{Bethe ansatz diagonalization}\label{sec:diagonal}

In Appendix \ref{sec:Bethe} we recall the Bethe ansatz eigenfunctions which, according to the result we now prove, diagonalize the higher spin zero range process with transition matrix $\revhighspinBoson{\alpha}{q\alpha}$. We will rely upon the Plancherel theory developed in \cite{BCPS2014} for these eigenfunctions (also reviewed in the appendix). In what follows we will, as described in Definition \ref{deftransprobJ1}, overload notation and let operators acting on functions of the $\vec{y}$ variables also act on functions of the $\vec{n}$ variables via their bijective association.

\begin{proposition}\label{prophighspindiagJ1}
Assuming $\big\vert \tfrac{1-z_i}{1-\nu z_i}\, \tfrac{\alpha+\nu}{1+\alpha}\big\vert<1$
for $1\leq i\leq k$, then\footnote{By applying the space reversal operator $P$ we can also produce eigenrelations for $\highspinBoson{\alpha}{q\alpha}$.}
$$
\big(\revhighspinBoson{\alpha}{q\alpha} \Psil_{\vec{z}}\big)(\vec{n}) = \prod_{i=1}^{k} \frac{1+q\alpha z_i}{1+ \alpha z_i} \, \Psil_{\vec{z}}(\vec{n})
$$
\end{proposition}
\begin{proof}
We appeal to the known eigenfunction relations for the $J=1$ higher spin six-vertex model. This can be derived on the periodic lattice via algebraic Bethe ansatz (cf. \cite{Reshetikhin2008}). We use \cite[Corollary 4.5 (i)]{Borodin2014vertex} wherein it is shown (via a symmetric function theory approach) that if
$$
\Big|\frac{z_i-s}{1-s z_i} \, \frac{u-s}{1-su}\Big|<1
$$
for all $i\in \{1,\ldots, k\}$,
then
$$
\prod_{i=1}^{k}\frac{1-qu z_i}{1-u z_i} c(\vec{n}) F_{\vec{n}}(z_1,\ldots,z_k) = \sum_{\vec{n}'} G_{\vec{n}'/\vec{n}}(v) c(\vec{n}') F_{\vec{n}'}(z_1,\ldots, z_k).
$$
Here $\vec{n}',\vec{n}\in \Ynspace{k}$ and, after changing $z\mapsto sz$,
$$
c(\vec{n}) F_{\vec{n}}(sz_1,\ldots,sz_k)= \prod_{i=1}^{k} \frac{(-s)^{n_i}}{1-s^2 z_i}  \Psir_{\vec{z}}(\vec{n})
$$
where $\nu=s^2$ in $\Psir_{\vec{z}}(\vec{n})$. According to \cite[Definition 3.2]{Borodin2014vertex} the term $G_{\vec{n}'/\vec{n}}(v)$ is equal to the product of the weights from $\vec{n}'$ to $\vec{n}$ where the weights are given as in \cite[Definition 2.1]{Borodin2014vertex}. Remark \ref{rem:BorR} explains the relation of the weights in \cite{Borodin2014vertex} to those considered herein and setting $\alpha=-s v$ and $\nu=s^2$  yields
$$
\prod_{i=1}^{k}\frac{1+q\alpha z_i}{1+\alpha z_i} \, \Psir_{\vec{z}}(\vec{n})  =\Big(\big(\revhighspinBoson{\alpha}{q\alpha}\big)^{T} \Psir_{\vec{z}}\Big)(\vec{n}).
$$
%
%
This, likewise, implies the desired relationship for left eigenfunctions as well.
\end{proof}

\begin{proposition}\label{propqHahndiag}
For $\alpha\in \C$ and $1-\nu z_j\neq 0$ for $1\leq j\leq k$,
$$
\big(\qHahnH{q}{\nu}{\alpha} \Psil_{\vec{z}}\big)(\vec{n}) = \prod_{i=1}^{k} \frac{1+\alpha z_i}{1-\nu z_i}\, \Psil_{\vec{z}}(\vec{n})
$$
\end{proposition}
\begin{proof}
This follows from \cite{Povolotsky2013,Corwin2014qmunu,BCPS2014}. In particular, \cite[Proposition 5.13]{BCPS2014} records the desired result for $0\leq \nu\leq -\alpha <1$. The operator $\qHahnH{q}{\nu}{\alpha}$ depends polynomially on $\alpha$, as does the eigenvalue. Thus, since both sides above are polynomial in $\alpha$ and equal for an interval of values, they must match for all $\alpha\in \C$.
\end{proof}

In Section \ref{sec:qHahn} we develop the relationship between the operator $\qHahnH{q}{\nu}{\alpha}$ and higher spin versions of the zero range process transition operator.

In order to understand the following corollary, the reader is encouraged to recall from Appendix \ref{sec:Bethe} the direct and inverse transforms $\Pld, \Pli$ and the space $\Wc^{k}_{\max}$ on which $\Pli\Pld$ acts as the identity.

\begin{corollary}\label{cor:secrat}
On the space $\Wc^{k}_{\max}$, $\highspinBoson{\alpha}{q\alpha} = \big(\qHahnHR{q}{\nu}{\alpha}\big)^{-1}\qHahnHR{q}{\nu}{q\alpha}$, $\revhighspinBoson{\alpha}{q\alpha} = \big(\qHahnH{q}{\nu}{\alpha}\big)^{-1}\qHahnH{q}{\nu}{q\alpha}$, and $\qHahnH{q}{\nu}{\alpha}$ commutes with itself for different values of $\alpha$.
\end{corollary}

\begin{proof}
This follows by spectral considerations. Let's prove the first claim. The formula for $\Pli$ implies that\footnote{In the below formulas, the dot represents the variable integrated in the application of the inverse transform.}
\begin{align*}
\Big(\big(\qHahnH{q}{\nu}{\alpha}\big)^{-1} f\Big)(\vec{n}) &= \Pli\Big(\big(ev_{\alpha}(\cdot)\big)^{-1} \big(\Pld f\big)(\cdot)\Big)(\vec{n})\\
\Big(\qHahnH{q}{\nu}{q\alpha} f\Big)(\vec{n}) &= \Pli\Big(ev_{q\alpha}(\cdot)\big(\Pld f\big)(\cdot)\Big)(\vec{n}),
\end{align*}
where $ev_{\alpha}(\vec{z}) = \prod_{i=1}^{k} \frac{1+\alpha z_i}{1-\nu z_i}$.
Combining these yields
$$
\Big(\big(\qHahnH{q}{\nu}{\alpha}\big)^{-1} \qHahnH{q}{\nu}{q\alpha} f\Big)(\vec{n}) = \Pli\Big(ev_{\alpha,q\alpha}(\cdot) \big(\Pld f\big)(\cdot)\Big)(\vec{n})
$$
where $ev_{a,b}(\vec{z}) = \prod_{i=1}^{k} \frac{1+ b z_i}{1+ a z_i}$. On the other hand, it follows from Proposition \ref{prophighspindiagJ1}
that
$$
\big(\revhighspinBoson{\alpha}{q\alpha} f\big)(\vec{n}) = \Pli\Big(ev_{\alpha,q\alpha}(\cdot) \big(\Pld f\big)(\cdot)\Big)(\vec{n})
$$
as well. Conjugating everything by the space reversal operator $P$ produces the second claimed result. Similar considerations and Proposition \ref{propqHahndiag} imply the last commutation relation.
\end{proof}

\begin{rem}
One of the methods used in Bethe ansatz is to rewrite an operator as the direct sum of one-dimensional operators subject to two-body boundary conditions. Not every higher spin transition operator is amenable to this method -- for instance, the stochastic six-vertex model \cite{BCG6V}. This proposition shows that $\highspinBoson{\alpha}{q\alpha}$ (and eventually via fusion, the general $\highspinBoson{\alpha}{\beta}$ operator) can be written as the ratio of $q$-Hahn transition operators. It would be interesting to see if this fact goes through to the case of transfer matrices on the finite lattice $\Z/L\Z$.
\end{rem}

\subsection{Self duality}\label{sec:implicitduality}

%

\begin{definition}\label{def:dualityfunctionals}
We define a number of duality functionals. For $\vec{x}\in \Xinf$ and $\vec{y}\in \Yspace{k}$ for some $k\in \Z_{\geq 1}$, define\footnote{We employ the convention that the product is zero if $y_i>0$ for any $i\leq 0$. This is in accordance with the convention that $x_i(\cdot)\equiv +\infty$ for $i\leq 0$.}
$$
\HH(\vec{x},\vec{y}) = \prod_{i\in \Z} q^{(x_i+i)y_i}.
$$
For $\vec{g}\in \Ginf$ and $\vec{y}\in \Yinf$ define
$$
\GG(\vec{g},\vec{y}) = q^{\sum_{i>j} g_i y_j},\qquad \textrm{and}\qquad \GGhat(\vec{g},\vec{y}) = q^{-\sum_{i\leq j} g_i y_j}.
$$
Notice that unless $\vec{g}\in \Gspace{k}$ and $\vec{y}\in \Yspace{k'}$ for some $k,k'\in \Z_{\geq 1}$, $\GG(\vec{g},\vec{y})$ will equal zero. In the case that $\vec{g}\in \Gspace{k}$ and $\vec{y}\in \Yspace{k'}$ notice that
\begin{align}\label{eq:kkprime}
\GG(\vec{g},\vec{y}) = \GGhat(\vec{g},\vec{y}) \, \cdot \, q^{kk'}.
\end{align}
For $\vec{g}\in \Gspace{k}$ and $\vec{y}\in \Yspace{k'}$ we may overload these functionals as described in Definition \ref{deftransprobJ1} by replacing $\vec{g}$ by $\vec{m}\in \Gmspace{k}$ and $\vec{y}$ by $\vec{n}\in \Ynspace{k'}$.
For $m\in \Z$ (or $n\in \Z$) and $\vec{g}\in \Ginf$ (or $\vec{y}\in \Yinf$)
$$
\Ndown_m(\vec{g}) = \sum_{\ell\leq m} g_{\ell},\qquad \textrm{and} \qquad \Nup_m(\vec{g}) = \sum_{\ell\geq m}g_{\ell}.
$$
For $\vec{g}\in \Gspace{k}$ bijectively equivalent to $\vec{m}\in \Gmspace{k}$ and $\vec{y}\in \Yinf$ we have
$$
\GG(\vec{g},\vec{y}) = \GG(\vec{m},\vec{y}) = \prod_{i=1}^{k} q^{\Ndown_{m_i -1}(\vec{y})},\qquad \textrm{and}\qquad \GGhat(\vec{g},\vec{y}) = \GGhat(\vec{m},\vec{y}) = \prod_{i=1}^{k} q^{-\Nup_{m_i}(\vec{y})}.
$$
For $\vec{g}\in \Ginf$ and $\vec{y}\in \Yspace{k'}$ bijectively equivalent to $\vec{n}\in \Ynspace{k'}$ we have
$$
\GG(\vec{g},\vec{y}) = \GG(\vec{g},\vec{n}) = \prod_{i=1}^{k'} q^{\Nup_{n_i +1}(\vec{g})},\qquad \textrm{and}\qquad \GGhat(\vec{g},\vec{y}) = \GGhat(\vec{g},\vec{n}) = \prod_{i=1}^{k'} q^{-\Ndown_{n_i}(\vec{g})}.
$$
\end{definition}

\begin{definition}\label{def:welladapt}
Call $\vec{x}\in \Xinf$ {\it well-adapted} to $\HH$ if for $\vec{y}\in \Yspace{k}$ the function $\vec{y}\mapsto \HH(\vec{x},\vec{y})$ lies in the space $\Wc^{k}_{\max}$ (see Definition \ref{def:Wmax}).
Call $\vec{g}\in \Ginf$ {\it well-adapted} to $\GG$ if for $\vec{y}\in \Yspace{k}$ the function $\vec{y}\mapsto \GG(\vec{g},\vec{y})$ lies in the space $\Wc^{k}_{\max}$.
Call $\vec{y}\in \Yinf$ {\it well-adapted} to $\GG$ if for $\vec{g}\in \Gspace{k}$ the space reversal of the function $\vec{g}\mapsto \GG(\vec{g},\vec{y})$ lies in the space $\Wc^{k}_{\max}$.
\end{definition}

\begin{rem}\label{rem:welladapt}
There are a number of readily accessible examples of well-adapted $\vec{x},\vec{g}$, and $\vec{y}$.
Step initial data $x_i=-i$, $i\in \Z_{\geq 1}$ is well-adapted to $\HH$. To show this, note that for $c>0$ small enough and $C>0$ large enough the function $\vec{y}\mapsto \HH(\vec{x},\vec{y})$ lies in $\Wc^{k}_{\exp(c,C)}$ which, by Corollary \ref{cor:Wmax} is a subset of $\Wc^{k}_{\max}$. Similarly, one shows that spiked initial data $g_i = +\infty \mathbf{1}_{i=0}$ and likewise $y_i = +\infty \mathbf{1}_{i=0}$ are well-adapted to $\GG$.
\end{rem}

We now state and prove our first duality result, one between the exclusion process $\vec{x}(t)$ and zero range process $\vec{y}(t)$. The approach of the proof follows, for the most part, that of the proof of the discrete time Bernoulli $q$-TASEP duality in \cite{BorodinCorwin2013discrete}. Indeed, setting $\nu=0$ our exclusion process becomes the Bernoulli $q$-TASEP.

\begin{theorem}\label{thmimplicit}
The $J=1$ higher spin exclusion process with transition operator $\highspinTASEP{\alpha}{q\alpha}$ is dual to the space reversed $J=1$ higher spin zero range process with transition operator $\revhighspinBoson{\alpha}{q\alpha}$ with respect to $\HH(\vec{x},\vec{y})$. Precisely,
$$
\highspinTASEP{\alpha}{q\alpha} \HH = \HH \big(\revhighspinBoson{\alpha}{q\alpha}\big)^{T},
$$
where the equality holds for all matrix elements indexed by $\vec{x}\in \Xinf$ well-adapted to $\HH$ and all $\vec{y}\in \Yspace{k}$.
\end{theorem}

\begin{proof}
We will demonstrate the following `implicit' version of the self-duality. Consider the exclusion process $\vec{x}(\cdot)$ from Definition \ref{deftransprobJ1AEP} with some (possibly random) initial data $\vec{x}(0)\in \Xinf$ and define
$$
I_t(\vec{y}) = \EE\Big[\HH\big(\vec{x}(t),\vec{y}\big)\Big]
$$
where the expectation is over the evolution of $\vec{x}(\cdot)$ (as well as the initial data, if it is random), and $\vec{y}\in \Yspace{k}$ for some $k\in \Z_{\geq 1}$. Then, assuming the initial data is such that $I_t(\vec{y})$ is everywhere finite, we claim that following equality
\begin{align}\label{eq:qhahnim}
\big(\qHahnH{q}{\nu}{\alpha} I_{t+1}\big)(\vec{y}) = \big(\qHahnH{q}{\nu}{q\alpha} I_t\big)(\vec{y}),
\end{align}
holds for all $t\in \Z_{\geq 0}$ and all $\vec{y}\in \Yspace{k}$.

Before proving \eqref{eq:qhahnim}, let us observe how it readily leads to the proof of the theorem. In matrix notation, \eqref{eq:qhahnim} implies that
$$
\highspinTASEP{\alpha}{q\alpha}\HH (\qHahnH{q}{\nu}{\alpha})^T =  \HH (\qHahnH{q}{\nu}{q\alpha})^T,
$$
where we note that $\HH$ is easily seen to be in the domain of these transition operators. Since in the hypotheses of the theorem we have assumed that $\vec{x}\in \Xinf$ well-adapted to $\HH$, we can appeal to Corollary \ref{cor:secrat} (which shows that $\revhighspinBoson{\alpha}{q\alpha} = \big(\qHahnH{q}{\nu}{\alpha}\big)^{-1}\qHahnH{q}{\nu}{q\alpha}$ when acting on function in $\Wc^{k}_{\max}$) to complete the proof of the theorem.

Thus, it remains to demonstrate \eqref{eq:qhahnim} and the remaining portion of this proof is devoted to that goal.

\medskip

Let $N$ be the smallest integer such that all $y_i=0$ for $i>N$.
Note that by the convention on $\HH$ being zero if $y_i>0$ for any $i\leq 0$, it suffices to restrict the product $i\leq N$ to $0\leq i\leq N$ and modify the definition of $\qHahnH{q}{\nu}{\alpha}$ to only include terms $\big[\qHahnA{q}{\nu}{\alpha}\big]_i$ for $1\leq i\leq N$. In other words, we can treat 0 as a sink since our functional becomes 0 for all $i\leq 0$. We may now argue similarly to \cite{BorodinCorwin2013discrete}.

For $j\ge1$ and $t\ge0$, let $\sigma^{j}_{t}$ denote the sigma-algebra generated by the
	random variables $x_{1}(t),\ldots,x_{j}(t)$.
	Conditioning on the history of the whole process up to time $t$,
	we will show that, as $\sigma^{N}_{t}$ measurable
	random variables,
	\begin{align}
		\qHahnH{q}{\nu}{\alpha}
		\EE\left[\prod_{i=0}^{N}q^{(x_i(t+1)+i)y_i}\,\big\vert\, \sigma^{N}_{t}\right]
		=
		\qHahnH{q}{\nu}{q\alpha}
		\prod_{i=0}^{N}q^{(x_i(t)+i)y_i}.
		\label{Bernoulli_duality_proof1}
	\end{align}
	
	Since operators $\qHahnH{q}{\nu}{\alpha}$
	and $\qHahnH{q}{\nu}{q\alpha}$
	in both sides have a sequential structure
	(corresponding to first moving $s_1$ particles from $y_1$ to $y_0$,
	then $s_2$ particles from $y_2$ to $y_1$, etc.),
	we can further condition on what happened to
	particles $x_1,\ldots,x_{i-1}$ during time step $t\to t+1$
	(for any $i=1,\ldots,N$).
	We will show that the relevant contributions
	to both sides of \eqref{Bernoulli_duality_proof1}
	behave as they should (i.e., the
	parameter $\alpha$ in the operator at time $t+1$ is replaced by $q \alpha$ at time $t$). That is, we will
	show that for $i\ge 2$,
	\begin{align}\label{Bernoulli_duality_proof2}
		&
		\sum_{s_i=0}^{y_i}
		\varphi_{-\alpha}(s_i|y_i)
		\EE\left[
		q^{(x_{i}(t+1)+i)(y_i-s_i)}
		q^{(x_{i-1}(t+1)+i-1)s_i}
		\,\big\vert\, \sigma_{t}^{N}, \sigma_{t+1}^{i-1}\right]
		\\&\hspace{140pt}\nonumber
		=\sum_{s_i=0}^{y_i}
		\varphi_{-q\alpha}(s_i|y_i)
		q^{(x_{i}(t)+i)(y_i-s_i)}
		q^{(x_{i-1}(t)+i-1)s_i},
	\end{align}
	and for $i=1$,
	\begin{align}\label{Bernoulli_duality_proof3}
		\varphi_{-\alpha}(0|y_1)
		\EE\left[
		q^{(x_{1}(t+1)+1)y_1}
		\,\big\vert\, \sigma_{t}^{N}\right]
		=\varphi_{-q\alpha}(0|y_1)
		q^{(x_{1}(t)+1)y_1}.
	\end{align}
	First, note that \eqref{Bernoulli_duality_proof3} is
	straightforward: conditioned on the knowledge of $x_1(t)$,
	the first particle jumps to the right by one with probability
	$\alpha/(1+\alpha)$ and stays put with probability $1/(1+\alpha)$.
	Therefore,
	\begin{align*}
		\varphi_{-\alpha}(0|y_1)
		\EE\left[
		q^{(x_{1}(t+1)+1)y_1}
		\,\big\vert\, \sigma_{t}^{N}\right]
		&=
		\varphi_{-\alpha}(0|y_1)
		\left(
		\frac{\alpha}{1+\alpha}q^{(x_{1}(t+1)+1)y_1}q^{y_1}
		+
		\frac{1}{1+\alpha}q^{(x_{1}(t+1)+1)y_1}
		\right)\\&=
		\varphi_{-q\alpha}(0|y_1)
		q^{(x_{1}(t)+1)y_1}.
	\end{align*}
	Here we have used $s=0$
	case of \eqref{varphi_property_2}
	with $\mu=-\alpha$,
	which reads
	\begin{align*}
		\varphi_{-q \alpha}(0|y)
		&=
		\frac{1}{1+\alpha}
		(1+\alpha q^{y-s})
		\varphi_{-\alpha}(0|y).
	\end{align*}

	To show \eqref{Bernoulli_duality_proof2},
	denote by $I$ the indicator of the event that
	$x_{i-1}(t+1)=x_{i-1}(t)+1$, i.e., that the particle $x_{i-1}$
	has jumped to the right by one during time step $t\to t+1$.
	This indicator is $\sigma^{i-1}_{t+1}$-measurable,
	and it will help us to compute the conditional expectation
	in the left-hand side of \eqref{Bernoulli_duality_proof2}.
	Also for any $r,s\ge0$ denote
	$Z_{r,s}:=q^{(x_i(t)+i)r}q^{(x_{i-1}(t)+i-1)s}$.
	Using the definition of the dynamics
	of $\vec{x}$, we can write
	\begin{align*}
		&\EE\Big[
		q^{(x_i(t+1)+i)r}q^{(x_{i-1}(t+1)+i-1)s}\,\big\vert\, \sigma^{N}_{t},\sigma^{i-1}_{t+1}\Big]
		\\&\hspace{20pt}=
		Iq^{s}Z_{r,s}
		\left(
		q^{r}
		\frac{\alpha+\nu q^{x_{i-1}(t)-x_{i}(t)-1}}{1+\alpha}+
		\frac{1-\nu q^{x_{i-1}(t)-x_{i}(t)-1}}{1+\alpha}\right)
		\\&\hspace{60pt}+
		(1-I)Z_{r,s}
		\left(
		q^{r}\frac{\alpha(1-q^{x_{i-1}(t)-x_{i}(t)-1})}{1+\alpha}
		+
		\frac{1+\alpha q^{x_{i-1}(t)-x_{i}(t)-1}}{1+\alpha}
		\right).
	\end{align*}
	Noting that $q^{x_{i-1}(t)-x_{i}(t)-1}Z_{r,s}=Z_{r-1,s+1}$,
	we can simplify the right-hand side above to
	\begin{align*}
		&=\frac{I}{1+\alpha}
		\Big(
		(\alpha q^{r+s}+q^{s}-\alpha q^{r}-1)Z_{r,s}
		+(\nu q^{r+s}-\nu q^{s}+\alpha q^{r}-\alpha)Z_{r-1,s+1}
		\Big)
		\\&\hspace{40pt}+\frac{1}{1+\alpha}
		\Big((\alpha q^{r}+1)Z_{r,s}+
		(-\alpha q^{r}+\alpha)Z_{r-1,s+1}\Big)
	\end{align*}
	(note that when $r=0$, the coefficient by $Z_{r-1,s+1}$
	is zero.)
	The left-hand side of \eqref{Bernoulli_duality_proof2}
	is then equal to the sum (over $s_i$) of
	$\varphi_{-\alpha}(s_i|y_i)$
	times
	the above expressions
	with $r=y_i-s_i$ and $s=s_i$.
	That is, left-hand side of \eqref{Bernoulli_duality_proof2}
	takes the form
	\begin{align*}
		&\frac{I}{1+\alpha}
		\sum_{s_i=0}^{y_i}Z_{y_i-s_i,s_i}
		\Big[
		\varphi_{-\alpha}(s_i|y_i)
		(\alpha q^{y_i}+q^{s_i}-\alpha q^{y_i-s_i}-1)
		\\&\hspace{120pt}+
		\varphi_{-\alpha}(s_i-1|y_i)
		(\nu q^{y_i}-\nu q^{s_i-1}+\alpha q^{y_i-s_i+1}-\alpha)
		\Big]
		\\&+
		\frac{1}{1+\alpha}
		\sum_{s_i=0}^{y_i}
		Z_{y_i-s_i,s_i}\Big[
		\varphi_{-\alpha}(s_i|y_i)
		(\alpha q^{y_i-s_i}+1)
		+\varphi_{-\alpha}(s_i-1|y_i)
		(-\alpha q^{y_i-s_i+1}+\alpha)
		\Big].
	\end{align*}
	By properties of $\varphi_{-\alpha}$
	\eqref{varphi_property_1}, \eqref{varphi_property_2},
	the expression in the square brackets in the first sum vanishes
	for any $s_i$; and
	the expression in the square brackets in the second sum
	is equal to $\varphi_{-q\alpha}(s_i|y_i)$.
	This yields \eqref{Bernoulli_duality_proof2}.
		
	\smallskip

	Having now established \eqref{Bernoulli_duality_proof2} and \eqref{Bernoulli_duality_proof3},
	we can now prove \eqref{Bernoulli_duality_proof1}. We
	have (assuming $y_0=0$, otherwise \eqref{Bernoulli_duality_proof1} is
	trivial)
	\begin{align*}
		&\big[\qHahnA{q}{\nu}{\alpha}\big]_1 \ldots \big[\qHahnA{q}{\nu}{\alpha}\big]_{N-1}
		\big[\qHahnA{q}{\nu}{\alpha}\big]_N
		\EE\bigg[
		q^{(x_1(t+1)+1)y_1}
		\EE\Big[q^{(x_2(t+1)+1)y_2}
		\ldots
		\\&\hspace{20pt}\ldots
		\EE\Big[
		q^{(x_{N-1}(t+1)+N-1)y_{N-1}}
		\EE\Big[
		q^{(x_{N}(t+1)+N)y_{N}}
		\,\big\vert\, \sigma^{N}_{t},\sigma^{N-1}_{t+1}
		\Big]
		\,\big\vert\, \sigma^{N}_{t},\sigma^{N-2}_{t+1}
		\Big]\ldots
		\,\big\vert\, \sigma^{N}_{t},\sigma^{1}_{t+1}
		\Big]
		\,\big\vert\, \sigma^{N}_{t}
		\bigg]
		\\&=
		\big[\qHahnA{q}{\nu}{\alpha}\big]_{1}\ldots
		\big[\qHahnA{q}{\nu}{\alpha}\big]_{N-1}
		\EE\bigg[
		q^{(x_1(t+1)+1)y_1}
		\EE\Big[q^{(x_2(t+1)+1)y_2}
		\ldots
		\EE\Big[
		q^{(x_{N-1}(t+1)+N-1)y_{N-1}}
		\\&\hspace{20pt}
		\sum_{s_N=0}^{y_N}
		\varphi_{-\alpha}(s_N|y_N)
		\EE\Big[
		q^{(x_{N}(t+1)+N)(y_{N}-s_N)}
		q^{(x_{N-1}(t+1)+N-1)s_{N}}
		\,\big\vert\, \sigma^{N}_{t},\sigma^{N-1}_{t+1}
		\Big]
		\,\big\vert\, \sigma^{N}_{t},\sigma^{N-2}_{t+1}
		\Big]\ldots
		\,\big\vert\, \sigma^{N}_{t},\sigma^{1}_{t+1}
		\Big]
		\,\big\vert\, \sigma^{N}_{t}
		\bigg]
		\\&=
		\big[\qHahnA{q}{\nu}{\alpha}\big]_{1}\ldots
		\big[\qHahnA{q}{\nu}{\alpha}\big]_{N-1}
		\EE\bigg[
		q^{(x_1(t+1)+1)y_1}
		\EE\Big[q^{(x_2(t+1)+1)y_2}
		\ldots\\&\hspace{20pt}\ldots
		\EE\Big[
		q^{(x_{N-1}(t+1)+N-1)y_{N-1}}
		\,\big\vert\, \sigma^{N}_{t},\sigma^{N-2}_{t+1}
		\Big]\ldots
		\,\big\vert\, \sigma^{N}_{t},\sigma^{1}_{t+1}
		\Big]
		\,\big\vert\, \sigma^{N}_{t}
		\bigg]
		\\&\hspace{20pt}\times\sum_{s_N=0}^{y_N}
		\varphi_{-q\alpha}(s_N|y_N)
		q^{(x_{N}(t)+N)(y_{N}-s_N)}
		q^{(x_{N-1}(t)+N-1)s_{N}}.
	\end{align*}
	The first equality is by definition, and the second equality is by an application of \eqref{Bernoulli_duality_proof2}
	corresponding to $i=N$,
	which leads to replacement of the operator
	$\big[\qHahnA{q}{\nu}{\alpha}\big]_N$
	(for time $t+1$) by the operator
	$\big[\qHahnA{q}{\nu}{q\alpha}\big]_N$
	(for time $t$).
	Continuing using
	\eqref{Bernoulli_duality_proof2} for $i=N-1,\ldots,2$
	and \eqref{Bernoulli_duality_proof3} for $i=1$,
	we arrive at the desired identity \eqref{Bernoulli_duality_proof1},
	and hence complete the proof of the theorem.
\end{proof}

We turn now to self-dualities of the zero range process. The proofs are considerably less involved and rely on an earlier discovered identity \cite[Proposition 1.2]{Corwin2014qmunu}.

\begin{theorem}\label{thm:Gduality}
The $J=1$ higher spin zero range process with transition operator $\highspinBoson{\alpha}{q\alpha}$ is dual to the space reversed process with generator $\revhighspinBoson{\alpha}{q\alpha}$ with respect to $\GG(\vec{g},\vec{y})$ as well as $\GGhat(\vec{g},\vec{y})$. Precisely,
$$
\highspinBoson{\alpha}{q\alpha} \GG = \GG \big(\revhighspinBoson{\alpha}{q\alpha}\big)^{T},
$$
where the equality holds for all matrix elements indexed by $\vec{g}\in \Gspace{k}$ and $\vec{y}\in \Yinf$, or by $\vec{g}\in \Ginf$ and $\vec{y}\in \Yspace{k}$.
Likewise,
$$
\highspinBoson{\alpha}{q\alpha} \GGhat = \GGhat \big(\revhighspinBoson{\alpha}{q\alpha}\big)^{T},
$$
where the equality holds for all matrix elements indexed by $\vec{g}\in \Ginf$ and $\vec{y}\in \Yinf$, provided both sides of the above equation are finite.
\end{theorem}

\begin{proof}
We prove this theorem in a few stages. Initially we deduce an `implicit' form of the $\GG$ duality in \eqref{eq:Gduality}. From this we deduce the $\GG$ duality claimed in the theorem. Finally, we use the $\GG$ duality to deduce the $\GGhat$ duality.

We begin by proving the following implicit form of the $\GG$ duality. For any $k\in \Z_{\geq 1}$,
	\begin{align}\label{eq:Gduality}
		\qHahnHR{q}{\nu}{q\alpha}
		\GG\big(\qHahnH{q}{\nu}{\alpha}\big)^{T}
		=
		\qHahnHR{q}{\nu}{\alpha}
		\GG\big(\qHahnH{q}{\nu}{q\alpha}\big)^{T},
	\end{align}
where the equality holds for all matrix elements with $\vec{g}\in \Gspace{k}$ and $\vec{y}\in \Yinf$ well-adapted to $\GG$, or with $\vec{g}\in \Ginf$ well-adapted to $\GG$ and $\vec{y}\in \Yspace{k}$. We will assume the second case below, though the first case follows similarly.
%

To show (\ref{eq:Gduality}), we demonstrate  first that 	
\begin{align}\label{qHahn_self_duality}
		\GG\big(\qHahnH{q}{\nu}{\mu}\big)^{T}
		=
		\qHahnHR{q}{\nu}{\mu}
		\GG
\end{align}
for all $\vec{g}\in \Ginf$ and $\vec{y}\in \Yinf$. This follows from the fact that for each $i\in \Z$,
	\begin{align*}
		\GG\big( [\qHahnH{q}{\nu}{\mu}]_{i}\big)^T (\vec{g},\vec{y}) &=
		\sum_{s_i=0}^{y_i}
		\varphi_{\mu}(s_i|y_i)
		\GG(\vec{g},\vec{y}^{s_i}_{i,i-1})
		=
		\sum_{s_i=0}^{y_i}
		\varphi_{\mu}(s_i|y_i)
		q^{s_ig_i}\GG(\vec{g},\vec{y})
		\\&=
		\sum_{r_i=0}^{g_i}
		\varphi_{\mu}(r_i|g_i)
		q^{r_iy_i}\GG(\vec{g},\vec{y})
		=
		\sum_{r_i=0}^{g_i}
		\varphi_{\mu}(r_i|g_i)
		\GG(\vec{g}^{r_i}_{i,i+1},\vec{y})
		\\&=
		[\qHahnHR{q}{\nu}{\mu}]_{i}
		\GG(\vec{g},\vec{y}).
	\end{align*}
In the first and last equation above we have written the composition of operators followed by $(\vec{g},\vec{y})$ to denote the corresponding matrix element. The equality between the end of the first and beginning of the second line above relies on an identity proved in \cite[Proposition 1.2]{Corwin2014qmunu}.

From \eqref{qHahn_self_duality}, we find that
\begin{align*}
		\qHahnHR{q}{\nu}{q\alpha}
		\GG\big(\qHahnH{q}{\nu}{\alpha}\big)^{T}
		=
        \GG\big (\qHahnH{q}{\nu}{q\alpha})^T(\qHahnH{q}{\nu}{\alpha}\big)^{T}
		=
        \GG\big (\qHahnH{q}{\nu}{\alpha})^T(\qHahnH{q}{\nu}{q\alpha}\big)^{T}
        =
		\qHahnHR{q}{\nu}{\alpha}
		\GG\big(\qHahnH{q}{\nu}{q\alpha}\big)^{T}.
\end{align*}
The only step in this deduction which requires justification is the commutation relation $(\qHahnH{q}{\nu}{q\alpha})^T(\qHahnH{q}{\nu}{\alpha}\big)^{T} = (\qHahnH{q}{\nu}{\alpha})^T(\qHahnH{q}{\nu}{q\alpha}\big)^{T}$. This, however, follows from Corollary \ref{cor:secrat} and the fact that we have assumed that $\vec{g}\in \Ginf$ is well-adapted to $\GG$ and $\vec{y}\in \Ynspace{k}$. Thus, we have established \eqref{eq:Gduality}.

Under the assumption that $\vec{g}\in \Ginf$ is well-adapted to $\GG$ and $\vec{y}\in \Yspace{k}$ we may apply the first two identities from Corollary \ref{cor:secrat} to equation \eqref{eq:Gduality} to deduce the $\GG$ duality statement in the theorem. The theorem asks for this duality to hold for all $\vec{g}\in \Ginf$ (without the well-adapted condition). Indeed, all $\vec{g}\in \Ginf$ with $\sum_{i\in \Z} g_i = \infty$ are well-adapted to $\GG$. Thus, it remains to show that we can extend the duality to $\vec{g}\in \Gspace{k'}$ for any $k'\in \Z_{\geq 1}$.
For $\vec{g}\in \Gspace{k'}$ fixed, and $M$ sufficiently negative (so as to be less than the location of the most negative particle in $\vec{g}$) let $\vec{g}^M = +\infty \mathbf{1}_{M} + \vec{g}$, where $\mathbf{1}_{M}$ is the vector of all zeros, except a one at $M$. In other words, $\vec{g}^M$ is equivalent to $\vec{g}$ except with an infinite number of particles added at site $M$. From the above argument we know that the $\GG$ duality in the theorem holds for matrix elements $\vec{g}^M\in \Ginf$ and any $\vec{y}\in\Yspace{k}$. It remains to show that both
\begin{align*}
\lim_{M\to -\infty} \highspinBoson{\alpha}{q\alpha} \GG(\vec{g}^M, \vec{y}) &=\highspinBoson{\alpha}{q\alpha} \GG(\vec{g}, \vec{y}), \\
\lim_{M\to -\infty} \GG \big(\revhighspinBoson{\alpha}{q\alpha}\big)^{T}(\vec{g}^M, \vec{y}) &= \GG \big(\revhighspinBoson{\alpha}{q\alpha}\big)^{T}(\vec{g}, \vec{y}).
\end{align*}
Let us justify the first limit, as the second follows similarly. Call $f(\vec{g}^M) = \GG(\vec{g}^M,\vec{y})$ and recall that $\big(\highspinBoson{\alpha}{q\alpha}f\big)(\vec{g}^M)$ gives the expectation of $f$ after one step of the zero range process started from initial data $\vec{g}^{M}$. The only way that the infinite number of particles at $M$ can affect the value of $f$ is if one of them makes its way past the left-most particle in $\vec{y}$. However, this requires a large number of $(0,1;0,1)$-vertices. Since the weight of these vertices is strictly less than one, this probability goes to zero as $M\to -\infty$. Since the above defined $f$ is bounded by one, the desired convergence result clearly holds. This establishes the $\GG$ duality in the theorem.

Turning to the $\GGhat$ duality, let us first prove it for $\vec{g}\in \Gspace{k}$ and $\vec{y}\in \Yspace{k'}$ for some $k,k'\in \Z_{\geq 1}$. Recalling (\ref{eq:kkprime}) we have that
$\GGhat(\vec{g},\vec{y}) = q^{-kk'} \GG(\vec{g},\vec{y})$. Since $q^{-kk'}$ is a constant, multiplying the $\GG$ duality by it yields the $\GGhat$ duality. We now extend to all $\vec{g}\in\Ginf$ and $\vec{y}\in \Yinf$ such that both sides of the $\GGhat$ duality are finite. Under these conditions there must be a finite number (say $k$) of particles in $\vec{g}$ which lie to the left of some particle in $\vec{y}$, and likewise a finite number (say $k'$) of particles in $\vec{y}$ which lie to the right of some particle in $\vec{g}$. It is easy to see\footnote{The update for $\vec{g}$ according to $ \highspinBoson{\alpha}{q\alpha}$ is from left to right, and opposite for $\vec{y}$ according to $\revhighspinBoson{\alpha}{q\alpha}$. The rightward movement of the particles in $\vec{g}$ besides the $k$ left-most does not change the value of $\GG$, and likewise for the leftward movement of the particles in $\vec{y}$ besides the $k'$ right-most.} that replacing $\vec{g}$ with its $k$ left-most particles and replacing $\vec{y}$ with its $k'$ right-most particles, the value of the left-hand and right-hand sides of the $\GGhat$ duality identity are unchanged. This completes the proof of the duality and hence the theorem.
\end{proof}

\begin{rem}
The duality involving $\HH$ and $\GGhat$ can be made to look rather similar, though, to our understanding, they are not equivalent. From a state $\vec{x}\in \Xinf$ define its gaps via $\tilde{g}_i = x_{i-1}-x_i-1$. Then $x_i+i = -g_i-\cdots - g_{2} +x_1 +1$. Thus, up to this last term $x_1+1$ the duality functional $\HH(\vec{x},\vec{y})$ can be written in a similar form as $\GGhat$. We are not aware of a way to derive, for instance, the $\HH$ duality from the $\GGhat$ duality (the proof of the latter is considerably simpler).
\end{rem}

\begin{rem}[2019 update]
	The previous published version of the paper contained as Theorem 2.23 
	a number of additional incorrect duality claims
	involving more complicated duality functionals. This theorem and related claims in 
	Sections \ref{sec:case3}, \ref{sec:case4}, and \ref{sec:ssvm} were removed. 
	See \cite{CP_erratum} for an erratum which contains a counterexample and summarizes the changes.
	The correct $I=J=1$ version of such duality (for the stochastic six vertex model)
	was established recently in \cite{Lin2019}.
\end{rem}

\section{Fusion and self duality for \texorpdfstring{$J\in \Z_{\geq 1}$}{J>=1}}\label{sec:fus}

Fusion of $R$-matrices is a representation theoretic mechanism introduced in \cite{KirillovReshetikhin1987Fusion} to construct $R$-matrices with higher horizontal spin (i.e. $J\in \Z_{\geq 1}$) from those with $J=1$ while maintaining the diagonalizability of the associated transfer matrices. The procedure simplifies on $\Z$ in our case of stochastic $\Lmat$-matrices. We provide (in Section \ref{sec:fusionMFT}) a rather simple probabilistic proof using Markov functions theory. Fusion also naturally provides a recursion relation in $J$ for the higher spin $\Lmat$-matrix $\bernw{J}{q}{\nu}{\alpha}$. We record the recursion relation in Section \ref{sec:RR} and solve it explicitly in terms of $q$-Racah polynomials (or regularized terminating basic hypergeometric series) in Section \ref{sec:qRacah}. The self-dualities proved earlier in Section \ref{sec:implicitduality} immediately generalize to all $J\in \Z_{\geq 1}$. In Remark \ref{sec:analcont} we observe how our $\Lmat$-matrix can be analytically continued  so that $ q^J\alpha$ is replaced by an arbitrary $\beta\in \C$. We comment briefly on the implications, and develop this further in Section \ref{sec:spec}.

\begin{definition}[General $J$ higher spin zero range and exclusion process transition operators]\label{def:highJBosonTASEP}
For $J\in \Z_{\geq 1}$ define\footnote{The order does not matter since each operator is diagonalized in the same basis.}
$$
\highspinBoson{\alpha}{q^J\alpha} = \highspinBoson{\alpha}{q\alpha} \highspinBoson{q\alpha}{q^2\alpha} \cdots \highspinBoson{q^{J-1}\alpha}{q^J\alpha},
$$
and likewise $\revhighspinBoson{\alpha}{q^J\alpha}$. Also define
$$
\highspinTASEP{\alpha}{q^J\alpha} = \highspinTASEP{\alpha}{q\alpha} \highspinTASEP{q\alpha}{q^2\alpha} \cdots \highspinTASEP{q^{J-1}\alpha}{q^J\alpha}.
$$
These correspond to taking $J$ steps of the processes with parameters $\alpha, q\alpha,\ldots, q^{J-1}\alpha$.
\end{definition}

The left-hand side of Figure \ref{fig:BRfix} illustrates the sequential composition of the $ \highspinBoson{q^{j-1}\alpha}{q^j\alpha}$ operators for $j=1,\ldots, J$. Since each transition operator is stochastic, their product is as well. Moreover, it follows from Proposition \ref{prophighspindiagJ1} that:

\begin{corollary}\label{cor:higherJeig}
If $\big\vert \tfrac{1-z_i}{1-\nu z_i}\, \tfrac{q^j\alpha+\nu}{1+q^j\alpha}\big\vert<1$
for $1\leq i\leq k$ and $1\leq j\leq J-1$, then
\begin{equation}\label{eqn:higherJeig}
\big(\revhighspinBoson{\alpha}{q^J\alpha} \Psil_{\vec{z}}\big)(\vec{n}) = \prod_{i=1}^{k} \frac{1+q^J\alpha z_i}{1+ \alpha z_i} \, \Psil_{\vec{z}}(\vec{n}).
\end{equation}
\end{corollary}

The of choice parameters $\alpha, q\alpha,\ldots, q^{J-1}\alpha$ implies the telescoping of the product of eigenvalues and hence the simple form of the result.

The following is an immediate corollary of the corresponding $J=1$ duality results contained in Theorems \ref{thmimplicit} and \ref{thm:Gduality}, along with the fact that the operators
$\highspinBoson{\alpha}{q\alpha}$ commute for different values of $\alpha$ (and the same fact for $\revhighspinBoson{\alpha}{q\alpha}$ and $\highspinTASEP{\alpha}{q\alpha}$). Recall the duality functionals from Definition \ref{def:dualityfunctionals}.

\begin{corollary}\label{cor:dualhigherJ}
For all $J\in \Z_{\geq 1}$ we have the following Markov dualities:
\begin{itemize}
\item For all matrix elements indexed by $\vec{x}\in \Xinf$ well-adapted to $\HH$ and all $\vec{y}\in \Yspace{k}$,
$$
\highspinTASEP{\alpha}{q^J\alpha} \HH = \HH \big(\revhighspinBoson{\alpha}{q^J\alpha}\big)^{T}.
$$
\item For all matrix elements indexed by $\vec{g}\in \Gspace{k}$ and $\vec{y}\in \Yinf$, or by $\vec{g}\in \Ginf$ and $\vec{y}\in \Yspace{k}$,
$$
\highspinBoson{\alpha}{q^J\alpha} \GG = \GG \big(\revhighspinBoson{\alpha}{q^J\alpha}\big)^{T}.
$$
\item For all matrix elements indexed by $\vec{g}\in \Ginf$ and $\vec{y}\in \Yinf$, provided both sides of the equation below are finite,
$$
\highspinBoson{\alpha}{q^J\alpha} \GGhat = \GGhat \big(\revhighspinBoson{\alpha}{q^J\alpha}\big)^{T}.
$$
\end{itemize}
\end{corollary}

\subsection{Fusion}\label{sec:fusionMFT}

\begin{figure}
\begin{center}
\includegraphics[scale=.6]{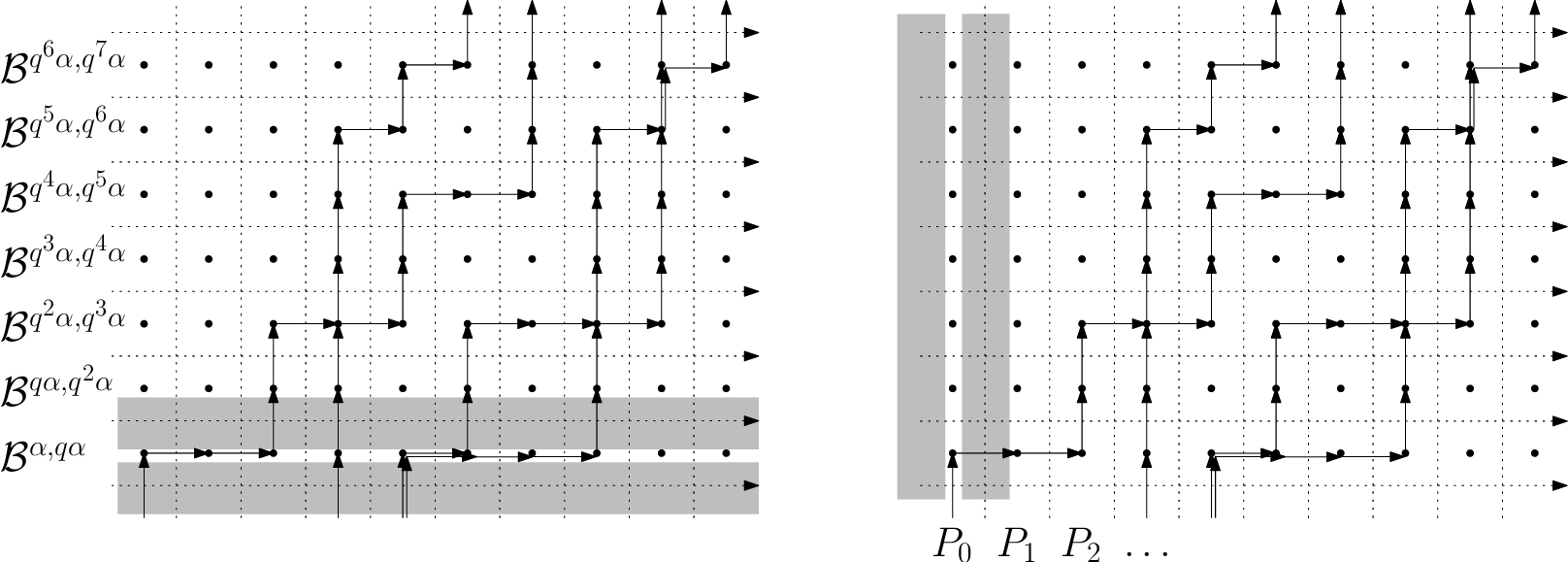}
\end{center}
\caption{Left: The trajectory (in terms of arrows) of the application of the Markov transition operator $\highspinBoson{\alpha}{q\alpha}$,  $\highspinBoson{q\alpha}{q^2\alpha}$, through $\highspinBoson{q^{J-1}\alpha}{q^J\alpha}$ (here $J=7$). The bottom grey row represents the input to $\highspinBoson{\alpha}{q\alpha}$ and the higher grey row the output. Right: Due to the sequential update rule for each $\highspinBoson{q^{j}\alpha}{q^{j+1}\alpha}$, instead of sampling row by row, one can also sample this trajectory sequentially by updating one column (such as indicated in grey) at a time from left to right.}
\label{fig:BRfix}
\end{figure}

%

We can describe the update procedure for the Markov chain with the transition operator $\highspinBoson{\alpha}{q^J\alpha}$ started in a state $\vec{g}\in \Gspace{k}$ in the following manner.
\begin{definition}
For $x\in \Z$, define the product of $\Lmat$-matrices
$$
\big[\bernwtensor{1}{q}{\nu}{\alpha}{J}\big]_{x,1} = [\bernw{1}{q}{\nu}{\alpha}]_{x,1}\cdots [\bernw{1}{q}{\nu}{q^{J-1}\alpha}]_{x,J}
$$
so that
$$
\big[\bernwtensor{1}{q}{\nu}{\alpha}{J}\big]_{x,1}: [\VI^I]_x \otimes [\VJ^1]_1\otimes\cdots \otimes [\VJ^1]_J  \to [\VI^I]_x \otimes [\VJ^1]_1\otimes\cdots \otimes [\VJ^1]_J
$$
has matrix elements
$$
\big[\bernwtensor{1}{q}{\nu}{\alpha}{J}\big]_{x,1}\big(g_x, h_{x,1}, \cdots, h_{x,J}; g_x', h_{x+1,1},\cdots, h_{x+1,J}\big)
$$
with $g_x,g_x'\in [\VI^I]_x$ and $h_{x,y}\in [\VJ^1]_y$ for $1\leq y\leq J$. These matrix elements represent the transition probabilities from inputs $g_x,h_{x,1},\ldots, h_{x,J}$ to outputs $g'_x,h_{x+1,1},\ldots, h_{x+1,J}$. In terms of the right-hand side of Figure \ref{fig:BRfix}, these provide the transition probabilities from the arrows coming into a column (such as the one in grey) from bottom and left, to those leaving to the top and right.
\end{definition}

As in Definition \ref{deftransprobJ1ZRP}, we use $\bernwtensor{1}{q}{\nu}{\alpha}{J}$ to update sequentially in the following manner. Find the first $x\in \Z$ such that $g_x>0$. In the first sequential update step, let $h_{x,y}\equiv 0$ for $1\leq y\leq J$ and randomly choose $g'_x$ and $h_{x+1,y}$ for $1\leq y\leq J$ according to the stochastic matrix $\bernwtensor{1}{q}{\nu}{\alpha}{J}$. The randomly chosen $h_{x+1,y}$ become input, along with $g_{x+1}$ for the next column update step, and so on sequentially increasing $x$. It is clear that the above described update from $\vec{g}$ to $\vec{g}'$ agrees with $\highspinBoson{\alpha}{q^J\alpha}$.

This update procedure can be recast in the following notation (see the right-hand side of Figure \ref{fig:BRfix}). First, for $x\in \Z$ define $S$ to be the set of basis elements of $[\VJ^1]_1\otimes\cdots \otimes [\VJ^1]_J$ and $S'$ to be the set of basis elements of $[\VJ^J]_1$. In other words, $S$ is isomorphic to $\{0,1\}^{J}$ and $S'$ to $\{0,1,\ldots, J\}$. For all $x\in \Z$, define a Markov transition operator $P_x:S \to S$  whose matrix elements indexed by $(h_{x,1},\ldots, h_{x,J}) \in S$ and $(h_{x+1,1},\ldots, h_{x+1,J}) \in S$ are given by
\begin{equation}\label{eq.Px}
P_x \big( h_{x,1}, \cdots, h_{x,J}; h_{x+1,1},\cdots, h_{x+1,J}  \big) =\big[\bernwtensor{1}{q}{\nu}{\alpha}{J}\big]_{x,1}\big(g_x, h_{x,1}, \cdots, h_{x,J}; g_x', h_{x+1,1},\cdots, h_{x+1,J}\big)
\end{equation}
where we consider $\vec{g}\in\Gspace{k}$ fixed, and have set
\begin{equation}\label{ghx}
g_x' = g_x + h_x - h_{x+1}, \qquad \textrm{with}\qquad h_x = \sum_{y=1}^J h_{x,y}.
\end{equation}
As a corollary of the above definitions we have
\begin{corollary}\label{cor:Px}
Fix $\vec{g}\in \Gspace{k}$ and assume without loss of generality that $x=0$ is the smallest $x\in \Z$ such that $g_x>0$. Consider the Markov chain with state space $S$, transition operator $P_x$, and initial state $h_{0,y}\equiv 0$ for $1\leq y\leq J$. Denote the value at `time' $x$ as $\{h_{x,y}\}_{1\leq y\leq J}$.  Then, from the trajectory of this Markov chain we can deterministically compute $\vec{g}'$ via (\ref{ghx}) and the probability of having an output $\vec{g}'$ given $\vec{g}$ is equal to $\highspinBoson{\alpha}{q^J\alpha}(\vec{g},\vec{g}')$.
\end{corollary}

Define the function $\phi: S \to S'$ which takes $(h_{x,1},\ldots, h_{x,J}) \in S$ to $h_x = \sum_{y=1}^{J}h_{x,y}\in S'$. Define the operator $\Phi: S \to S'$ which acts on functions as $(\Phi f)\big(h_{x,1}, \cdots, h_{x,J}\big) = f\Big(\phi\big(h_{x,1}, \cdots, h_{x,J}\big)\Big)$. In other words, the matrix elements of $\Phi$, indexed by $(h_{x,1},\ldots, h_{x,J}) \in S$ and $h_x\in S'$, are given by
\begin{equation}\label{eq.phi}
\Phi(h_{x,1}, \cdots, h_{x,J};h_x) = \mathbf{1}_{\phi(h_{x,1},\ldots, h_{x,J}) = h_x}.
\end{equation}

Define an operator $\Lambda : S'\to S$ whose matrix elements indexed by $h_x\in S'$ and $(h_{x,1},\ldots, h_{x,J}) \in S$ are given by
\begin{equation}\label{eq.Lambda}
\Lambda\Big(h_x; \big(h_{x,1}, \ldots, h_{x,J}\big)  \Big) = Z^{-1}\,  \mathbf{1}_{\phi(h_{x,1},\ldots, h_{x,J})= h_x} \, \cdot\, \prod_{y\colon h_{x,y}=1} q^y,
\end{equation}
where $Z$ equals the sum of the weights $\mathbf{1}_{\phi(h_{x,1},\ldots, h_{x,J}) = h_x} \, \cdot\, \prod_{y:h_{x,y}=1} q^y$ over all $(h_{x,1},\ldots, h_{x,J}) \in S$. Observe that  $\Lambda$ does not depend on $x$, and, moreover, it is a Markov transition operator, meaning that for each $h_x\in S'$, $\Lambda(h_x; \cdot)$ is a probability measure in the second slot (over the set $S$), and for each $ \big(h_{x,1}, \ldots, h_{x,J}\big)\in S$, $\Lambda\Big(\cdot;  \big(h_{x,1}, \ldots, h_{x,J}\big)\Big)$ is bounded and measurable in the first slot (over the set $S'$).

\begin{proposition}\label{prop:twoprops}
The following two identities hold:
\begin{enumerate}
\item $\Lambda \Phi = I$, the identity operator on $S'$,
\item For each $x\in \Z$ the Markov operator
\begin{equation}\label{eq.Qx}
Q_x = \Lambda P_x \Phi
\end{equation}
from $S'$ to $S'$ satisfies $\Lambda P_x = Q_x \Lambda$.
\end{enumerate}
\end{proposition}
\begin{proof}
The first identity amounts to the claim that for any $h_x$, the probability measure $\Lambda(h_x;\cdot)$ is supported entirely upon $\phi^{-1}(h_x)$, the pre-image of $h_x$ under $\phi$. This, however, is immediate from the definition of $\Lambda$.

The second identity relies on more involved properties of the $\Lmat$-matrix weights. These are recorded in the following lemma (see also Figure \ref{fig:YBEresults} for a pictorial representation of the three identities of the lemma).
\begin{lemma}\label{lem.threeids}
For all $\alpha$ and all $m\in \Z_{\geq 0}$
\begin{align}
&q \bernw{1}{q}{\nu}{\alpha}(m,0;m-1,1) \bernw{1}{q}{\nu}{q \alpha}(m-1,0;m-1,0) = \bernw{1}{q}{\nu}{\alpha}(m,0;m,0) \bernw{1}{q}{\nu}{q \alpha}(m,0;m-1,1), \label{eq.lem1}\\
&q \bernw{1}{q}{\nu}{\alpha}(m,1;m,1) \bernw{1}{q}{\nu}{q \alpha}(m,1;m+1,0) = \bernw{1}{q}{\nu}{\alpha}(m,1;m+1,0) \bernw{1}{q}{\nu}{q \alpha}(m+1,1;m+1,1), \label{eq.lem2}\\
\nonumber&q \Big( \bernw{1}{q}{\nu}{\alpha}(m,1;m,1) \bernw{1}{q}{\nu}{q \alpha}(m,0;m,0) + q  \bernw{1}{q}{\nu}{\alpha}(m,0;m-1,1) \bernw{1}{q}{\nu}{q \alpha}(m-1,1;m,0)\Big)  \\
&=
 \bernw{1}{q}{\nu}{\alpha}(m,1;m+1,0) \bernw{1}{q}{\nu}{q \alpha}(m+1,0;m,1) + q  \bernw{1}{q}{\nu}{\alpha}(m,0;m,0) \bernw{1}{q}{\nu}{q \alpha}(m,1;m,1). \label{eq.lem3}
\end{align}
\end{lemma}
\begin{proof}
Each identity is readily checked by direct calculation.
\end{proof}

\begin{figure}
\begin{center}
\includegraphics[scale=.75]{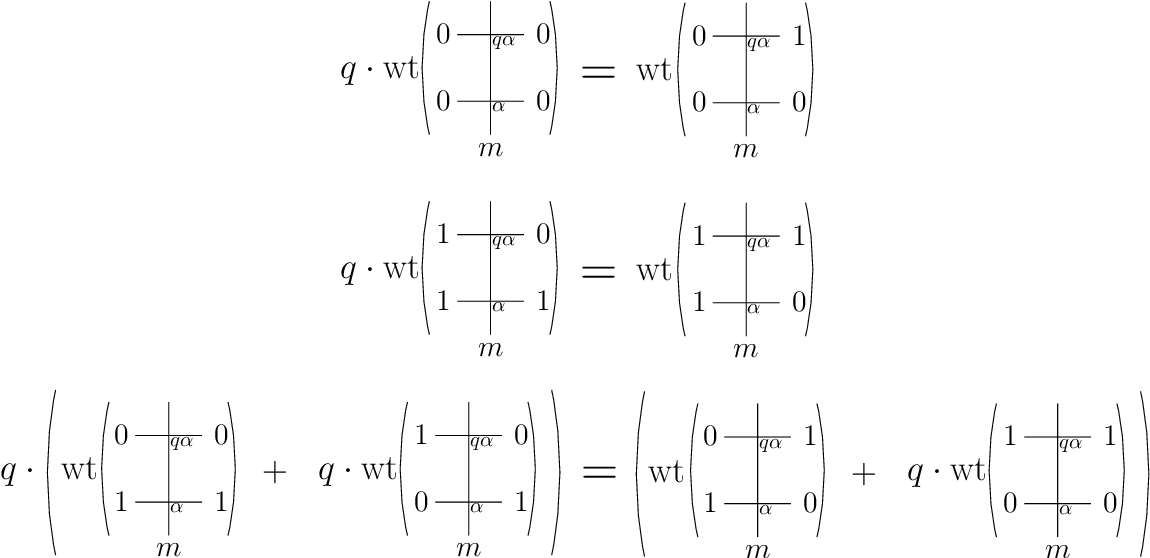}
\end{center}
\caption{Pictorial representation of the three identities in Lemma \ref{lem.threeids}. The weight of a diagram is given by the product of $\Lmat$-matrices for each vertex. The value of the $\alpha$ parameter is specified as either $\alpha$ or $q\alpha$ depending on the height of the vertex.}
\label{fig:YBEresults}
\end{figure}

We will use the lemma to show that for any $j\in \{1,\ldots, J-1\}$,
\begin{align}\label{eq.claiminverse}
&\Lambda P_x \big(h_x ; h_{x+1,1},\ldots,h_{x+1,j},h_{x+1,j+1},\ldots h_{x+1,J}\big) \\
\nonumber &= q^{h_{x+1,j+1}-h_{x+1,j}} \Lambda P_x \big(h_x ;h_{x+1,1},\ldots,h_{x+1,j+1},h_{x+1,j},\ldots h_{x+1,J}\big).
\end{align}
Notice that the terms $h_{x+1,j}$ and $h_{x+1,j+1}$ have been switched between the left-hand and right-hand sides.

Lemma \ref{lem.threeids} readily implies (\ref{eq.claiminverse}). To see this, first write
$$
\mathrm{LHS}(\ref{eq.claiminverse}) = \sum_{h_{x,1},\ldots, h_{x,J}\in \phi^{-1}(h_x)} \Lambda\big(h_x;h_{x,1},\ldots, h_{x,J}\big) P_x\big(h_{x,1},\ldots, h_{x,J};h_{x+1,1},\ldots, h_{x+1,J}\big).
$$
Let $(j,j+1)$ denote the permutation which transposes $j$ and $j+1$, and $\mathrm{Id}$ denote the identity permutation. Then it follows by double counting that we can continue the line of equalities as
$$
= \tfrac{1}{2}\!\! \sum_{h_{x,1},\ldots, h_{x,J}\in \phi^{-1}(h_x)} \sum_{\sigma\in \big\{\mathrm{Id},(j,j+1)\big\}} \Lambda\big(h_x;h_{x,\sigma(1)},\ldots, h_{x,\sigma(J)}\big) P_x\big(h_{x,\sigma(1)},\ldots, h_{x,\sigma(J)};h_{x+1,1},\ldots, h_{x+1,J}\big).
$$
We may now expand the definition of $P_x$ into the product of $\Lmat$-matrices and factor out all terms unaffected by the permutation $\sigma$ (i.e., terms not involving level $j$ or $j+1$). To facilitate this expansion, let us temporarily introduce the notation that for any permutation $\sigma$, $g^{\sigma}_{x,1}=g_x$ (which is given) and $g^{\sigma}_{x,y+1} = g^{\sigma}_{x,y} + h_{x,\sigma(y)} - h_{x+1,y}$. Note that for $\sigma = (j,j+1$), the only value of $g^{\sigma}(x,y)$ which may differ from the case $\sigma=\mathrm{Id}$ is for $y=j+1$. Then, we can continue the above line of equalities as
\begin{align*}
&= \tfrac{1}{2}\!\!\!\!\!\! \sum_{h_{x,1},\ldots, h_{x,J}\in \phi^{-1}(h_x)}  \Lambda\big(h_x;h_{x,1},\ldots, h_{x,J}\big) \prod_{\substack{y=1\\y\neq j,j+1}}^{J} \bernw{1}{q}{\nu}{\alpha}(g_{x,y},h_{x,y};g_{x,y+1},h_{x+1,y}) \\
&\quad \times\!\!\!\! \sum_{\sigma\in \big\{\mathrm{Id},(j,j+1)\big\}} q^{h_{x,\sigma(j+1)}-h_{x,j+1}} \bernw{1}{q}{\nu}{\alpha}(g_{x,j},h_{x,\sigma(j)};g^{\sigma}_{x,j+1},h_{x+1,j})
\bernw{1}{q}{\nu}{\alpha}(g^{\sigma}_{x,j+1},h_{x,\sigma(j+1)};g_{x,j+2},h_{x+1,j+1}).
\end{align*}
We have used here the fact that
$\Lambda\big(h_x;h_{x,\sigma(1)},\ldots, h_{x,\sigma(J)}\big) =  q^{h_{x,\sigma(j+1)}-h_{x,j+1}}   \Lambda\big(h_x;h_{x,1},\ldots, h_{x,J}\big)$.

Finally, observe that by applying Lemma \ref{lem.threeids} we have
\begin{align*}
& \sum_{\sigma\in \big\{\mathrm{Id},(j,j+1)\big\}} q^{h_{x,\sigma(j+1)}-h_{x,j+1}} \bernw{1}{q}{\nu}{\alpha}(g_{x,j},h_{x,\sigma(j)};g^{\sigma}_{x,j+1},h_{x+1,j})
\bernw{1}{q}{\nu}{\alpha}(g^{\sigma}_{x,j+1},h_{x,\sigma(j+1)};g_{x,j+2},h_{x+1,j+1})\\
&=
q^{h_{x+1,j+1}-h_{x+1,j}} \!\!\!\! \!\!\!\!\sum_{\sigma\in \big\{\mathrm{Id},(j,j+1)\big\}} \!\!\!\!\!\!\!\! q^{h_{x,\sigma(j+1)}-h_{x,j+1}} \bernw{1}{q}{\nu}{\alpha}(g_{x,j},h_{x,\sigma(j)};g^{\sigma}_{x,j+1},h_{x+1,j+1})
\bernw{1}{q}{\nu}{\alpha}(g^{\sigma}_{x,j+1},h_{x,\sigma(j+1)};g_{x,j+2},h_{x+1,j}).
\end{align*}
Notice that besides the factor of  $q^{h_{x+1,j+1}-h_{x+1,j}}$, the change on the right-hand side is that $h_{x+1,j}$ and $h_{x+1,j+1}$ have been switched. Plugging this equality into the above line of equalities, and gathering terms back into their original form we arrive at
$$
q^{h_{x+1,j+1}-h_{x+1,j}} \Lambda P_x \big(h_x; h_{x+1,1},\ldots,h_{x+1,j+1},h_{x+1,j},\ldots h_{x+1,J}\big),
$$
as desired to prove (\ref{eq.claiminverse}).

It remains to use (\ref{eq.claiminverse}) to prove the second identity of Proposition \ref{prop:twoprops}. Observe that
\begin{align}\label{eq.probexpand}
& \Lambda P_x \big(h_x ; h_{x+1,1},\ldots, h_{x+1,J}\big) \\
&= \sum_{h_{x+1}\in S'} \Lambda P_x \big(h_{x}; \phi^{-1}(h_{x+1})\big) \mathbf{1}_{h_{x+1} = \phi(h_{x+1,1},\ldots, h_{x+1,J})} \frac{\Lambda P_x \big(h_x ; h_{x+1,1},\ldots, h_{x+1,J}\big)}{\Lambda P_x \big(h_x; \phi^{-1}(h_{x+1})\big)}.
\end{align}

We claim that
\begin{equation}\label{eq.equaltolambda}
\mathbf{1}_{h_{x+1} = \phi(h_{x+1,1},\ldots, h_{x+1,J})} \frac{\Lambda P_x \big(h_x; h_{x+1,1},\ldots, h_{x+1,J}\big)}{\Lambda P_x \big(h_x; \phi^{-1}(h_{x+1})\big)} = \Lambda\big(h_{x+1};h_{x+1,1},\ldots, h_{x+1,J}\big).
\end{equation}
This follows from two facts. First, given $h_{x+1}$, the left-hand side of (\ref{eq.equaltolambda}) is a probability measure on the set $\phi^{-1}(h_{x+1})\subseteq S$. To state the second fact, let us introduce short-hand that $\mathrm{LHS}(\ref{eq.equaltolambda})^{(j,j+1)}$ is equal to the left-hand side of (\ref{eq.equaltolambda}) when $h_{x+1,j}$ and $h_{x+1,j+1}$ are switched. Likewise define $\mathrm{RHS}(\ref{eq.equaltolambda})^{(j,j+1)}$. It follows from (\ref{eq.claiminverse}) that for any $1\leq j\leq J-1$,
$$
\frac{\mathrm{LHS}(\ref{eq.equaltolambda})}{\mathrm{LHS}(\ref{eq.equaltolambda})^{(j,j+1)}} = \frac{\mathrm{RHS}(\ref{eq.equaltolambda})}{\mathrm{RHS}(\ref{eq.equaltolambda})^{(j,j+1)}}.
$$
In fact, both sides are either equal to $q^{-1}, 1$ or $q$. Thus, combining these two facts (along with the fact that transpositions $(j,j+1)$, $1\leq j\leq J-1$ generate the symmetric group $S(J)$) proves (\ref{eq.equaltolambda}).

To complete the proof of the second identity of Proposition \ref{prop:twoprops} combine (\ref{eq.probexpand}) and (\ref{eq.equaltolambda}) to find that
$$
\Lambda P_x \Big(h_x; \big(h_{x+1,1},\ldots, h_{x+1,J}\big)\Big) = \sum_{h_{x+1}\in S'} \Lambda P_x \big(h_{x}; \phi^{-1}(h_{x+1})\big) \Lambda\Big(h_{x+1};\big(h_{x+1,1},\ldots, h_{x+1,J}\big)\Big).
$$
But $\Lambda P_x \big(h_{x}; \phi^{-1}(h_{x+1})\big) = \Lambda P_x \Phi (h_{x};h_{x+1}) = Q_x (h_x;h_{x+1})$ implying $\Lambda P_x = Q_x \Lambda$.
\end{proof}



\begin{figure}
\begin{center}
\includegraphics[scale=1]{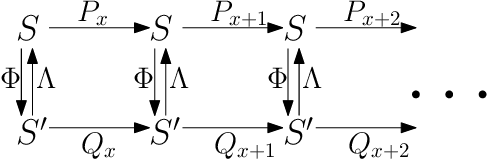}
\end{center}
\caption{The intertwining relation for Markov function theory.}
\label{fig:comdia}
\end{figure}
We make use of the following sufficient condition for when a function of a Markov chain is itself Markov. This comes from \cite{rogerspitman1981}, though since we deal with an inhomogeneous Markov chain, we provide the (essentially unchanged) proof.
\begin{proposition}\label{prop:MFT}
Consider measurable spaces $S$ and $S'$ and a measurable transform $\phi:S\to S'$. Let $\Phi$ be a Markov transition operator from $S\to S'$ defined according to $\Phi f =f \circ \phi$. Consider a collection of $(x\in \Z_{\geq 0})$-indexed Markov transition operators $P_x:S\to S$. Suppose that there exists a Markov transition operator $\Lambda:S'\to S$ such that (see Figure \ref{fig:comdia})
\begin{itemize}
\item $\Lambda \Phi = I$, the identity operator on $S'$,
\item For each $x\in \Z_{\geq 0}$ the Markov operator
$Q_x = \Lambda P_x \Phi$
from $S'$ to $S'$ satisfies $\Lambda P_x = Q_x \Lambda$.
\end{itemize}
Let $X(x)$ be Markov with $x$-indexed Markov transition operators $P_x$ (so that $P_x$ takes one from the $x$ to $x+1$ state) and initial distribution $\Lambda(y,\cdot)$ where $y\in S'$. Then $Y(x)=\phi\big(X(x)\big)$ is Markov with starting state $Y(0)=y$ and $x$-indexed Markov transition operators $Q_x$.
\end{proposition}
\begin{proof}
For Borel functions $f:S\to \R$ and $f':S'\to \R$ the first condition implies that\footnote{Concatenation of operators and functions should be read from right to left, unless indicated by parentheses.} $\Lambda (\Phi f') f = f' \Lambda f$. Using the second condition we find that
$$
\Lambda P_x (\Phi f') f = Q_x \Lambda (\Phi f') g = Q_x f' \Lambda f.
$$
In the same manner, we find that for any $x\in \Z_{\geq 0}$ and Borel functions $f:S\to \R$ and $f'_0,\ldots, f'_x:S'\to \R$,
$$
\Lambda (\Phi f'_0) P_0 (\Phi f'_1) \cdots P_{x-1} (\Phi f'_x) f = f'_0 Q_0 f'_1 Q_1 f'_2 \cdots  Q_{x-1} f'_x \Lambda f.
$$
This immediately implies the conclusion of the theorem.
\end{proof}

We return now to the specific definitions of $S,S',P_x,\phi,\Phi,\Lambda,$ and $Q_x$ given earlier in this section.
\begin{corollary}\label{cor:PxQx}
Initialize the Markov chain with state space $S$ and transition operator $P_x$ to initial state $h_{0,y}\equiv 0$ for $1\leq y\leq J$ and denote the value at `time' $x$ as $\{h_{x,y}\}_{1\leq y\leq J}$. Then $h_x = \phi\big(h_{x,1},\ldots, h_{x,J}\big)$ is Markov with starting state $h_0=0$ and Markov transition operator $Q_x$.
\end{corollary}


\begin{definition}[General $J$ $\Lmat$-matrix]
The $J$ higher spin $\Lmat$-matrix $\bernw{J}{q}{\nu}{\alpha}: \VI^I \otimes \VJ^J \to \VI^I\otimes \VJ^J$ is defined such that (here $x\in \Z$ is an arbitrary and inconsequential index)
\begin{equation}\label{eq:onetoJ}
[\bernw{J}{q}{\nu}{\alpha}]_{x,1}(g_x,h_x;g_x',h_{x+1})= \mathbf{1}_{g_x+h_x=g_x'+h_{x+1}} Q_x(h_x;h_{x+1}),
\end{equation}
where $Q_x$ is defined with respect to $g_x$ (via $P_x$) in (\ref{eq.Qx}).
\end{definition}

\begin{figure}
\begin{center}
\includegraphics[scale=.9]{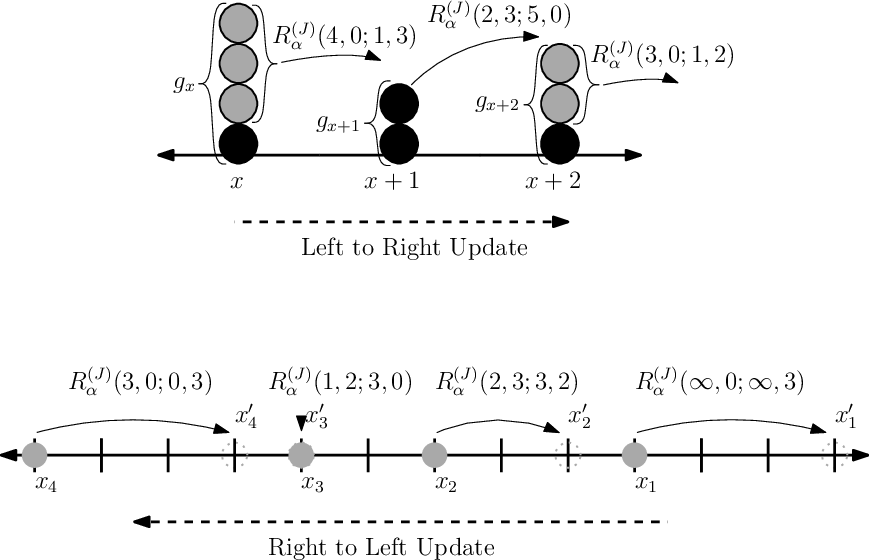}
\end{center}
\caption{Illustration of the $J$ higher spin zero range process (on top) and exclusion process (on bottom). Here $J$ is at least three (as there are clusters / jumps of size at least three. The $\Lmat$-matrix $\bernw{J}{q}{\nu}{\alpha}$ determines the distribution of the sequential updates.}
\label{fig:Jspinprocesses}
\end{figure}

We may now define the higher spin versions of the zero range and exclusion processes introduced in Definitions \ref{deftransprobJ1ZRP} and \ref{deftransprobJ1AEP} respectively. Observe that these depend now on four parameters, $q,\nu,\alpha,J$. As before, we assume that $q,\nu,\alpha$ satisfy the conditions of (\ref{eq:star7}).

\begin{definition}[General $J$ higher spin zero range process and exclusion process]\label{def:highspins}
For $J\in \Z_{\geq 1}$ define the $J$ higher spin zero range process according to the sequential construction of Definition \ref{deftransprobJ1ZRP} with the $\Lmat$-matrix $\bernw{J}{q}{\nu}{\alpha}$ defined in (\ref{eq:onetoJ}). Now, anywhere between $0$ and $J$ particles can move right in each update (i.e., the $h_x$ are basis elements of $\VJ^J$). We still write the corresponding Markov chain at $\vec{g}(t)$. Likewise, define the space-reversed zero range process $\vec{y}(t)$ as in the $J=1$ case. Similarly, define the $J$ higher spin exclusion process as in Definition \ref{deftransprobJ1AEP} with the $\Lmat$-matrix $\bernw{J}{q}{\nu}{\alpha}$. Now, particles can move right by up to and including $J$ spaces. See Figure \ref{fig:Jspinprocesses} for illustrations of these processes.
\end{definition}

\begin{corollary}Recall Definition \ref{def:highspins}.
The transition operator for the $J$ higher spin zero range process is $\highspinBoson{\alpha}{q^J\alpha}$, it space-reversal is $\revhighspinBoson{\alpha}{q^J\alpha}$ , and the exclusion process is $\highspinTASEP{\alpha}{q^J\alpha}$.
\end{corollary}
\begin{proof}
These results follow from Corollaries \ref{cor:Px} and \ref{cor:PxQx}.
\end{proof}

\subsection{Recursion relation for $\Lmat$-matrix weights}\label{sec:RR}
In (\ref{eq:onetoJ}) we have provided a method to calculate $\bernw{J}{q}{\nu}{\alpha}$ in terms of $\bernw{1}{q}{\nu}{\alpha}$. We can also write a compact recursive formula (in $J$) that $\bernw{J}{q}{\nu}{\alpha}$ must satisfy.

In defining $\bernw{J}{q}{\nu}{\alpha}$ we applied $\Lambda$ to $h_x\in [\VJ^J]_x$, so as to yield a probability measure on the pre-image $\phi^{-1}(h_x)$. Let $p^{(J)}(h_{x,1}|h_x)$ denote the projection of this probability measure onto $h_{x,1}$, the first outgoing arrow. In other words, $p^{(J)}(0|h_x)$ and $p^{(J)}(1|h_x)$ (respectively) represent the probability that $h_{x,1}=0$ or $1$ given the value of $h_x$. One calculates that for $j\in \{0,1\}$,
$$
p^{(J)}(j|h_x)  = \sum_{(j=h_{x,1},h_{x,2},\ldots, h_{x,J})\in \phi^{-1}(h_x)} \Lambda\big(h_x; (j=h_{x,1},h_{x,2},\ldots, h_{x,J})\big) =
\begin{cases}
\frac{q^{h_x}-q^J}{1-q^J}, &j=0;\\
\frac{1-q^{h_x}}{1-q^J}, &j=1.
\end{cases}
$$

\begin{lemma}\label{lem:rec}
The $\Lmat$-matrix $\bernw{J}{q}{\nu}{\alpha}$ satisfies the following recursion relation (let $i_2=i_1+j_1-j_2$)
\begin{equation}\label{eq:rec}
\bernw{J}{q}{\nu}{\alpha}(i_1,j_1,i_2,j_2) = \sum_{a,b\in \{0,1\}} p^{(J)}(a|j_1) \bernw{1}{q}{\nu}{\alpha}(i_1,a;i_1+a-b,b) \bernw{J-1}{q}{\nu}{q\alpha}(i_1+a-b,j_1-a;i_2,j_2-b),
\end{equation}
which, along with its value at $J=1$, uniquely characterizes its value for $J\in \Z_{\geq 1}$.
\end{lemma}
\begin{proof}
This follows from the definitions  of $\bernw{J}{q}{\nu}{\alpha}$ and $p^{(J)}(a|j_1)$, and the decomposition in Figure \ref{fig:decomposes}.
\end{proof}

\begin{figure}
\begin{center}
\includegraphics[scale=.5]{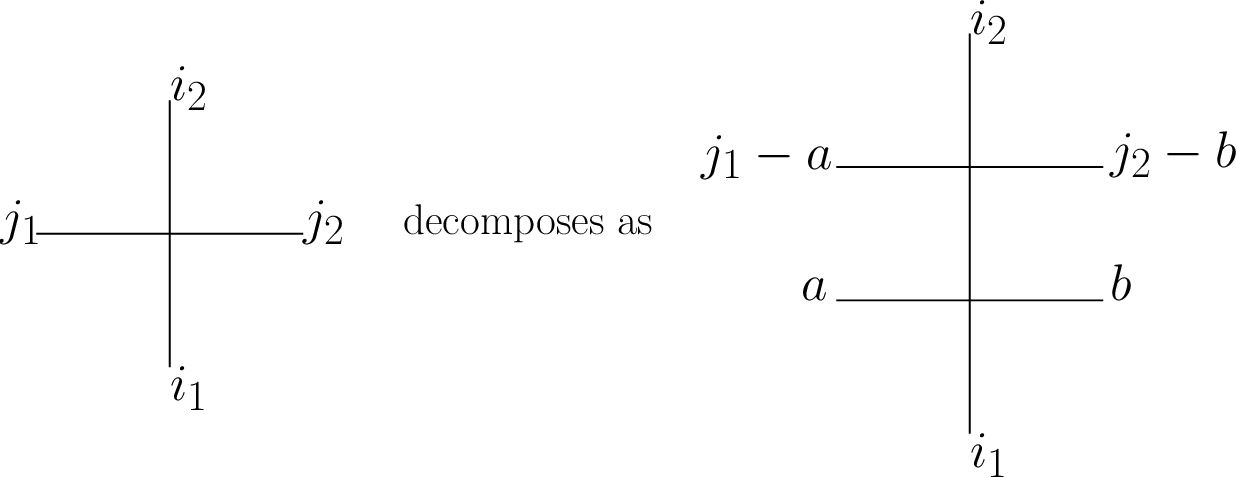}
\end{center}
\caption{The decomposition of a higher horizontal spin vertex.}
\label{fig:decomposes}
\end{figure}

%
%

\subsection{$q$-Racah form of $\Lmat$-matrix weights}\label{sec:qRacah}

%
%
%

We explicitly solve the recurrence relation in Lemma \ref{lem:rec} in a form similar to that which was (to our knowledge) discovered in \cite{Mangazeev2014}. We provide an explicit proof of our $\Lmat$-matrix formula.

\begin{definition}
The regularized terminating basic hypergeometric series is given by\footnote{Pochhammer symbols $(a;q)_{n}$, where $n$ is allowed to be negative are defined as in \texttt{Mathematica} by:
$$
(a;q)_n = \begin{cases} \prod_{k=0}^{n-1} (1-aq^k), & n>0,\\ 1,& n=0,\\ \prod_{k=0}^{-n-1}(1-aq^{n+k})^{-1}, &n<0.\end{cases}
$$}
\begin{multline*}
{}_{r+1}\bar{\phi}_r\left(\begin{matrix} q^{-n};a_1,\dots,a_r\\b_1,\dots,b_r\end{matrix}
\Bigl| q,z\right)=\sum_{k=0}^n z^k\,\frac{(q^{-n};q)_k}{(q;q)_k}	\prod_{i=1}^r (a_i;q)_k(b_iq^k;q)_{n-k}\\
=\prod_{i=1}^r (b_i;q)_n\cdot {}_{r+1}{\phi}_r\left(\begin{matrix} q^{-n},a_1,\dots,a_r\\b_1,\dots,b_r\end{matrix}
\Bigl| q,z\right).
\end{multline*}
\end{definition}

\begin{theorem}\label{thm:qracah}
Fix $J\in \Z_{\geq 1}$. Then, with $\beta :=\alpha q^J$,
\begin{align*}
	\bernw{J}{q}{\nu}{\alpha}(i_1,j_1;i_2,j_2)&= \mathbf{1}_{i_1+j_1=i_2+j_2}
	q^{\frac{2j_1-j_1^2}{4}-\frac{2j_2-j_2^2}{4}+\frac{i_2^2+i_1^2}4+\frac{i_2(j_2-1)+i_1j_1}{2}}
	\\
&\hspace{20pt}\times
	\frac{
	\nu^{j_1-i_2}\alpha^{j_2-j_1+i_2}(-\alpha \nu^{-1};q)_{j_2-i_1}}{(q;q)_{i_2} (-\alpha;q)_{i_2+j_2}
	(\beta \alpha^{-1} q^{1-j_1};q)_{j_1-j_2}}
	{}_{4}\bar{\phi}_3\left(\begin{matrix} q^{-i_2};q^{-i_1},-\beta,-q\nu \alpha^{-1}\\\nu,q^{1+j_2-i_1},\beta \alpha^{-1} q^{1-i_2-j_2}\end{matrix}
	\Bigl|\, q,q\right).
\end{align*}
In particular, $\bernw{J}{q}{\nu}{\alpha}$ is stochastic and satisfies the recursion (\ref{eq:rec}).
\end{theorem}

\begin{rem}
To relate to \cite{Borodin2014vertex}, we may rewrite
\begin{equation*}
\bernw{J}{q}{\nu}{\alpha}(i_1,j_1;i_2,j_2)=(-s)^{j_2} (-sv^{-1})^{j_1-j_2} q^{\frac{2j_1-j_1^2}{4}-\frac{2j_2-j_2^2}{4}} \frac{(q^J q^{1-j_2};q)_{\infty}}{(q^J q^{1-j_1};q)_{\infty}} \tilde{w}^{(J)}_{v}(i_2,j_2,i_1,j_1),
\end{equation*}
where the parameters $(s,v)$ are related to our parameters $(\al,\nu)$ as $\al=-sv,\quad \nu=s^{2}$ and $v=-\al/\sqrt\nu,\quad s=\sqrt\nu$. The general $J$ vertex weight $\widetilde w_v^{(J)}$ from \cite[Corollary 6.5]{Borodin2014vertex} is given by:
\begin{multline*}
\widetilde w_v^{(J)}(i_1,j_1;i_2,j_2)\\
=\mathbf{1}_{i_1+j_1=i_2+j_2}\frac{(-1)^{i_1+j_1} q^{\frac{i_1^2+i_2^2}4+\frac{i_1(j_1-1)+i_2j_2}{2}}s^{j_1-i_1}v^{i_1}(vs^{-1};q)_{j_1-i_2}}{(q;q)_{i_1} (vs;q)_{i_1+j_1}}\, {}_{4}\bar{\phi}_3\left(\begin{matrix} q^{-i_1};q^{-i_2},q^J sv,qsv^{-1}\\s^2,q^{1+j_1-i_2},q^{1+J-i_1-j_1}\end{matrix}
\Bigl|\, q,q\right).
\end{multline*}
\end{rem}

\begin{proof}
In order to prove this, it suffices to check that the $J=1$ formula in the statement of the theorem matches the formulas given in Definition \ref{defweightsJ1}, and then to check that for $J\in \Z_{\geq 1}$, the formula in the theorem  satisfies the recursion (\ref{eq:rec}). The $J=1$ case is easily checked from definitions so we proceed to the recursion. To achieve this aim, we utilize an identity (\ref{general_hypergeometric_identity}) involving ${}_{4}\bar{\phi}_3$, as well as a recasting of our $\Lmat$-matrix in terms of $q$-Racah polynomials and an associated three term recursion for those polynomials.

Let us start by rewriting $\bernw{J}{q}{\nu}{\alpha}$ and the desired recursion. Changing from $\alpha,\nu$ to $s,v$ variables via $\al=-sv,\quad \nu=s^{2}$, we can rewrite
\begin{align}\label{expressing_R_via_phi}
	\bernw{J}{q}{\nu}{\alpha}(i_1,j_1,i_2,j_2)&=\mathbf{1}_{i_1+j_1=i_2+j_2}
	(-1)^{i_2+j_2+j_1}
	q^{\frac{2j_1-j_1^2}{4}-\frac{2j_2-j_2^2}{4}+\frac{i_2^2+i_1^2}4+\frac{i_2(j_2-1)+i_1j_1}{2}}
	\\\nonumber	&\hspace{20pt}\times
	\frac{
	s^{j_1+j_2-i_2}v^{j_2-j_1+i_2}(vs^{-1};q)_{j_2-i_1}}{(q;q)_{i_2} (vs;q)_{i_2+j_2}
	(q^{J}q^{1-j_1};q)_{j_1-j_2}}
	{}_{4}\bar{\phi}_3\left(\begin{matrix} q^{-i_2};q^{-i_1},q^J sv,qsv^{-1}\\s^2,q^{1+j_2-i_1},q^{1+J-i_2-j_2}\end{matrix}
	\Bigl|\, q,q\right)
\end{align}
For the remainder of the proof we will replace $(i_1,j_1,i_2,j_2)$ by $(g,h,g',h')$.
%
By utilizing the formula for $\bernw{1}{q}{\nu}{\alpha}$ and $p^{(J)}$, the desired recursion can be rewritten as follows
(by agreement, $g'=g+h-h'$; note that this quantity does not change throughout the identity):
\begin{multline}\label{R_recursion_proof_1}
	\bernw{R}{q}{\nu}{\alpha}(g,h,g',h')
	\\=
	\frac{(-sv+s^{2} q^{g})(1-q^{h})}{(1-sv)(1-q^{J})} \bernw{J-1}{q}{\nu}{q\alpha}(g,h-1,g',h'-1)
	+
	\frac{(1-s^{2} q^{g})(1-q^{h})}{(1-sv)(1-q^{J})}\bernw{J-1}{q}{\nu}{q\alpha}(g+1,h-1,g',h')
	\\
	+\frac{sv(1-q^{g})(q^{J}-q^{h})}{(1-sv)(1-q^{J})}\bernw{J-1}{q}{\nu}{q\alpha}(g-1,h,g',h'-1)
	-
	\frac{(1-sv q^{g})(q^{J}-q^{h})}{{(1-sv)(1-q^{J})}}\bernw{J-1}{q}{\nu}{q\alpha}(g,h,g',h').
\end{multline}

Using \eqref{expressing_R_via_phi}, we can recast
this desired recursion as an identity between ${}_{4}\bar{\phi}_3$ functions:
\begin{align*}
	&{}_{4}\bar{\phi}_3\left(
	\begin{matrix} q^{-g'};q^{-g},q^J sv,qsv^{-1}\\s^2,q^{1+h'-g},q^{1+J-h-g}
	\end{matrix}
	\Bigl|\, q,q\right)
	=
	\frac{\left(q^h-1\right) \left(s q^g-v\right)}{\left(q^J-1\right) (s-v)}
	{}_{4}\bar{\phi}_3\left(
	\begin{matrix} q^{-g'};q^{-g},q^J sv,sv^{-1}\\s^2,q^{h'-g},q^{1+J-h-g}
	\end{matrix}
	\Bigl|\, q,q\right)
	\\&\hspace{90pt}-\frac{v \left(q^h-1\right) \left(s^2 q^g-1\right) \left(q^{h'}-q^J\right) q^{g+h-h'}}{\left(q^J-1\right) (s-v) \left(s v q^{g+h}-1\right)}
	{}_{4}\bar{\phi}_3\left(
	\begin{matrix} q^{-g'};q^{-g-1},q^J sv,sv^{-1}\\s^2,q^{h'-g},q^{J-h-g}
	\end{matrix}
	\Bigl|\, q,q\right)
	\\&\hspace{90pt}+\frac{q^{-g} \left(q^g-1\right) \left(s q^g-v q^{h'}\right)}{\left(q^J-1\right) (s-v)}
	{}_{4}\bar{\phi}_3\left(
	\begin{matrix} q^{-g'};q^{1-g},q^J sv,sv^{-1}\\s^2,q^{1+h'-g},q^{1+J-h-g}
	\end{matrix}
	\Bigl|\, q,q\right)
	\\&\hspace{90pt}+
	\frac{q^{h-h'} \left(q^{h'}-q^J\right) \left(s v q^g-1\right) \left(v q^{h'}-s q^g\right)}{\left(q^J-1\right) (s-v) \left(s v q^{g+h}-1\right)}
	{}_{4}\bar{\phi}_3\left(
	\begin{matrix} q^{-g'};q^{-g},q^J sv,sv^{-1}\\s^2,q^{1+h'-g},q^{J-h-g}
	\end{matrix}
	\Bigl|\, q,q\right).
\end{align*}
For better notation, let us set $q^{g}=G$, $q^{h}=H$,
and replace $q^{h'}$ by $q^{-g'}GH$. We arrive at the following
identity:
\begin{align}
	&{}_{4}\bar{\phi}_3\left(
	\begin{matrix} q^{-g'};1/G,q^J sv,qsv^{-1}\\s^2,Hq^{1-g'},q^{1+J}/(GH)
	\end{matrix}
	\Bigl|\, q,q\right)
	=
	\nonumber
	\frac{(1-H)(Gs-v)}{(1-q^{J})(s-v)}
	{}_{4}\bar{\phi}_3\left(
	\begin{matrix} q^{-g'};1/G,q^J sv,sv^{-1}\\s^2,Hq^{-g'},q^{1+J}/(GH)
	\end{matrix}
	\Bigl|\, q,q\right)
	\\&\hspace{100pt}\nonumber
	-\frac{v(1-H)(GH-q^{g'+J})(1-s^{2}G)}{(1-q^J)(1-GHsv)(s-v)}
	{}_{4}\bar{\phi}_3\left(
	\begin{matrix} q^{-g'};q^{-1}/G,q^J sv,sv^{-1}\\s^2,Hq^{-g'},q^{J}/(GH)
	\end{matrix}
	\Bigl|\, q,q\right)
	\\&\hspace{100pt}+\nonumber
	\frac{(1-G)(s-Hvq^{-g'})}{(1-q^J)(s-v)}
	{}_{4}\bar{\phi}_3\left(
	\begin{matrix} q^{-g'};q/G,q^J sv,sv^{-1}\\s^2,Hq^{1-g'},q^{1+J}/(GH)
	\end{matrix}
	\Bigl|\, q,q\right)
	\\&\hspace{100pt}+
	\frac{(GH-q^{g'+J})(s-Hvq^{-g'})(1-Gsv)}{(s-v)(1-q^J)(1-GH sv)}
	{}_{4}\bar{\phi}_3\left(
	\begin{matrix} q^{-g'};1/G,q^J sv,sv^{-1}\\s^2,Hq^{1-g'},q^{J}/(GH)
	\end{matrix}
	\Bigl|\, q,q\right).
	\label{main_hypergeometric_identity}
\end{align}
Note that for integer $g'=0,1,2,\ldots$,
both sides of the above identity are rational functions
in $G,H,q^{J},s,v$, and $q$.

We will now use the following simple identity
for basic hypergeometric series
(which can be readily verified
by looking at individual terms in ${}_{4}\bar{\phi}_3$):
\begin{multline}
	{}_{4}\bar{\phi}_3\left(\begin{matrix} q^{-n};b,c,d\\u,v,w\end{matrix}
	\Bigl|\, q,z\right)\\=
	\frac{(1-d) (c-w) }{(c-d) \left(1-w q^n\right)}
	{}_{4}\bar{\phi}_3\left(\begin{matrix} q^{-n};b,c,dq\\u,v,wq\end{matrix}
	\Bigl|\, q,z\right)
	+\frac{(1-c) (w-d)}{(c-d) \left(1-w
   q^n\right)}
   {}_{4}\bar{\phi}_3\left(\begin{matrix} q^{-n};b,cq,d\\u,v,wq\end{matrix}
	\Bigl|\, q,z\right).
	\label{general_hypergeometric_identity}
\end{multline}
Using this identity, one can rewrite summands in the right-hand side of \eqref{main_hypergeometric_identity} with bottom arguments $Hq^{-g'}$ or $q^{J}/(GH)$ in terms of basic hypergeometric functions
with bottom arguments $Hq^{1-g'}$ and $q^{1+J}/(GH)$.
Let us denote the three resulting
functions by
\begin{align*}
	\Phi_q&:={}_{4}\bar{\phi}_3\left(
	\begin{matrix} q^{-g'};q/G,q^J sv,sv^{-1}\\s^2,Hq^{1-g'},q^{1+J}/(GH)
	\end{matrix}
	\Bigl|\, q,q\right),
	\\
	\Phi_1&:={}_{4}\bar{\phi}_3\left(
	\begin{matrix} q^{-g'};1/G,q^J sv,qsv^{-1}\\s^2,Hq^{1-g'},q^{1+J}/(GH)
	\end{matrix}
	\Bigl|\, q,q\right),
	\\
	\Phi_{q^{-1}}&:=
	{}_{4}\bar{\phi}_3\left(
	\begin{matrix} q^{-g'};q^{-1}/G,q^J sv,q^2sv^{-1}\\s^2,Hq^{1-g'},q^{1+J}/(GH)
	\end{matrix}
	\Bigl|\, q,q\right).
\end{align*}
Applying \eqref{general_hypergeometric_identity}
to the first term
in the right-hand side of \eqref{main_hypergeometric_identity},
we get
\begin{align*}
	&\frac{(1-H)(Gs-v)}{(1-q^{J})(s-v)}{}_{4}\bar{\phi}_3\left(
	\begin{matrix} q^{-g'};1/G,q^J sv,sv^{-1}\\s^2,Hq^{-g'},q^{1+J}/(GH)
	\end{matrix}
	\Bigl|\, q,q\right)
	=-\frac{(1-G) \left(s-H q^{-g'} v\right)}{\left(1-q^J\right) (s-v)}
	\Phi_{q}
	+
	\frac{1-G H q^{-g'}}{1-q^J}\Phi_{1}.
\end{align*}
The second term in the right-hand side of \eqref{main_hypergeometric_identity}
is rewritten as follows (note that we need to use \eqref{general_hypergeometric_identity}
twice):
\begin{align*}
	&-\frac{v(1-H)(GH-q^{g'+J})(1-s^{2}G)}{(1-q^J)(1-GHsv)(s-v)}
	{}_{4}\bar{\phi}_3\left(
	\begin{matrix} q^{-g'};q^{-1}/G,q^J sv,sv^{-1}\\s^2,Hq^{-g'},q^{J}/(GH)
	\end{matrix}
	\Bigl|\, q,q\right)
	\\&\hspace{90pt}=
	\frac{(G-1) v (G q-1) \left(G s^2-1\right) q^{-g'} \left(s q^{g'}-H v\right) \left(G H s-v q^J\right)}{\left(q^J-1\right) (s-v) (G s-v) (G H s v-1) (G q s-v)}
	{\Phi_{q}}
	\\&\hspace{120pt}
	+
	\frac{G v (G q-1) \left(G s^2-1\right) q^{-g'} \left(q^J-H\right) \left(H v-s q^{g'}\right)}{\left(q^J-1\right) (G s-v) (G H s v-1) (G q s-v)}
	{\Phi_{1}}
	\\&\hspace{120pt}
	-\frac{v (G q-1) \left(G s^2-1\right) q^{-g'} \left(q^{g'}-G H q\right) \left(v q^J-G H q s\right)}{\left(q^J-1\right) (G H s v-1) (G q s-v) \left(G q^2 s-v\right)}
	{\Phi_{1}}
	\\&\hspace{120pt}
	-\frac{G v \left(G s^2-1\right) q^{-g'} \left(q^{J+1}-H\right) (q s-v) \left(q^{g'}-G H q\right)}{\left(q^J-1\right) (G H s v-1) (G q s-v) \left(G q^2 s-v\right)}
	{\Phi_{q^{-1}}}.
\end{align*}
The third term
already contains $\Phi_{q}$, so there is no
need to apply
\eqref{general_hypergeometric_identity}
to it.
The fourth term in the right-hand side of
\eqref{main_hypergeometric_identity} takes the form:
\begin{align*}
	&\frac{(GH-q^{g'+J})(s-Hvq^{-g'})(1-Gsv)}{(s-v)(1-q^J)(1-GH sv)}
	{}_{4}\bar{\phi}_3\left(
	\begin{matrix} q^{-g'};1/G,q^J sv,sv^{-1}\\s^2,Hq^{1-g'},q^{J}/(GH)
	\end{matrix}
	\Bigl|\, q,q\right)
	\\&\hspace{20pt}=
	\frac{(G-1) q^{-g'} (G s v-1) \left(s q^{g'}-H v\right) \left(v q^J-G H s\right)}{\left(q^J-1\right) (s-v) (G s-v) (G H s v-1)}
	{\Phi_{q}}
	+
	\frac{G q^{-g'} (G s v-1) \left(q^J-H\right) \left(s q^{g'}-H v\right)}{\left(q^J-1\right) (G s-v) (G H s v-1)}
	{\Phi_1}.
\end{align*}

We see that the right-hand
side of \eqref{main_hypergeometric_identity}
can be written as a linear combination of $\Phi_{q}$, $\Phi_{1}$, and $\Phi_{q^{-1}}$. Note that the left-hand side of
\eqref{main_hypergeometric_identity} is a multiple of $\Phi_1$.
Moreover, collecting coefficients by
$\Phi_{q}$, $\Phi_{1}$, and $\Phi_{q^{-1}}$, we see that
\eqref{main_hypergeometric_identity}
is equivalent to
\begin{align}\label{main_hypergeometric_identity_qRacah_preliminary}
	\frac{G H q^{-g'}\left(1-q^{g'}\right) \left(1-s v q^J\right)}
	{\left(1-q^J\right) (1-G H s v)}\Phi_1=\tilde B \Phi_{q^{-1}}
	-(\tilde B+\tilde D)\Phi_{1}+\tilde D\Phi_{q},
\end{align}
where we are using the notation
\begin{align*}
	\tilde B&:=-\frac{G v \left(G s^2-1\right) \left(q^{J+1}-H\right) (q s-v)
		\left(1-G H q^{1-g'} \right)}{\left(q^J-1\right) (G H s v-1) (G q s-v) \left(G q^2 s-v\right)},
	\\
	\tilde D&:=\frac{(1-G) \left(s-H q^{-g'} v\right)\left(G H s-v q^J\right)G (q-s v)}
		{\left(1-q^J\right)(G s-v) (G H s v-1)(G q s-v)}.
\end{align*}

To complete the proof, we aim to match this desired identity
\eqref{main_hypergeometric_identity_qRacah_preliminary} to the following $q$-difference relation for the
$q$-Racah orthogonal polynomials \cite[(3.2.6)]{Koekoek1996}:
\begin{align}\label{q_racah_recurrence_koekoek}
	q^{-n}(1-q^{n})(1-\alqr \beqr q^{n+1})\mathsf{R}_n(x)=
	B(x)\mathsf{R}_n(x+1)
	-
	(B(x)+D(x))\mathsf{R}_n(x)
	+D(x)\mathsf{R}_n(x-1),
\end{align}
where
\begin{align*}
	B(x)&=\frac{(1-\alqr q^{x+1})(1-\beqr\deqr q^{x+1})(1-\gammaqr q^{x+1})(1-\gammaqr \deqr q^{x+1})}
	{(1-\gammaqr \deqr q^{2x+1})(1-\gammaqr \deqr q^{2x+2})},
	\\
	D(x)&=
	\frac{q(1-q^{x})(1-\deqr q^{x})(\beqr-\gammaqr q^{x})(\alqr-\gammaqr\deqr q^{x})}
	{(1-\gammaqr \deqr q^{2x})(1-\gammaqr \deqr q^{2x+1})}
	,
\end{align*}
and
\begin{align}
	\label{qracah}
	\mathsf{R}_n(x)=
	{}_{4}{\phi}_3\left(
	\begin{matrix} q^{-n},\alqr\beqr q^{n+1},q^{-x},\gammaqr\deqr q^{x+1}
	\\\alqr q,\beqr \deqr q,\gammaqr q
	\end{matrix}
	\Bigl|\, q,q\right), \qquad
	n=0,1,2,\ldots,N,
\end{align}
with either of the following three truncations holding:
\begin{align}
	q^{-N}=\begin{cases}
		\alqr q;\\
		\beqr\deqr q;\\
		\gammaqr q.
	\end{cases}
	\label{q_racah_truncations}
\end{align}
The above quantities $\mathsf{R}_n(x)$, $x=0,1,2,\ldots,N$,
are polynomials in $q^{-x}+\gammaqr\deqr q^{x+1}$ (of degree $n$), and are
orthogonal on $\{0,1,2,\ldots,N\}$
with a certain weight.

In fact, for the identity \eqref{q_racah_recurrence_koekoek}
the truncations \eqref{q_racah_truncations} are not relevant
since \eqref{q_racah_recurrence_koekoek} is
an identity between rational functions in $\alqr,\beqr,\gammaqr,\deqr$.
We will utilize \eqref{q_racah_recurrence_koekoek} with this
understanding.

Note also that since the bottom arguments
of the hypergeometric function for $\mathsf{R}_n(x)$
do not depend on $x$ or $n$, the same identity
\eqref{q_racah_recurrence_koekoek}
also holds if one
replaces ${}_{4}{\phi}_3$
by ${}_{4}\bar{\phi}_3$
in \eqref{qracah}.


Recalling our notation
$G=q^{g}$ and $H=q^{h}$
and setting\footnote{In fact, it is also possible to assign other values to $\alqr,\beqr,\gammaqr,\deqr$ which match all the parameters.
We will not discuss all possible ways of
choosing $\alqr,\beqr,\gammaqr,\deqr$.}
\begin{align}
	n=g';\quad x=g;\quad
	\alqr=s^{2}/q
	;\quad
	\beqr=vq^{J-g'}/s
	;\quad
	\gammaqr=q^{h-g'}
	;\quad
	\deqr=sq^{g'-g-h}/v,
	\label{qracah_parameters}
\end{align}
one readily sees that
\begin{align*}
	&\Phi_1={}_{4}\bar{\phi}_3\left(
	\begin{matrix} q^{-g'};1/G,q^J sv,qsv^{-1}\\s^2,Hq^{1-g'},q^{1+J}/(GH)
	\end{matrix}
	\Bigl|\, q,q\right)=
	\mathrm{const}\cdot \mathsf{R}_n(x),
	\\&
	\Phi_q=\mathrm{const}\cdot \mathsf{R}_n(x-1),
	\qquad
	\Phi_{q^{-1}}=\mathrm{const}\cdot \mathsf{R}_n(x+1),
\end{align*}
with the overall constant
\begin{align*}
	\mathrm{const}=
	(s^2;q)_{g'}
	(Hq^{1-g'};q)_{g'}
	(q^{1+J}/(GH);q)_{g'}.
\end{align*}
The coefficients in
\eqref{main_hypergeometric_identity_qRacah_preliminary}
and \eqref{q_racah_recurrence_koekoek}
also match:
\begin{align*}
	\left(1-q^J\right) (1-G H s v)\tilde B=
	{GH}\cdot B(x);\qquad
	\left(1-q^J\right) (1-G H s v)\tilde D=
	{GH}\cdot D(x).
\end{align*}
This implies the desired recursion on the R matrix.
\end{proof}

\begin{rem}
		From the above proof we see that the vertex weights
	$\bernw{J}{q}{\nu}{\alpha}(i_1,j_1;i_2,j_2)$ can be interpreted as an analytic
	continuation of $q$-Racah polynomials.
	Namely, if $\nu=s^{2}=q^{-I}$ for some $I=1,2,\ldots$,
	then $\bernw{J}{q}{\nu}{\alpha}(i_1,j_1;i_2,j_2)$
	is a multiple of the $q$-Racah polynomial
	$\mathsf{R}_{i_2}(i_1)$ \eqref{qracah}
	on
	$i_1\in\{0,1,\ldots,I\}$.
	The parameters
	of the polynomial $\mathsf{R}_{i_2}(i_1)$
	are\footnote{Similarly
	one can take $\mathsf{R}_{i_1}(i_2)$, which involves interchanging
	parameters $(\alqr,\beqr)$ with $(\gammaqr,\deqr)$.}
	\begin{align*}
		\alqr=q^{-I-1},\qquad
		\beqr=-\al q^{I+J-i_2}
		,\qquad
		\gammaqr=q^{j_1-i_2}
		,\qquad
		\deqr=-\al^{-1}q^{-I-j_2}
	\end{align*}
	(these are the same parameters as in \eqref{qracah_parameters}
	expressed through $\al=-sv$ and $\nu=s^{2}=q^{-I}$).
	The fact that $\alqr q$ is a negative integer power of $q$
	ensures that the index $i_2$ and the variable $i_1$
	of the $q$-Racah polynomials are restricted to the finite integer segment
	$\{0,1,\ldots,I\}$.

	Note that the parameters of the $q$-Racah polynomials also depend on $(i_1,j_1;i_2,j_2)$,
	and one cannot interpret the weights as values of a single $q$-Racah polynomial
	$\mathsf{R}_{n}(x)$
	(with constant parameters $\alqr,\beqr,\gammaqr,\deqr$)
	at various integer points $x$.
	This observation agrees with the fact that we can choose
	vertex weights to be non-negative (see Proposition \ref{rem:stochweights}),
	but the orthogonal polynomial
	$\mathsf{R}_{n}(x)$ must change sign as a function of $x$.		
\end{rem}

\begin{rem}\label{sec:analcont}
Observe that our $\Lmat$-matrix formula in Theorem \ref{thm:qracah} admits an analytic continuation in the parameter $\beta :=\alpha q^J$.
For $\beta\in \C$ which is not equal $\alpha q^J$ for some $J\in \Z_{\geq 1}$, it is unclear whether $\bernwbeta{\beta}{q}{\nu}{\alpha}$ is stochastic. One can argue, as in the comments after \cite[Corollary 6.6]{Borodin2014vertex}, that for fixed $i_1,j_1$, the sum over $i_2,j_2$ of the $\Lmat$-matrix elements equals 1 (though the nature of this convergence may not be so straightforward). It is the positivity, however, which becomes suspect for general $\beta$. By inspection of terms in the summation defining ${}_{4}\bar{\phi}_3$ one finds that unless $\beta$ takes the special form (which results in many terms being zero in the sum), the terms in the summation are not all of the same sign. Empirical computer testing indicates that for generic $\beta$, $\Lmat$-matrix entries are not always positive.

It is because of this lack of stochasticity that we do not pursue this $\beta$ continuation further at this time. One can certainly (as in the comment from \cite{Borodin2014vertex}) extend the eigenfunction relations to general $\beta$. It may also be possible to extend the duality and the formulas proved in Section \ref{sec:mom} to this non-stochastic setting (in which case, expectations in Section \ref{sec:mom} must be replaced by signed expectations). The only case in which we deal with an extension to general $\beta$ is in Section \ref{sec:qHahn}.
\end{rem}

\section{Moment and $e_q$-Laplace transform formulas}\label{sec:mom}

Combining the first duality result of Corollary \ref{cor:dualhigherJ} with the integrability of $\revhighspinBoson{\alpha}{q^J\alpha}$ given in Corollary \ref{cor:higherJeig}, we may employ results of \cite{BCPS2014} to compute moment formulas and then follow the general scheme developed in \cite{BorodinCorwin2011Macdonald,BorodinCorwinSasamoto2012} to arrive at an $e_q$-Laplace transform formulas for the $J$ higher spin exclusion process (recall from Definition \ref{def:highJBosonTASEP}) with step initial data. The other dualities of Corollary \ref{cor:dualhigherJ} may prove similarly useful in computing moment formulas (perhaps for different choices of initial data), but we do not pursue these here.

\begin{theorem}\label{t.moments}
Fix $J\in \Z_{\geq 1}$, $\beta=\alpha q^J$ and consider the $J$ higher spin exclusion process $\vec{x}(\cdot)$ with step initial data $x_i(0)=-i, i\in \Z_{\geq 1}$. For all $n_1\geq \cdots \geq n_k\geq 1$ and $t\in \Z_{\geq 0}$,
$$
\EE\bigg[\prod_{i=1}^{k} q^{x_{n_i}(t)+n_i}\bigg] = \frac{(-1)^k q^{\frac{k(k-1)}{2}}}{(2\pi \I)^k} \oint_{\ga_1} \cdots \oint_{\ga_k} \prod_{1\leq A<B\leq k} \frac{z_A-z_B}{z_A-q z_B} \prod_{j=1}^{k} \Big(\frac{1-\nu z_j}{1-z_j}\Big)^{n_j} \Big(\frac{1+\beta z_j}{1+\alpha z_j}\Big)^{t} \frac{dz_j}{z_j(1-\nu z_j)}
$$
where the simple closed integration contours $\ga_1,\ldots,\ga_k$ are chosen such that they all contain $1$, the $\ga_A$ contour contains $q\ga_B$ for all $B>A$ and all contours exclude $0$ and $1/\nu$.
\end{theorem}
\begin{proof}
This is proved in the same manner as \cite[Corollary 5.19]{BCPS2014}, replacing the eigenvalue $\prod_{j=1}^{k}\frac{1-\mu z_j}{1-\nu z_j}$ therein by the present value $\prod_{j=1}^{k}\frac{1+\beta z_j}{1+\alpha z_j}$. Rather than repeating the proof, we briefly explain the idea behind it. The first duality result of Corollary \ref{cor:dualhigherJ} implies that
$$\EE\Big[\prod_{i=1}^{k} q^{x_{n_i}(t+1)+n_i}\Big] = \revhighspinBoson{\alpha}{q^J\alpha} \EE\Big[\prod_{i=1}^{k} q^{x_{n_i}(t)+n_i}\Big],$$
where $\revhighspinBoson{\alpha}{q^J\alpha}$ acts on the $\vec{n}$ variables via the association of $\vec{n}$ with $\vec{y}$, as in Definition \ref{deftransprobJ1}. The operator $\revhighspinBoson{\alpha}{q^J\alpha}$ is diagonalized (Corollary \ref{cor:higherJeig}) by the eigenfunctions introduced in Appendix \ref{sec:Bethe} and utilizing the associated Plancherel theory, we arrive at the solution to the time evolution equation satisfied by $\EE\Big[\prod_{i=1}^{k} q^{x_{n_i}(t+1)+n_i}\Big]$, thus proving the theorem.
\end{proof}

For the exclusion process with step initial data, the formula from Theorem \ref{t.moments} provide a complete characterization of the distribution of $\vec{x}(t)$. This is because each random variable $q^{x_{n}(t) +n}$, $1\leq n\leq N$, is in $(0,1]$ for all $t$. The knowledge of all joint moment thus suffices to characterize the joint distribution. Despite this fact, it is not obvious how to extract meaningful asymptotic distribution information from these formulas. In the case of one-point distributions (i.e. the distribution of $x_n(t)$ for a single $n$) this was achieved in \cite{BorodinCorwin2011Macdonald}. We will apply the approach developed therein (in particular, the general restatement of the calculation in \cite{BorodinCorwin2011Macdonald} which can be found in \cite[Section 3]{BorodinCorwinSasamoto2012}).

Theorem \ref{t.fred} provides two Fredholm determinant formulas for what is called the $e_q$-Laplace transform of the observable $q^{x_{n}(t)+n}$, and consequently for the one-point distribution of $x_n(t)$ (see \cite[Proposition 3.1.1]{BorodinCorwin2011Macdonald} or \cite[Proposition 7.1]{BorodinCorwinSasamoto2012}).  This type of Fredholm determinant formula (in particular that of (\ref{MellinBarnes})) is quite amenable to asymptotic analysis -- see for instance \cite{BorodinCorwin2011Macdonald,BorodinCorwinFerrari2012,BorodinCorwinFerrariVeto2014,BorodinCorwinRemenik,Barraquand_qTASEP_2014,FerrariVeto2013,Veto2014qhahn,OConnellOrthmann2014,CorwinSeppalainenShen2014}.

\begin{theorem}\label{t.fred}
Fix $J\in \Z_{\geq 1}$, $\beta=\alpha q^J$ and consider the $J$ higher spin exclusion process $\vec{x}(\cdot)$ with step initial data $x_i(0)=-i, i\in \Z_{\geq 1}$. For all $n\in \Z_{\geq 1}$, $t\in \Z_{\geq 0}$, and $\zeta\in\C\setminus \R_+$,
\begin{equation}\label{MellinBarnes}
\EE \left[\frac{1}{\big(\zeta q^{x_{n}(t)+n};q\big)_{\infty}}\right] = \det\big(I + K_{\zeta}\big)
\end{equation}
where $\det\big(I + K_{\zeta}\big)$ is the Fredholm determinant of $K_\zeta: L^2(C_1)\to L^2(C_1)$ for $C_1$ a positively oriented circle containing 1 with small enough radius so as to not contain 0, $1/q$, and $1/\nu$. The operator $K_\zeta$ is defined in terms of its integral kernel
\begin{equation*}
K_{\zeta}(w,w') = \frac{1}{2\pi \i} \int_{-\i \infty + 1/2}^{\i\infty +1/2} \frac{\pi}{\sin(-\pi s)} (-\zeta)^s \frac{\fredg(w)}{\fredg(q^s w)} \frac{1}{q^s w - w'} ds
\end{equation*}
with
\begin{equation*}
\fredg(w) = \left(\frac{(\nu w;q)_{\infty}}{(w;q)_{\infty}}\right)^{n} \left( \frac{(-\beta w;q)_{\infty}}{(-\alpha w;q)_{\infty}}\right)^t \frac{1}{(\nu w;q)_{\infty}}.
\end{equation*}

The following second formula also holds:
\begin{equation}\label{Cauchy}
\EE \left[\frac{1}{\big(\zeta q^{x_{n}(t)+n};q\big)_{\infty}}\right] = \frac{\det\big(I + \zeta \tilde{K}\big)}{(\zeta;q)_{\infty}}
\end{equation}
where $\det\big(I + \zeta \tilde{K}\big)$ is the Fredholm determinant of $\zeta$ times the operator $\tilde{K}_\zeta: L^2(C_{0,1})\to L^2(C_{0,1})$ for $C_{0,1}$ a positively oriented circle containing 0 and 1 (but not $1/\nu$ or $-1/\alpha$). The operator $\tilde{K}$ is defined in terms of its integral kernel
\begin{equation*}
\tilde{K}(w,w') = \frac{\fredg(w)/\fredg(qw)}{qw'-w}
\end{equation*}
where the function $\fredg$ is as above.
\end{theorem}

\begin{proof}
This type of deduction of Fredholm determinant formulas from $q$-moment formulas has appeared before, cf. \cite{BorodinCorwin2011Macdonald,BorodinCorwinSasamoto2012,BorodinCorwinFerrari2012,BorodinCorwin2013discrete,BorodinCorwinFerrariVeto2014}. Hence, we provide only the steps of the proof, without going into much detail. We also do not recall the definition of Fredholm determinants, cf. \cite[Section 3.2.2]{BorodinCorwin2011Macdonald}.

In order to prove the first formula (which is sometimes called a Mellin-Barnes type formula) we utilize the formula for $\EE\big[q^{k(x_n(t)+n)}\big]$ from specializing all $n_i\equiv n$ in Theorem \ref{t.moments}. Call $\mu_k := \EE\big[q^{k(x_n(t)+n)}\big]$ so as to match it with the formula present in \cite[Definition 3.1]{BorodinCorwinSasamoto2012}, subject to defining $f(w):= \fredg(w)/\fredg(qw)$, with $\fredg$ from the statement of Theorem \ref{t.fred}. We may then apply \cite[Propositions 3.3 and 3.6]{BorodinCorwinSasamoto2012} with the contour $C_{A}=C_1$, and $D_{R,d}, D_{R,d;k}$ specified by setting $R=1/2$ (and $d$ arbitrary, as it does not matter for this choice of $\Lmat$). The output of these propositions is that
$$
\sum_{k\geq 0} \mu_k \frac{\zeta^k}{k_q!} =  \det\big(I + K_{\zeta}\big).
$$
In the course of applying these propositions, it is necessary to check that a few technical conditions on the contours, as well as $\zeta$ and $\fredg$ are satisfied. These are easily confirmed for $|\zeta|$ small enough, and $C_{1}$ a small enough circle around 1. The only condition depending on the function $\fredg$ is that $|\fredg(w)/\fredg(q^sw)|$ remain uniformly bounded as $w\in C_{1}$, $k\in\Z_{\geq 1}$ and $s\in D_{R,d;k}$ varies. This is readily confirmed for $\fredg$ from the statement of Theorem \ref{t.fred}.

Now, observe that for $\zeta$ with $|\zeta|$ small enough, we also have that
$$
\sum_{k\geq 0} \mu_k \frac{\zeta^k}{k_q!}= \EE \left[\frac{1}{\big(\zeta q^{x_{n}(t)+n};q\big)_{\infty}}\right].
$$
This is justified (as in \cite[Theorem 3.2.11]{BorodinCorwin2011Macdonald}) by the fact that $q^{x_{n}(t)+n}\in (0,1)$ and an application of the $q$-Binomial theorem. This establishes (\ref{MellinBarnes}) for $|\zeta|$ sufficiently small and since both sides are analytic in $\C\setminus \R_{+}$ the general $\zeta$ result of (\ref{MellinBarnes}) follows via analytic continuation.

The Cauchy type formula (\ref{Cauchy}) follows from \cite[Proposition 3.10]{BorodinCorwinSasamoto2012} along with a small amount of algebra. The proof essentially follows that of \cite[Theorem 3.2.16]{BorodinCorwin2011Macdonald}.
\end{proof}

\begin{rem}
Setting $g_i(t) := x_{i-1}(t)-x_{i}(t)$, $\vec{g}(t)$ evolves as the zero range process with step initial data corresponding to having $g_1(0)=+\infty$ and $g_i(0)=0$ for $i>1$. Let $C_s(t) = \sum_{i=s+1}^{\infty} g_i(t)$ be the number of particles of $\vec{g}(t)$ strictly to the right of site $s$ at time $t$. Then clearly $\{x_n(t)+n\geq s\} = \{C_{s}(t)\geq n\}$ and hence Theorem \ref{t.fred} provides an exact formula characterizing  the distribution of $C_s(t)$ as well.
\end{rem}

%

\section{Generalizations, specializations and degenerations}\label{sec:spec}

\subsection{Case (2) of Proposition \ref{rem:stochweights}}\label{sec:case2}

Recall that in case (2) of Proposition \ref{rem:stochweights} we assume $q\in (-1,0]$, $\alpha\in (0,1/|q|)$, and $\nu\in \big(-1/|q|,\min(1,\alpha/|q|)\big)$. As already observed in Proposition \ref{rem:stochweights}, under these conditions one readily confirms that the $\Lmat$-matrix is stochastic. Thus, the zero range and exclusion processes built from the $\Lmat$-matrix in Definitions \ref{deftransprobJ1ZRP} and \ref{deftransprobJ1AEP} remain valid. The $q$-Hahn operators $\qHahnHR{q}{\nu}{\alpha}$ and $\qHahnH{q}{\nu}{\alpha}$ remain well-defined. All of the duality and fusion results extend as we now explain.

In Section \ref{sec:jones} all of the results can be extended. Proposition \ref{prophighspindiagJ1} holds since the proof only relies upon \cite[Corollary 4.5 (i)]{Borodin2014vertex} which remains valid in such an extension of parameters. Alternatively, one can observe that the equality demonstrated by the proposition can be analytically continued into the regime of parameters we are considering. In order to show the analyticity, one uses the bound that $\big\vert \tfrac{1-z_i}{1-\nu z_i}\, \tfrac{\alpha+\nu}{1+\alpha}\big\vert<1$ which ensures the convergence of the left-hand side of equality in the proposition. The factor $\tfrac{1-z_i}{1-\nu z_i}$ comes from the eigenfunctions whereas the factor $\frac{\alpha+\nu}{1+\alpha}$ comes from the weight associated with $(0,1;0,1)$, which is the only vertex weight which can occur an unbounded number of times in the application of the operator $\revhighspinBoson{\alpha}{q\alpha}$.
Proposition \ref{propqHahndiag} holds since both sides are rational functions in the parameters and can be extended to the desired range.
Corollary \ref{cor:secrat} holds since the results of Appendix \ref{sec:Bethe} hold for this present case of parameters.
Theorem \ref{thmimplicit} holds since the only properties the proof relies upon are those notes in Remark \ref{rem:followingprop} (both of which remain valid under the present choice of parameters).
Theorem \ref{thm:Gduality} holds since the proof only relies upon \cite[Proposition 1.2]{Corwin2014qmunu} which holds as rational identity in the parameters $q,\nu,\alpha$ and clearly extends to the range of parameters presently considered.

All of the results of Section \ref{sec:fus} besides those pertaining to the exclusion process hold as long as the triple $(q,\nu,\alpha q^j)$ satisfies the conditions of case (2) of Proposition \ref{rem:stochweights} for $1\leq j\leq J$. This is because they all follow from results of Section \ref{sec:jones}.

We do not pursue modifying the results of Section \ref{sec:mom}. As it stands, the formula in Theorem \ref{t.moments} is not well adapted to taking $q$ negative (because of the nested structure of the contours). Instead, one might first shrink the nested contours to all lie around $1$ (cf. \cite[Proposition 3.2]{BCPS2014}) and then argue by analytic continuation that the formula in which the parameters have been extends to case (2) of Proposition \ref{rem:stochweights}. Such a formula would then serve as the input to establishing a result like Theorem \ref{t.fred}. We leave the justifications necessary to prove such formulas to future work.

%
%
%
%

\subsection{$\Lmat$-matrix reflection and inversion}\label{sec:reflectioninversion}

Let $I,J\in\Z_{\geq 1}$. Temporarily, to emphasize the role of $I$ and $J$, let us write
$$R(i_1,j_1;i_2,j_2|I,J,\alpha,q) = \bernw{J}{q}{\nu}{\alpha},$$
where the right-hand side implicitly depends on $q$ as well as $I$ via $\nu = q^{-I}$.

There are two actions on a vertex we consider. The first is reflection in the diagonal under which $(i_1,j_1;i_2,j_2)\mapsto (j_1,i_1;j_2,i_2)$ and the second is inversion of arrows under which $(i_1,j_1;i_2,j_2)\mapsto (I-i_1,J-j_1;I-i_2,J-j_2)$. The result of these two actions on the $\Lmat$-matrix is quite simple.
\begin{proposition}
Let $\hat{\alpha}=1/\alpha$ and $\hat{q}=1/q$ then
\begin{align*}
R(j_1,i_1;j_2,i_2|I,J,\alpha,q) &=  R(i_1,j_1;i_2,j_2|J,I,\hat{\alpha},\hat{q}),\\
R(I-i_1,J-j_1;I-i_2,J-j_2|I,J,\alpha,q) &= R(i_1,j_1;i_2,j_2|I,J,\hat{\alpha},\hat{q}).
\end{align*}
Composing the two transformations (reflection and inversion) results in yet a fourth stochastic $\Lmat$-matrix in which $(I,J)\mapsto (J,I)$ and $\alpha,q$ remain fixed.
\end{proposition}
\begin{proof}
These relations can be confirmed directly from our explicit formulas for the $\Lmat$-matrices given in Theorem \ref{thm:qracah}.
\end{proof}
Cases (3) and (4) of Proposition \ref{rem:stochweights} are related by inversion of arrows: The range $q\in [0,1)$ maps to $q\in (1,+\infty)$ and the range $\alpha<-q^{-I}$ maps to $-q^{-I}<\alpha<0$. Despite this relationship, we will still consider each case separate. This is because when we construct our zero range process from the $\Lmat$-matrix, the inversion of arrows takes a state $\vec{g}$ with $\sum g_i=k$ to an state with an infinite number of particles.

Note that applying reflection to cases (1) and (2) of Proposition \ref{rem:stochweights} leads to different $\Lmat$-matrices and corresponding processes than we have presently considered. We do not pursue the study of the resulting systems any further here.

\subsection{Case (3) of Proposition \ref{rem:stochweights}}\label{sec:case3}

Recall that in case (3) of Proposition \ref{rem:stochweights} we assume that $q\in (0,1)$, $\nu = q^{-I}$ for $I\in \Z_{\geq 1}$, and $\alpha<-\nu$. With this choice of $I$, $\VI^I$ is finite dimensional and given by the span of $\{0,\ldots, I\}$. Consequently we must replace $\Ginf, \Gspace{k}, \Gmspace{k}, \Yinf, \Yspace{k}, \Ynspace{k}$ by $\Ginf_I, \Gspace{k}_I, \Gmspace{k}_I, \Yinf_I, \Yspace{k}_I, \Ynspace{k}_I$ as given at the end of Definition \ref{deftransprobJ1}. The construction of the zero range processes with transition operators $\highspinBoson{\alpha}{q\alpha}$ and $\revhighspinBoson{\alpha}{q\alpha}$ remains unchanged. However, in this finite spin case, we cannot define the exclusion process with a right-most particle (whose transition operator was $\highspinTASEP{\alpha}{q\alpha}$). This is because that construction required having an infinite gap and in the finite spin setting, that is not allowed. It may be possible to define an exclusion process with a doubly infinite state space, so long as the gaps between particles is bounded by $I$. However, as we have no use for that presently, we do not pursue it. Definition \ref{defweightsJ1} remains well-defined, even though $\varphi_{q,-\alpha,\nu}(s|y)$ is not necessarily positive (in fact, it has sign $(-1)^s$) for integers $0\leq s\leq y\leq I$. Likewise, the two equations of Remark \ref{rem:followingprop} remain valid (one readily observes that $1+\alpha  q^{y-s}$ remains non-zero for our choice of $\alpha$ and for all $0\leq s\leq y\leq I$).

The eigenfunction
relation of Proposition \ref{prophighspindiagJ1} follows since the
results of  \cite[Corollary 4.5 (i)]{Borodin2014vertex} are stated in
sufficient generality. It is also possible to show that the left-hand
and right-hand sides of the relation in the proposition are analytic
functions of $\nu$ and $\alpha$ in a suitable open domain which
connects the parameters from case (1) of Proposition
\ref{rem:stochweights} to the present choices. Analytic continuation
then implies the extension of the relation to the present case. This
argument only requires showing that for some fixed $z_i$, there is an
open domain of $(\nu,\alpha)\in \C^2$ connecting case (1) and case (3)
of Proposition \ref{rem:stochweights} such that the relation
$\big\vert \tfrac{1-z_i}{1-\nu z_i}\,
\tfrac{\alpha+\nu}{1+\alpha}\big\vert<1$ is preserved throughout the
domain. This fact is readily checked. Proposition \ref{propqHahndiag}
holds as written.

We will not make use of Corollary \ref{cor:secrat}
or
Definition \ref{def:welladapt} and there appear not to be natural
examples of initial data satisfying the well-adaptness of Remark
\ref{rem:welladapt}. Since we have not formulated an exclusion process
for these parameters, we do not have an analog of Theorem
\ref{thmimplicit} presently. Theorem \ref{thm:Gduality}, however,
holds for the choice of parameters. Since we cannot make use of
well-adaptedness, we demonstrate this result by analytic continuation.
Let us focus on establishing the first of the dualities (the other follows similarly)
$$
\highspinBoson{\alpha}{q\alpha} \GG = \GG \big(\revhighspinBoson{\alpha}{q\alpha}\big)^{T}.
$$
Similarly to the above, note that there exists an open domain of $(\nu,\alpha)\in \C^2$ connecting case (1) and case (3) of Proposition \ref{rem:stochweights} along which $\big\vert \tfrac{\alpha+\nu}{1+\alpha}\big\vert<\delta$ for some $\delta<1$. Consequently, both sides of the above identity are analytic functions of $\nu$ and $\alpha$ and the result follows. The reason for analyticity is due to the fact that for $\vec{g}$ and $\vec{y}$ fixed, the only $\Lmat$-matrix weight which can be used an unbounded number of times is $\tfrac{\alpha+\nu}{1+\alpha}$, corresponding with $(0,1;0,1)$-vertices. That this weight is bounded above by $\delta<1$ in magnitude ensures the convergence of the left-hand and right-hand sides and hence the analyticity. It should be noted that as we deform parameters, the transition operators are no longer stochastic and it is only after arriving at the terminal locations for parameters that this property is restored.

All of the results of Section \ref{sec:fus} besides those pertaining to the exclusion process hold as long as the triple $(q,\nu,\alpha q^j)$ satisfy the conditions of case (3) of Proposition \ref{rem:stochweights} for $1\leq j\leq J$. This is because they all follow from results of Section \ref{sec:jones}.

\subsection{Case (4) of Proposition \ref{rem:stochweights}}\label{sec:case4}

Recall that in case (3) of Proposition \ref{rem:stochweights} we assume that $q\in (1,\infty)$, $\nu = q^{-I}$ for $I\in \Z_{\geq 1}$, and $-\nu<\alpha<0$. With this choice of $I$, $\VI^I$ is finite dimensional and given by the span of $\{0,\ldots, I\}$. Consequently we must replace $\Ginf, \Gspace{k}, \Gmspace{k}, \Yinf, \Yspace{k}, \Ynspace{k}$ by $\Ginf_I, \Gspace{k}_I, \Gmspace{k}_I, \Yinf_I, \Yspace{k}_I, \Ynspace{k}_I$ as given at the end of Definition \ref{deftransprobJ1}. The construction of the zero range processes with transition operators $\highspinBoson{\alpha}{q\alpha}$ and $\revhighspinBoson{\alpha}{q\alpha}$ remains unchanged. However, as in Section \ref{sec:case3} we  cannot define the exclusion process with a right-most particle. Definition \ref{defweightsJ1} remains well-defined, even though $\varphi_{q,-\alpha,\nu}(s|y)$ is not necessarily positive (in fact, it has sign $(-1)^s$) for integers $0\leq s\leq y\leq I$. Likewise, the two equations of Remark \ref{rem:followingprop} remain valid.

The eigenfunction
relation of Proposition \ref{prophighspindiagJ1} follows since the
results of  \cite[Corollary 4.5 (i)]{Borodin2014vertex} are stated in
sufficient generality. Proposition \ref{propqHahndiag} holds as
written. We will not make use of Corollary \ref{cor:secrat} or Definition
\ref{def:welladapt}. Since we have not formulated an exclusion process
for these parameters, we do not have an analog of Theorem
\ref{thmimplicit} presently.

Theorem \ref{thm:Gduality} follows from analytic continuation of the
analogous result in case (3) with the following modification: In the
case of the $\GG$ duality, one should restrict to $\vec{g}\in
\Gspace{k}$ and $\vec{y}\in \Yspace{k'}$ for some $k,k'\in \Z_{\geq
1}$ (otherwise terms may fail to be finite). We may also, without loss
of generality, assume that $\vec{g}\in \Gspace{k}$ and $\vec{y}\in
\Yspace{k'}$ for some $k,k'\in \Z_{\geq 1}$ when considering the
$\GGhat$ duality (see the end of the proof of Theorem
\ref{thm:Gduality} for an explanation of how to go from this to
general $\vec{g}$ and $\vec{y}$). So, given  $\vec{g}\in \Gspace{k}$
and $\vec{y}\in \Yspace{k'}$ let us see how this analytic continuation
works. Let us focus on establishing the first duality (the other one
follows similarly)
$$
\highspinBoson{\alpha}{q\alpha} \GG = \GG \big(\revhighspinBoson{\alpha}{q\alpha}\big)^{T}.
$$
This holds for case (3), in which $q\in (0,1)$ and $\alpha<-q^{-I}$. Think of $q,\alpha$ as complex variables and fix $\nu=q^{-I}$ for $I\in \Z_{\geq 1}$. We want to extend this to also hold for case (4) in which $q\in (1,\infty)$ and $\alpha \in (-q^{-I},0)$. We do this by observing that there exists an open domain of $(q,\alpha)\in \C^2$ connecting case (3) to case (4) along which $\big\vert \tfrac{\alpha+\nu}{1+\alpha}\big\vert<\delta$ for some $\delta<1$.
Consequently, both sides of the above identity are analytic functions of $q$ and $\alpha$ and the result follows. The reason for analyticity is due to the fact that for $\vec{g}$ and $\vec{y}$ fixed, the only $\Lmat$-matrix weight which can be used an unbounded number of times is $\tfrac{\alpha+\nu}{1+\alpha}$, corresponding with $(0,1;0,1)$-vertices. That this weight is bounded above by $\delta<1$ in magnitude ensures the convergence of the left-hand and right-hand sides and hence the analyticity.

All of the results of Section \ref{sec:fus} besides those pertaining to the exclusion process hold as long as the triple $(q,\nu,\alpha q^j)$ satisfy the conditions of case (3) of Proposition \ref{rem:stochweights} for $1\leq j\leq J$. This is because they all follow from results of Section \ref{sec:jones}.

\subsection{Stochastic six-vertex model}\label{sec:ssvm}

If $I=J=1$ (so $\nu=1/q$), then the zero range process $\vec{g}(t)$ we have been considering degenerates to the stochastic six-vertex model. In that case, the $\Lmat$-matrix has six non-zero vertex configurations,
and can be parameterized by two parameters $(b_1,b_2)$ which are between zero and one
\cite{GwaSpohn1992},
\cite{BCG6V}:
\begin{center}
	\begin{tabular}{c|c|c|c|c|c}
		\begin{tikzpicture}
		[scale=.75, thick]
		\node (i1) at (0,-1) {$0$};
		\node (j1) at (-1,0) {$0$};
		\node (i2) at (0,1) {$0$};
		\node (j2) at (1,0) {$0$};
		\draw[densely dotted] (j1) -- (j2);
		\draw[densely dotted] (i1) -- (i2);
	\end{tikzpicture}
	&
		\begin{tikzpicture}
		[scale=.75, thick]
		\node (i1) at (0,-1) {$1$};
		\node (j1) at (-1,0) {$0$};
		\node (i2) at (0,1) {$1$};
		\node (j2) at (1,0) {$0$};
		\draw[densely dotted] (j1) -- (j2);
		\draw[->] (i1) -- (0,0);
		\draw[->] (0,0) -- (i2);
		\draw[densely dotted] (i1) -- (i2);
	\end{tikzpicture}
	&
		\begin{tikzpicture}
		[scale=.75, thick]
		\node (i1) at (0,-1) {$1$};
		\node (j1) at (-1,0) {$0$};
		\node (i2) at (0,1) {$0$};
		\node (j2) at (1,0) {$1$};
		\draw[densely dotted] (j1) -- (j2);
		\draw[densely dotted] (i1) -- (i2);
		\draw[->] (i1) -- (0,0);
		\draw[->] (0,0) -- (j2);
	\end{tikzpicture}
	&
		\begin{tikzpicture}
		[scale=.75, thick]
		\node (i1) at (0,-1) {$0$};
		\node (j1) at (-1,0) {$1$};
		\node (i2) at (0,1) {$0$};
		\node (j2) at (1,0) {$1$};
		\draw[densely dotted] (j1) -- (j2);
		\draw[densely dotted] (i1) -- (i2);
		\draw[->] (j1) -- (0,0);
		\draw[->] (0,0) -- (j2);
	\end{tikzpicture}
	&
		\begin{tikzpicture}
		[scale=.75, thick]
		\node (i1) at (0,-1) {$0$};
		\node (j1) at (-1,0) {$1$};
		\node (i2) at (0,1) {$1$};
		\node (j2) at (1,0) {$0$};
		\draw[densely dotted] (j1) -- (j2);
		\draw[densely dotted] (i1) -- (i2);
		\draw[->] (j1) -- (0,0);
		\draw[->] (0,0) -- (i2);
	\end{tikzpicture}
	&
		\begin{tikzpicture}
		[scale=.75, thick]
		\node (i1) at (0,-1) {$1$};
		\node (j1) at (-1,0) {$1$};
		\node (i2) at (0,1) {$1$};
		\node (j2) at (1,0) {$1$};
		\draw[densely dotted] (j1) -- (j2);
		\draw[densely dotted] (i1) -- (i2);
		\draw[->] (0,0) -- (i2);
		\draw[->] (0,0) -- (j2);
		\draw[->] (j1) -- (0,0);
		\draw[->] (i1) -- (0,0);
	\end{tikzpicture}
	\\
	\hline\rule{0pt}{20pt}
	$a_1=1$
	&
	$b_1=\dfrac{1+\al q}{1+\al}$
	&
	$c_1=\dfrac{\al(1-q)}{1+\al}$
	&
	$b_2=\dfrac{\al+q^{-1}}{1+\al}$
	&
	$c_2=\dfrac{1-q^{-1}}{1+\al}$
	&
	$a_2=1$
	\end{tabular}
\end{center}

\medskip
\medskip
Here $c_1=1-b_1$, $c_2=1-b_2$.

The $I=J=1$ vertex weights are nonnegative if and only if the parameters $(\al,q)$ belong to one of the following two families:
	\begin{enumerate}
		\item $q\in (0,1)$, $\al \le -1/q$;
		\item $q\in(1,\infty)$, $\al\in(-1/q,0)$.
	\end{enumerate}
Case (1) above corresponds with case (3) of Proposition \ref{rem:stochweights} whereas the case (2) above corresponds with case (4) of that proposition.
The ratio $b_2/b_1=1/q$ is denoted by $\tau$ in \cite{BCG6V}. Under the present choices of parameters, Theorem \ref{prophighspindiagJ1} degenerates to match the eigenrelations proved in \cite[Section 3.3]{BCG6V} via coordinate Bethe ansatz (see also \cite{Lieb67}). In \cite{BCG6V}, the eigenrelations are used to compute transition probabilities for the finite particle stochastic six-vertex model. The authors then focus on the case when $\tau<1$ (in other words, case (2) above where $q\in (1,\infty)$) and when $\vec{g}(t)$ is started from the step initial data ($g_{i}(0)= \mathbf{1}_{i\geq 1}$). Building on a combination of the approaches developed in \cite{TW_ASEP1} and \cite{BorodinCorwin2011Macdonald}, they compute a contour integral formula for $\EE[\tau^{L\Ndown_n(\vec{g})}]$, $L\in \Z_{\geq 1}$,  and eventually utilize this to compute a $e_{\tau}$-Laplace transform formula for $\tau^{\Ndown_n(\vec{g})}$.

\begin{rem}
If $I=J=2$ and $q,\alpha$  are as in cases (3) and (4) of Proposition \ref{rem:stochweights}, then the $\Lmat$-matrix is stochastic. In this case, there can be up to two particles per site and up to two particles can move in each update step. If $\alpha = -1/q^2$ then the zero range process becomes a deterministic shift. In the six-vertex case ($I=J=1$) this occurs for $\alpha = -1/q$. In that case, ASEP arises from an expansion around this shift, as $\alpha = -1/q + \epsilon$. We can perform the same expansion in the $I=J=2$ case, setting $\alpha = -1/q^2 + \epsilon$. The below table is calculated from the $J=2$ case of Appendix \ref{sec:J123}. The overall matrix (three by three) has rows and columns indexed by $j_1$ and $j_2$ in $\{0,1,2\}$ (respectively) and each matrix entry (a length three column vector) has entries indexed by $i_1$ in $\{0,1,2\}$. All terms of order smaller than $\epsilon$ are left out of the matrix. A quick inspection reveals that subtracting the shift (the order one terms) leaves something which is not stochastic. Indeed, both  $q^3(1+q)/(q-1)$ and $q^5/(1-q^2)$ arises in this $\epsilon$ expansion and their respective signs will always differ.
$$
\left(
\begin{array}{ccc}
 \left(
\begin{array}{c}
 1 \\
 0 \\
 0 \\
\end{array}
\right) & \left(
\begin{array}{c}
 \frac{\epsilon }{1-\frac{1}{q^2}} \\
 \frac{q^2 \epsilon }{1-q^2}+1 \\
 0 \\
\end{array}
\right) & \left(
\begin{array}{c}
 \frac{q \epsilon }{1-q^2} \\
 \frac{q (q+1) \epsilon }{q-1} \\
 1-\frac{q^2 (q+2) \epsilon }{q^2-1} \\
\end{array}
\right) \\
 \left(
\begin{array}{c}
 1-\frac{q^4 \epsilon }{q^2-1} \\
 \frac{q^4 \epsilon }{q^2-1} \\
 0 \\
\end{array}
\right) & \left(
\begin{array}{c}
 \frac{q^2 \epsilon }{q^2-1} \\
 \frac{\left(q^4+q^2\right) \epsilon }{1-q^2}+1 \\
 \frac{q^4 \epsilon }{q^2-1} \\
\end{array}
\right) & \left(
\begin{array}{c}
 0 \\
 \frac{\epsilon }{1-\frac{1}{q^2}} \\
 \frac{q^2 \epsilon }{1-q^2}+1 \\
\end{array}
\right) \\
 \left(
\begin{array}{c}
 1-\frac{q^3 (2 q+1) \epsilon }{q^2-1} \\
 \frac{q^3 (q+1) \epsilon }{q-1} \\
 \frac{q^5 \epsilon }{1-q^2} \\
\end{array}
\right) & \left(
\begin{array}{c}
 0 \\
 1-\frac{q^4 \epsilon }{q^2-1} \\
 \frac{q^4 \epsilon }{q^2-1} \\
\end{array}
\right) & \left(
\begin{array}{c}
 0 \\
 0 \\
 1 \\
\end{array}
\right) \\
\end{array}
\right)
$$
This seems to be a negative indication as to whether one can extract higher spin versions of ASEP in this manner.
\end{rem}

\subsection{$q$-Hahn processes}\label{sec:qHahn}

In Remark \ref{sec:analcont} we observed the possibility to analytically continue our $\Lmat$-matrix weights so as to depend on parameters $\alpha$ and $\beta$. However, for general $\beta \neq \alpha q^{J}$, $J\in \Z_{\geq 1}$, these weights were not always positive. The following proposition provides an exception to this, in which our $\Lmat$-matrix weights reduce to $q$-Hahn distribution weights (Definition \ref{def:qhahn}).

\begin{proposition}
Let $\alpha = -\nu$ and $\beta = -\mu$ for $0\leq \nu\leq \mu<1$. Then
$$\bernwbeta{\beta}{q}{\nu}{\alpha}(i_1,j_1;i_2,j_2) = \mathbf{1}_{i_1+j_1=i_2+j_2} \, \cdot\,\varphi_{\mu}(j_2|i_1)$$
for $i_1,j_1,i_2,j_2\in \Z_{\geq 0}$. Consequently, the space reversed higher spin zero range process $\vec{y}$ and higher spin exclusion process  for these parameters coincides with the $q$-Hahn zero range process\footnote{Also called the $q$-Hahn Boson process, or $(q,\mu,\nu)$-Boson process.} and $q$-Hahn TASEP studied in \cite{Povolotsky2013,Corwin2014qmunu}.
\end{proposition}
\begin{rem}
For the above choices of $\alpha,\beta$, the proposition implies that the $\Lmat$-matrix weights do not depend on $j_1$ and hence the processes constructed from these weights become parallel update.
\end{rem}
\begin{proof}
This reduction is essentially proved in \cite[Proposition 6.7]{Borodin2014vertex} by studying how the hypergeometric functions specialized with these parameters. The equality of the associated processes constructed from these $\Lmat$-matrix weights follows immediately. Note, however, that as far as the equality of the processes is concerned, this could be shown directly by noting that the eigenfunctions and eigenvalues coincide under this parameter specialization (cf. Proposition \ref{propqHahndiag} and Corollary \ref{cor:higherJeig}). This route would require extending Corollary \ref{cor:higherJeig} to general $\beta$ in the manner described in Remark \ref{sec:analcont}.
\end{proof}

\subsection{Inhomogeneous parameters}\label{sec:inhomo}

It is possible to define a time-inhomogeneous versions of the zero range and exclusion processes considered earlier by replacing the time $t$ transition operator (from the state time $t$ to that at time $t+1$) by $\highspinBoson{\alpha_t}{\beta_t}$ (or likewise $\highspinTASEP{\alpha_t}{\beta_t}$). As long as we assume $\beta_t = q^{J_t}\alpha_t$ for $J_t\in \Z_{\geq 1}$, this is stochastic and hence generates a Markov chain. The duality of Corollary \ref{cor:dualhigherJ} clearly extends, and since these operators for different $t$ are still diagonalized in the same basis (owing to Corollary \ref{cor:higherJeig}), we are able to develop analogous results to Theorems \ref{t.moments} and \ref{t.fred}. The only change in Theorem \ref{t.moments} is the replacement
$$
\Big(\frac{1+\beta z_j}{1+\alpha z_j}\Big)^{t} \mapsto \prod_{s=0}^{t-1}\frac{1+\beta_s z_j}{1+\alpha_s z_j},
$$
and the corresponding change in Theorem \ref{t.fred} is the replacement in the function $\fredg$
$$
\left( \frac{(-\beta w;q)_{\infty}}{(-\alpha w;q)_{\infty}}\right)^t \mapsto  \prod_{s=0}^{t-1} \frac{(-\beta_s w;q)_{\infty}}{(-\alpha_s w;q)_{\infty}}.
$$

It is also natural to consider spatial inhomogeneities. In that case, the $\Lmat$-matrices used to construct the zero range and exclusion processes should depend on location or particle number (respectively) in so far as that the parameter $\alpha$ can be replaced by $\alpha_x$. Under this generalization, the eigenfunctions considered in Appendix \ref{sec:Bethe} no longer suffice for diagonalization and no suitable replacements are presently known. An inspect of the duality proofs as well as the proof of the fusion procedure seem to suggest that these results can be modified to apply in this setting. We do not pursue this direction any further here. However, it is worth noting that when $\nu=0$, these processes relate to those considered in \cite{BorodinCorwin2013discrete} wherein duality and moment formulas were proved for both time and space inhomogeneities. The moment formulas were generalizable in such a manner due to the connect between the $\nu=0$ case and the theory of Macdonald processes \cite[Section 6.2]{BorodinCorwin2013discrete}.

\appendix
\section{Bethe ansatz eigenfunctions}\label{sec:Bethe}

We recall results about the Bethe ansatz eigenfunctions, most of which come from \cite{BCPS2014}.

\begin{definition}\label{def:bethans}
Assume that $q,\nu$ are as in the first two cases of Proposition \ref{rem:stochweights}. Recall $\Ynspace{k}$ from Definition \ref{deftransprobJ1} and let $\Wc^{k}$ equal the space of all compactly supported functions from $\Ynspace{k}$ to $\C$. Define {\it left} and {\it right Bethe ansatz eigenfunctions}\footnote{See Propositions \ref{prophighspindiagJ1} and \ref{propqHahndiag} which describe certain operators for which these are eigenfunctions.}
\begin{align}
	\begin{array}{>{\displaystyle}rc>{\displaystyle}l}
	\Psil_{\vec{z}}(\vec{n})&=&
	\sum_{\sigma\in S(k)}\prod_{1\le B<A\le k}
	\frac{z_{\sigma(A)}-qz_{\sigma(B)}}
	{z_{\sigma(A)}-z_{\sigma(B)}}\prod_{j=1}^{k}
	\left(\frac{1-z_{\sigma(j)}}{1-\nu z_{\sigma(j)}}\right)^{-n_j},
	\\
	\Psir_{\vec{z}}(\vec{n})&=&
	(-1)^k(1-q)^{k}q^{\frac{k(k-1)}{2}}\st(\vec{n})
	\sum_{\sigma\in S(k)}\prod_{1\le B<A\le k}
	\frac{z_{\sigma(A)}-q^{-1}z_{\sigma(B)}}
	{z_{\sigma(A)}-z_{\sigma(B)}}\prod_{j=1}^{k}
	\left(\frac{1-z_{\sigma(j)}}{1-\nu z_{\sigma(j)}}\right)^{n_j},
	\end{array}
	\label{eqneigens}
\end{align}
where $\vec{z}=(z_1,\ldots,z_k)\in(\C\setminus\{1,\nu^{-1}\})^{k}$, and $\st(\vec{n})$ is given by
\begin{align}\label{stationary_measure}
	\st(\vec{n})=\prod_{j=1}^{M(\vec{n})}\frac{(\nu;q)_{c_j}}{(q;q)_{c_j}}.
\end{align}
where $c_1,\ldots, c_{M(\vec{n})}$ are the cluster sizes of $\vec{n}$ (i.e. $n_1 = \cdots = n_{c_1}>n_{c_1+1}=\cdots n_{c_1+c_2}>\cdots$).

Let $\Pld$ be the {\it direct transform} which takes a function $f\in\Wc^{k}$ in the spatial variables $\vec{n}$ and produces a function in the spectral variables $\vec{z}$ according to
$$
	(\Pld f)(\vec{z})=\sum_{\vec{n}\in\Ynspace{k}}f(\vec{n})\Psir_{\vec{z}}(\vec{n}).
$$
The function $(\Pld f)(\vec{z})$ is a symmetric Laurent polynomial in $(1-z_j)/(1-\nu z_j)$, $j=1,\ldots,k$. We denote the space of all such Laurent polynomials by $\Cc^{k}$.

Let $\Pli$ be the {\it inverse transform} which maps Laurent polynomials $G\in\Cc^{k}$ to functions in $\Wc^{k}$ according to the following nested contour integration formula:
\begin{equation*}
(\Pli G)(\vec{n})
= \oint_{\ga}\ldots\oint_{\ga}	d\Plm_{(1^{k})}(\vec{z})\prod_{j=1}^{k}\frac{1}{(1-z_j)(1-\nu z_j)}\Psil_{\vec{z}}(\vec{n}) G(\vec{z}).
\end{equation*}
The contour $\ga$ is a circle containing around $1$, not containing $\nu^{-1}$ and such that it contains its image under multiplication by $q$. The term
\begin{align}\label{dmu_large}
d\Plm_{(1^{k})}(\vec{z})= \frac{1}{k!} \frac{(-1)^{\frac{k(k-1)}{2}}\Vand(\vec{z})^2}{\prod_{i\ne j}(z_i-qz_j)}\prod_{j=1}^{k}\frac{dz_j}{2\pi\i},
\end{align} 	
where $\Vand(\vec{z})=\prod_{1\le i<j\le k}(z_i-z_j)$ is the Vandermonde determinant.
\end{definition}

The following result comes from \cite[Theorems 3.4 and 3.9]{BCPS2014}.
\begin{proposition}\label{propspectral}
Assume that $q,\nu$ are as in the first two cases of Proposition \ref{rem:stochweights}. The transforms $\Pld$ and $\Pli$ are mutual inverses in the sense that $\Pli\Pld$ acts as the identity on $\Wc^{k}$, and $\Pld\Pli$ as the identity on $\Cc^{k}$.
\end{proposition}

It is useful to extend the space $\Wc^{k}$ to include non-compactly supported functions which still have nice growth properties. In particular, for $c,C>0$ define $\Wc^{k}_{\exp(c,C)}$ as those functions $f:\Ynspace{k}\to \C$ such that $|f(\vec{n})|< C \exp\{c \sum_{i=1}^k n_i\}$ for all $\vec{n}\in \Ynspace{k}$. 
\begin{proposition}\label{prop:cCexp}
	There exist $c>0$ small enough and $C>0$ large enough such that $\Pli\Pld$ acts as the identity on $\Wc^{k}_{\exp(c,C)}$.
\end{proposition}
\begin{proof}
Let $f_M$ equal $f$ on the support $[-M,M]^k$ and zero outside. For the choice of contours $\ga$ in the definition of $\Pli$, there exists $C_1,C_2$ such that $C_1 < \frac{1-z}{1-\nu z} < C_2$ holds for all $z\in \ga$. This along with the exponential growth bounds on $f$ implies that $(\Pld f_M)(\vec{z})$ is uniformly convergent as the $z_i$ vary along $\ga$. This, along with the fact that $\Pli \Pld f_M = f_M$ implies the desired result.
\end{proof}

\begin{definition}\label{def:Wmax}
	Let $\Wc^{k}_{\max}$ be the space of all functions $f:\Ynspace{k}\to \C$ such that  $\Pli \Pld f=f$.
\end{definition}
The following corollary is a consequence of Proposition \ref{prop:cCexp}.
\begin{corollary}\label{cor:Wmax}
	For $c>0$ small enough and $C>0$ large enough, $\Wc^{k}_{\exp(c,C)}\subset \Wc^{k}_{\max}$.
\end{corollary}

\section{Vertex weights for $J=1,2,3$}\label{sec:J123}

\subsection{$J=1$ vertex weights}

\begin{center}
	\def\epsx{.12}
	\begin{tabular}{c|c|c}
	&$j_2=0$&$j_2=1$\\\hline
	\raisebox{30pt}{$j_1=0$}&
	\begin{tikzpicture}
		[scale=1, thick]
		\node (i1) at (0,-1) {$g$};
		\node (j1) at (-1,0) {$0$};
		\node (i2) at (0,1) {\phantom{1}$g$\phantom{1}};
		\node (j2) at (1,0) {$0$};
		\draw[densely dotted] (j1) -- (j2);
		\foreach \shi in {(0,0), (\epsx,0), (-\epsx,0)}
		{\begin{scope}[shift=\shi]
			\node (shi1) at (0,-1) {\phantom{$g$}};
			\node (shi2) at (0,1) {\phantom{$g$}};
			\draw[->] (shi1) -- (0,0);
			\draw[->] (0,0) -- (shi2);
		\end{scope}}
		\node at (2.5,0) {$\dfrac{1+\al q^{g}}{1+\al}$};
	\end{tikzpicture}
	&
	\begin{tikzpicture}
		[scale=1, thick]
		\node (i1) at (0,-1) {$g$};
		\node (j1) at (-1,0) {$0$};
		\node (i2) at (0,1) {$g-1$};
		\node (j2) at (1,0) {$1$};
		\foreach \shi in {(0,0), (-\epsx,0)}
		{\begin{scope}[shift=\shi]
			\node (shi1) at (0,-1) {\phantom{$g$}};
			\node (shi2) at (0,1) {\phantom{$g$}};
			\draw[->] (shi1) -- (0,0);
			\draw[->] (0,0) -- (shi2);
		\end{scope}}
		\foreach \shi in {(\epsx,0)}
		{\begin{scope}[shift=\shi]
			\node (shi1) at (0,-1) {\phantom{$g$}};
			\node (shi2) at (0,1) {\phantom{$g$}};
			\draw[->] (shi1) -- (0,0);
		\end{scope}}
		\draw[densely dotted] (j1) -- (j2);
		\draw[->] (\epsx,0) -- (j2);
		\node at (2.5,0) {$\dfrac{\al (1-q^{g})}{1+\al}$};
	\end{tikzpicture}
	\\\hline
	\raisebox{30pt}{$j_1=1$}&
	\begin{tikzpicture}
		[scale=1, thick]
		\node (i1) at (0,-1) {$g$};
		\node (j1) at (-1,0) {$1$};
		\node (i2) at (0,1) {$g+1$};
		\node (j2) at (1,0) {$0$};
		\draw[densely dotted] (j1) -- (j2);
		\draw[->] (j1) -- (-3/2*\epsx,0);
		\foreach \shi in {(1/2*\epsx,0), (3/2*\epsx,0), (-1/2*\epsx,0)}
		{\begin{scope}[shift=\shi]
			\node (shi1) at (0,-1) {\phantom{$g$}};
			\node (shi2) at (0,1) {\phantom{$g$}};
			\draw[->] (shi1) -- (0,0);
			\draw[->] (0,0) -- (shi2);
		\end{scope}}
		\foreach \shi in {(-3/2*\epsx,0)}
		{\begin{scope}[shift=\shi]
			\node (shi1) at (0,-1) {\phantom{$g$}};
			\node (shi2) at (0,1) {\phantom{$g$}};
			\draw[->] (0,0) -- (shi2);
		\end{scope}}
		\node at (2.5,0) {$\dfrac{1-\nu q^{g}}{1+\al}$};
	\end{tikzpicture}
	&
	\begin{tikzpicture}
		[scale=1, thick]
		\node (i1) at (0,-1) {$g$};
		\node (j1) at (-1,0) {$1$};
		\node (i2) at (0,1) {\phantom{1}$g$\phantom{1}};
		\node (j2) at (1,0) {$1$};
		\draw[->] (j1) -- (-\epsx,0);
		\draw[->] (\epsx,0) -- (j2);
		\foreach \shi in {(0,0), (\epsx,0), (-\epsx,0)}
		{\begin{scope}[shift=\shi]
			\node (shi1) at (0,-1) {\phantom{$g$}};
			\node (shi2) at (0,1) {\phantom{$g$}};
			\draw[->] (shi1) -- (0,0);
			\draw[->] (0,0) -- (shi2);
		\end{scope}}
		\draw[densely dotted] (j1) -- (j2);
		\node at (2.5,0) {$\dfrac{\al+\nu q^{g}}{1+\al}$};
	\end{tikzpicture}
	\end{tabular}
\end{center}
When $g=0$, the configuration $(g,0,g-1,1)$ has zero weight, as it should be.

\subsection{$J=2$ vertex weights}

\begin{center}
	\def\epsx{.12}
	\def\scll{.8}
	\begin{tabular}{c|c|c|c}
	&$j_2=0$&$j_2=1$&$j_2=2$\\\hline
	\raisebox{30pt}{$j_1=0$}&
		\begin{tikzpicture}
			[scale=\scll, thick]
			\node (i1) at (0,-1) {$g$};
			\node (j1) at (-1,0) {$0$};
			\node (i2) at (0,1) {\phantom{1}$g$\phantom{1}};
			\node (j2) at (1,0) {$0$};
			\draw[densely dotted] (j1) -- (j2);
			\foreach \shi in {(0,0), (\epsx,0), (-\epsx,0)}
			{\begin{scope}[shift=\shi]
				\node (shi1) at (0,-1) {\phantom{$g$}};
				\node (shi2) at (0,1) {\phantom{$g$}};
				\draw[->] (shi1) -- (0,0);
				\draw[->] (0,0) -- (shi2);
			\end{scope}}
			\node at (3,.5)
			{$\frac{(1+\al q^{g})(1+\al q^{g+1})}{(1+\al)(1+\al q)}$};
		\end{tikzpicture}
		&
		\begin{tikzpicture}
		[scale=\scll, thick]
		\node (i1) at (0,-1) {$g$};
		\node (j1) at (-1,0) {$0$};
		\node (i2) at (0,1) {$g-1$};
		\node (j2) at (1,0) {$1$};
		\foreach \shi in {(0,0), (-\epsx,0)}
		{\begin{scope}[shift=\shi]
			\node (shi1) at (0,-1) {\phantom{$g$}};
			\node (shi2) at (0,1) {\phantom{$g$}};
			\draw[->] (shi1) -- (0,0);
			\draw[->] (0,0) -- (shi2);
		\end{scope}}
		\foreach \shi in {(\epsx,0)}
		{\begin{scope}[shift=\shi]
			\node (shi1) at (0,-1) {\phantom{$g$}};
			\node (shi2) at (0,1) {\phantom{$g$}};
			\draw[->] (shi1) -- (0,0);
		\end{scope}}
		\draw[densely dotted] (j1) -- (j2);
		\draw[->] (\epsx,0) -- (j2);
		\node at (3,.5)
		{$\frac{\al(1+q)(1-q^{g})(1+\al q^{g})}{(1+\al)(1+\al q)}$};
	\end{tikzpicture}
		&
		\begin{tikzpicture}
		[scale=\scll, thick]
		\node (i1) at (0,-1) {$g$};
		\node (j1) at (-1,0) {$0$};
		\node (i2) at (0,1) {$g-2$};
		\node (j2) at (1,0) {$2$};
		\node (j21) at (1,-\epsx/2) {\phantom{$2$}};
		\node (j22) at (1,\epsx/2) {\phantom{$2$}};
		\foreach \shi in {(-\epsx,0)}
		{\begin{scope}[shift=\shi]
			\node (shi1) at (0,-1) {\phantom{$g$}};
			\node (shi2) at (0,1) {\phantom{$g$}};
			\draw[->] (shi1) -- (0,0);
			\draw[->] (0,0) -- (shi2);
		\end{scope}}
		\foreach \shi in {(0,0), (\epsx,0)}
		{\begin{scope}[shift=\shi]
			\node (shi1) at (0,-1) {\phantom{$g$}};
			\node (shi2) at (0,1) {\phantom{$g$}};
			\draw[->] (shi1) -- (0,0);
		\end{scope}}
		\draw[densely dotted] (j1) -- (j2);
		\draw[->] (\epsx,-\epsx/2) -- (j21);
		\draw[->] (\epsx,\epsx/2) -- (j22);
		\node at (3,.5)
		{$\frac{\al^{2}(1-q^{g})(q-q^{g})}{(1+\al)(1+\al q)}$};
	\end{tikzpicture}
		\\\hline
		\raisebox{30pt}{$j_1=1$}&
		\begin{tikzpicture}
		[scale=\scll, thick]
		\node (i1) at (0,-1) {$g$};
		\node (j1) at (-1,0) {$1$};
		\node (i2) at (0,1) {$g+1$};
		\node (j2) at (1,0) {$0$};
		\draw[densely dotted] (j1) -- (j2);
		\draw[->] (j1) -- (-3/2*\epsx,0);
		\foreach \shi in {(1/2*\epsx,0), (3/2*\epsx,0), (-1/2*\epsx,0)}
		{\begin{scope}[shift=\shi]
			\node (shi1) at (0,-1) {\phantom{$g$}};
			\node (shi2) at (0,1) {\phantom{$g$}};
			\draw[->] (shi1) -- (0,0);
			\draw[->] (0,0) -- (shi2);
		\end{scope}}
		\foreach \shi in {(-3/2*\epsx,0)}
		{\begin{scope}[shift=\shi]
			\node (shi1) at (0,-1) {\phantom{$g$}};
			\node (shi2) at (0,1) {\phantom{$g$}};
			\draw[->] (0,0) -- (shi2);
		\end{scope}}
		\node at (3,.5)
		{$\frac{(1-\nu q^{g})(1+\al q^{g+1})}{(1+\al)(1+\al q)}$};
	\end{tikzpicture}
		&
		\begin{tikzpicture}
		[scale=\scll, thick]
		\node (i1) at (0,-1) {$g$};
		\node (j1) at (-1,0) {$1$};
		\node (i2) at (0,1) {\phantom{1}$g$\phantom{1}};
		\node (j2) at (1,0) {$1$};
		\draw[->] (j1) -- (-\epsx,0);
		\draw[->] (\epsx,0) -- (j2);
		\foreach \shi in {(0,0), (\epsx,0), (-\epsx,0)}
		{\begin{scope}[shift=\shi]
			\node (shi1) at (0,-1) {\phantom{$g$}};
			\node (shi2) at (0,1) {\phantom{$g$}};
			\draw[->] (shi1) -- (0,0);
			\draw[->] (0,0) -- (shi2);
		\end{scope}}
		\draw[densely dotted] (j1) -- (j2);
		\node at (3.1,0.6)
		{\parbox{.19\textwidth}
		{$1-\frac{(1-\nu q^{g})(1+\al q^{g+1})}{(1+\al)(1+\al q)}$
		
		\quad$-\frac{\al(1-q^{g})(\al q+\nu q^{g})}{(1+\al)(1+\al q)}$}};
	\end{tikzpicture}
		&
	\begin{tikzpicture}
		[scale=\scll, thick]
		\node (i1) at (0,-1) {$g$};
		\node (j1) at (-1,0) {$1$};
		\node (i2) at (0,1) {$g-1$};
		\node (j2) at (1,0) {$2$};
		\node (j21) at (1,-\epsx/2) {\phantom{$2$}};
		\node (j22) at (1,\epsx/2) {\phantom{$2$}};
		\foreach \shi in {(-\epsx,0),(0,0)}
		{\begin{scope}[shift=\shi]
			\node (shi1) at (0,-1) {\phantom{$g$}};
			\node (shi2) at (0,1) {\phantom{$g$}};
			\draw[->] (shi1) -- (0,0);
			\draw[->] (0,0) -- (shi2);
		\end{scope}}
		\foreach \shi in {(0,0), (\epsx,0)}
		{\begin{scope}[shift=\shi]
			\node (shi1) at (0,-1) {\phantom{$g$}};
			\node (shi2) at (0,1) {\phantom{$g$}};
			\draw[->] (shi1) -- (0,0);
		\end{scope}}
		\draw[densely dotted] (j1) -- (j2);
		\draw[->] (j1) -- (-\epsx,0);
		\draw[->] (\epsx,-\epsx/2) -- (j21);
		\draw[->] (\epsx,\epsx/2) -- (j22);
		\node at (3,.5)
		{$\frac{\al(1-q^{g})(\al q+\nu q^{g})}{(1+\al)(1+\al q)}$};
	\end{tikzpicture}

		\\\hline
		\raisebox{30pt}{$j_1=2$}&
	
	\begin{tikzpicture}
		[scale=\scll, thick]
		\node (i1) at (0,-1) {$g$};
		\node (j1) at (-1,0) {$2$};
		\node (i2) at (0,1) {$g+2$};
		\node (j2) at (1,0) {$0$};
		\node (j11) at (-1,-\epsx/2) {\phantom{$2$}};
		\node (j12) at (-1,\epsx/2) {\phantom{$2$}};
		\foreach \shi in {(0,0), (\epsx,0), (2*\epsx,0),
		(-\epsx,0), (-2*\epsx,0)}
		{\begin{scope}[shift=\shi]
			\node (shi1) at (0,-1) {\phantom{$g$}};
			\node (shi2) at (0,1) {\phantom{$g$}};
			\draw[->] (0,0) -- (shi2);
		\end{scope}}
		\foreach \shi in {(0,0),(\epsx,0),(-\epsx,0)}
		{\begin{scope}[shift=\shi]
			\node (shi1) at (0,-1) {\phantom{$g$}};
			\node (shi2) at (0,1) {\phantom{$g$}};
			\draw[->] (shi1) -- (0,0);
		\end{scope}}
		\draw[densely dotted] (j1) -- (j2);
		\draw[->] (j11) -- (-2*\epsx,-\epsx/2);
		\draw[->] (j12) -- (-2*\epsx,\epsx/2);
		\node at (3,.5)
		{$\frac{(1-\nu q^{g})(1-\nu q^{g+1})}{(1+\al)(1+\al q)}$};
	\end{tikzpicture}

		&
	\begin{tikzpicture}
		[scale=\scll, thick]
		\node (i1) at (0,-1) {$g$};
		\node (j1) at (-1,0) {$2$};
		\node (i2) at (0,1) {$g+1$};
		\node (j2) at (1,0) {$1$};
		\node (j11) at (-1,-\epsx/2) {\phantom{$2$}};
		\node (j12) at (-1,\epsx/2) {\phantom{$2$}};
		\foreach \shi in {(0,0), (\epsx,0), (2*\epsx,0),
		(-\epsx,0)}
		{\begin{scope}[shift=\shi]
			\node (shi1) at (0,-1) {\phantom{$g$}};
			\node (shi2) at (0,1) {\phantom{$g$}};
			\draw[->] (0,0) -- (shi2);
		\end{scope}}
		\foreach \shi in {(0,0),(\epsx,0),(-\epsx,0)}
		{\begin{scope}[shift=\shi]
			\node (shi1) at (0,-1) {\phantom{$g$}};
			\node (shi2) at (0,1) {\phantom{$g$}};
			\draw[->] (shi1) -- (0,0);
		\end{scope}}
		\draw[densely dotted] (j1) -- (j2);
		\draw[->] (2*\epsx,0) -- (j2);
		\draw[->] (j11) -- (-1*\epsx,-\epsx/2);
		\draw[->] (j12) -- (-1*\epsx,\epsx/2);
		\node at (3,.5)
		{$\frac{(1+q)(1-\nu q^{g})(\al+\nu q^{g})}{(1+\al)(1+\al q)}$};
	\end{tikzpicture}
		&
	
	\begin{tikzpicture}
		[scale=\scll, thick]
		\node (i1) at (0,-1) {$g$};
		\node (j1) at (-1,0) {$2$};
		\node (i2) at (0,1) {$\phantom{1}g\phantom{1}$};
		\node (j2) at (1,0) {$2$};
		\node (j11) at (-1,-\epsx/2) {\phantom{$2$}};
		\node (j12) at (-1,\epsx/2) {\phantom{$2$}};
		\node (j21) at (1,-\epsx/2) {\phantom{$2$}};
		\node (j22) at (1,\epsx/2) {\phantom{$2$}};
		\foreach \shi in {(0,0), (\epsx,0),
		(-\epsx,0)}
		{\begin{scope}[shift=\shi]
			\node (shi1) at (0,-1) {\phantom{$g$}};
			\node (shi2) at (0,1) {\phantom{$g$}};
			\draw[->] (0,0) -- (shi2);
		\end{scope}}
		\foreach \shi in {(0,0),(\epsx,0),(-\epsx,0)}
		{\begin{scope}[shift=\shi]
			\node (shi1) at (0,-1) {\phantom{$g$}};
			\node (shi2) at (0,1) {\phantom{$g$}};
			\draw[->] (shi1) -- (0,0);
		\end{scope}}
		\draw[densely dotted] (j1) -- (j2);
		\draw[->] (\epsx,-\epsx/2) -- (j21);
		\draw[->] (\epsx,\epsx/2) -- (j22);
		\draw[->] (j11) -- (-1*\epsx,-\epsx/2);
		\draw[->] (j12) -- (-1*\epsx,\epsx/2);
		\node at (3,.5)
		{$\frac{(\al+\nu q^{g})(\al q+\nu q^{g})}{(1+\al)(1+\al q)}$};
	\end{tikzpicture}

	\end{tabular}
\end{center}

Again, note the automatic vanishing of suitable probabilities
triggered by factors $1-q^{g}$ and $q-q^{g}$.

\subsection{$J=3$ vertex weights}

For $J=3$, there are 16 vertex types, so we will no longer draw
the arrow configurations. We will also
omit the common denominator $(1+\al)(1+q\al)(1+q^{2}\al)$
which is present in all probabilities.
The table of $J=3$ vertex weights is the following:
\begin{center}
	\begin{tabular}{c|c|c|c|c}
		&$j_2=0$&$j_2=1$&$j_2=2$&$j_2=3$
		\\
		\hline
		$j_1=0$&
		\parbox{.21\textwidth}
		{
		$(1+\al q^{g})(1+\al q^{g+1})\\\times(1+\al q^{g+2})$
		}
		&\parbox{.21\textwidth}
		{
		$\al(1+q+q^{2})(1+\al q^{g})\\\times(1+\al q^{g+1})(1-q^{g})$
		}
		&
		\parbox{.21\textwidth}
		{$\al^{2}(1+q+q^{2})\\\times(1 + \al q^g)\\\times(1 - q^g)(q - q^g)$}
		&
		\parbox{.21\textwidth}
		{$\al^{3}(1 - q^g)\\\times(q - q^g)(q^{2}-q^{g})$}
		\\
		\hline
		$j_1=1$&\parbox{.21\textwidth}
		{
		$(1-\nu q^{g})(1+\al q^{g+1})\\\times(1+\al q^{g+2})$
		}
		&\parbox{.21\textwidth}
		{
		$(1+\al q^{g+1})
		\\\times\bigg(
		\al(1+q+q^{2})
		\\{}\ -\al q^{g+1}(1+q-\al q)
		\\{}\ +\nu q^g(1-\al-\al q)
		\\{}\ +\al\nu q^{2g}(1+q+q^{2})
		\bigg)$
		}
		&\parbox{.21\textwidth}
		{
		$\alpha ^2 q \left(1+q+q^2\right)
		\\+
		\alpha q^{g} (\alpha  q-q-1)
		\\{}\ \times\left(-\nu +\alpha  q^2+\alpha  q\right)
		\\
		+\alpha q^{2g}
		(\alpha +\alpha  q-1)
		\\{}\ \times
		\left(\nu -\alpha  q^2+\nu  q\right)
		\\
		+\alpha ^2 \nu q^{3g} \left(1+q+q^{2}\right)$
		}
		&\parbox{.21\textwidth}
		{
		$\al^{2}(1-q^{g})(q-q^{g})\\\times
		(\al q^{2}+\nu q^{g})$
		}
		\\
		\hline
		$j_1=2$
		&\parbox{.21\textwidth}
		{
		$(1-\nu q^{g})(1-\nu q^{g+1})
		\\\times(1+\al q^{g+2})$
		}
		&\parbox{.21\textwidth}
		{
		$(1-\nu q^{g})\\\times
		\bigg(
		\al(1+q+q^{2})
		\\{}\ +
		\al q^{g+2}(\al+q\al-1)
		\\{}\ +
		\nu q^{g}(1+q-q\al)
		\\{}\ +
		\al\nu q^{2g+1}\\{}\ \ \ \ \times(1+q+q^{2})
		\bigg)$
		}
		&\parbox{.21\textwidth}
		{
		$(\al q+\nu q^{g})
		\\\times
		\bigg(
		\al(1+q+q^{2})
		\\{}\ -
		\al q^{g+1}(1+q-q\al)
		\\{}\ +
		\nu q^{g}(1-\al-q\al)
		\\{}\ +
		\al \nu q^{2g}(1+q+q^{2})
		\bigg)$
		}
		&\parbox{.21\textwidth}
		{
		$\al(1-q^{g})(\al q+\nu q^{g})\\\times(\al q^{2}+\nu q^{g})$
		}
		\\
		\hline
		$j_1=3$
		&\parbox{.21\textwidth}
		{
		$(1-\nu q^{g})(1-\nu q^{g+1})\\\times(1-\nu q^{g+2})$
		}
		&\parbox{.21\textwidth}
		{
		$(1+q+q^{2})(\al+\nu q^{g})\\\times
		(1-\nu q^{g})(1-\nu q^{g+1})$
		}
		&\parbox{.21\textwidth}
		{
		$(1+q+q^{2})(\al+\nu q^{g})\\\times
		(\al q+\nu q^{g})(1-\nu q^{g})$
		}
		&\parbox{.21\textwidth}
		{
		$(\al+\nu q^{g})
		(\al q+\nu q^{g})\\\times
		(\al q^{2}+\nu q^{g})$
		}
	\end{tabular}
\end{center}

\section{Yang-Baxter equation}\label{sec:YBE}

Denote, for this section,
$\bernw{m,n}{q}{\nu}{\alpha_1,\alpha_2}(k_1,k_2;k_1',k_2') = \bernw{1}{q}{\nu}{\alpha_1}(m,k_1; \ell ,k_1')\bernw{1}{q}{\nu}{\alpha_2}(\ell,k_2; n ,k_2')$
where $\ell = m+k_1-k_1'=n+k_2'-k_2$. Let $\bernwtilde{m,n}{q}{\nu}{\alpha_1,\alpha_2}(k_1,k_2;k_1',k_2')=\bernw{m,n}{q}{\nu}{\alpha_1,\alpha_2}(k_1',k_2';k_1,k_2)$ and define the matrix
$$
Y=
\left(
  \begin{array}{cccc}
    1 & 0 & 0 & 0 \\
    0 & -\frac{q(u_1-u_2)}{s u_1 u_2(u_1-q u_2)} & -\frac{1-q}{s u_1 (u_1-q u_2)} & 0 \\
    0 & -\frac{1-q}{s u_2 (u_1-q u_2)} & -\frac{u_1-u_2}{s u_1 u_2(u_1-q u_2)} & 0 \\
    0 & 0 & 0 & \frac{1}{s^2 u_1^2 u_2^2} \\
  \end{array}
\right)
$$
under the association $\alpha_i = -su_i$ and $\nu = s^2$. Then the Yang-Baxter equation amounts to the fact that $\bernw{m,n}{q}{\nu}{\alpha_2,\alpha_1}$ and $\bernwtilde{m,n}{q}{\nu}{\alpha_1,\alpha_2}$ (note the interchange of indices in $\alpha_1,\alpha_2$) are similar with respect to $Y$:
$$
\bernw{m,n}{q}{\nu}{\alpha_2,\alpha_1}Y = Y \bernwtilde{m,n}{q}{\nu}{\alpha_1,\alpha_2}.
$$
This can be derived, in light of Remark \ref{rem:BorR}, from \cite[Proposition 2.5]{Borodin2014vertex}, which itself is just a restatement of the standard six-vertex Yang-Baxter equation.
%


\end{document}